\documentclass{amsart}

\usepackage{amsmath,amssymb,amsfonts}
\usepackage{esint, cancel}
\usepackage{amsthm}
\usepackage{mathrsfs}
\usepackage{geometry}
\usepackage{verbatim}
\usepackage[mathcal]{euscript}
\usepackage{graphicx}
\usepackage{color}
\usepackage{cite}

\usepackage[colorlinks=true,urlcolor=blue, citecolor=red,linkcolor=blue,
linktocpage,pdfpagelabels, bookmarksnumbered,bookmarksopen]{hyperref}
\usepackage[hyperpageref]{backref}

\newtheorem{theorem}{Theorem}[section]
\newtheorem{lemma}[theorem]{Lemma}
\newtheorem{proposition}[theorem]{Proposition}
\newtheorem{corollary}[theorem]{Corollary}

\theoremstyle{definition}
\newtheorem{remark}[theorem]{Remark}

\newtheorem{definition}[theorem]{Definition}

\newtheorem*{ack}{Acknowledgments}
\newtheorem*{disclosure}{AI disclosure statement}

\numberwithin{equation}{section}

\author[Bianchi]{Francesca Bianchi}
\address[F.\ Bianchi]{Liceo Scientifico Statale  ``Vito Volterra'' 
	\newline\indent
	Via Rinalda Pavoni 14, 60044 Fabriano (AN), Italy
	}
\email{f.bianchi@liceoscientificofabriano.edu.it}

\date{\today}

\author[Bozzola]{Francesco Bozzola}
\address[F.\ Bozzola]{Grupo de F\'isica Matem\'atica, Instituto Superior T\'ecnico
\newline\indent Universidade de Lisboa
\newline\indent Av. Rovisco Pais, 1049-001 Lisboa, Portugal
}
\email{francesco.bozzola@tecnico.ulisboa.pt}

\author[Brasco]{Lorenzo Brasco}
\address[L.\ Brasco]{Dipartimento di Matematica e Informatica
	\newline\indent
	Universit\`a degli Studi di Ferrara
	\newline\indent
	Via Machiavelli 35, 44121 Ferrara, Italy}
\email{lorenzo.brasco@unife.it}


\title[clamped plate eigenvalue problem in low dimensions]{A geometric lower bound\\ for the clamped plate eigenvalue problem\\ in low dimensions and extensions}

\begin{document}
	
	\begin{abstract}
We show that the spectrum of the Dirichlet-biLaplacian can be quantitatively bounded from below in terms of the inradius only, for a general open set in dimensions $2$ and $3$. For the same dimensions, we also extend this result to fractional powers $s$ of the Laplacian, with $N/2<s<2$. The resulting estimate comes with a constant having the correct asymptotic behavior as $s$ approaches the two endpoints $N/2$ and $2$. The proofs exploit capacitary techniques in the vein of Maz'ya. 
	\end{abstract}
	
	\maketitle
	\begin{center}
		\begin{minipage}{9cm}
			\small
			\tableofcontents
		\end{minipage}
	\end{center}

	\section{Introduction}
	
	\subsection{Scope of the paper}
For an open set $\Omega\subseteq\mathbb{R}^N$, we want to consider the following quantity
	\begin{equation}\label{lambdone}
			\Lambda(\Omega):=\inf_{\varphi\in C^\infty_0(\Omega)}\left\{\|\Delta \varphi\|^2_{L^2(\Omega)}\, :\, \|\varphi\|_{L^2(\Omega)}=1\right\}.
		\end{equation}
		The symbol $\Delta$ stands for the usual Laplacian operator.	By definition, it coincides with the sharp constant in the following higher order Poincar\'e inequality
		\[
		C\,\int_\Omega |\varphi|^2\,dx\le \int_\Omega |\Delta \varphi|^2\,dx,\qquad \text{for every}\ \varphi\in C^\infty_0(\Omega).
		\]
The constant $\Lambda(\Omega)$ also coincides with the bottom of the spectrum of the Dirichlet-biLaplacian on $\Omega$, i.e. the fourth order differential operator
\[
\varphi\mapsto (-\Delta)^2 \varphi=(-\Delta)\circ (-\Delta) \varphi,
\]
coupled with the homogeneous boundary conditions
\begin{equation}
\label{bc}
\varphi=\frac{\partial \varphi}{\partial\nu}=0,\qquad \text{on}\ \partial\Omega.
\end{equation}
As usual, we denote by $\nu$ the unit outer normal vector.
\par
Such an operator is better understood from a variational point of view, at least for our purposes. To this aim, let us take $W^{2,2}_0(\Omega)$, i.e. the closure of $C^\infty_0(\Omega)$ in the second order Sobolev space $W^{2,2}(\Omega)$. Then the Dirichlet-biLaplacian of $\varphi\in W^{2,2}_0(\Omega)$ is defined in a weak sense by
\[
\langle (-\Delta)^2 \varphi, \psi\rangle=\int_\Omega \Delta \varphi\,\Delta \psi\,dx,\qquad \text{for every}\ \psi\in C^\infty_0(\Omega).
\]
The homogeneous conditions \eqref{bc}  are encoded (again in a weak sense) in the membership of $\varphi$ to the space $W^{2,2}_0(\Omega)$.
\par
 When the infimum in \eqref{lambdone} is attained on the closure space $W^{2,2}_0(\Omega)$, we have that $\Lambda(\Omega)$ coincides with the smallest eigenvalue of the Dirichlet-biLaplacian on $\Omega$, i.e. it is the smallest real number $\Lambda$ such that the following boundary value problem	
\[
(-\Delta)^2 u=\Lambda\,u,\quad \text{in}\ \Omega,\qquad \text{with}\, u=\frac{\partial u}{\partial\nu}=0,\quad \text{on}\ \partial\Omega,
\]	
admits a non-trivial weak solution $u\in W^{2,2}_0(\Omega)$. We refer to \cite[Chapter 11]{Hen} for more details on this problem, which is usually called {\it clamped plate eigenvalue problem}.
\par
The main scope of this paper is to investigate the validity of the following geometric lower bound
\begin{equation}
\label{base}
\frac{1}{C}\,\left(\frac{1}{r_\Omega}\right)^4\le \Lambda(\Omega),
\end{equation}
under minimal assumptions on the open set $\Omega$. Here $r_\Omega$ is the {\it inradius} of the open set $\Omega$, defined by
\[
r_\Omega=\Big\{r>0\, :\, \exists\ \text{an open ball of radius $r$ contained in}\ \Omega \Big\}.
\]
The power $-4$ on the inradius in \eqref{base} reflects the dimensionality of the constant $\Lambda$, i.e. the latter obeys the following scaling law
\[
\Lambda(t\,\Omega) = \frac{\Lambda(\Omega)}{t^4}.
\]
\subsection{One step back}
In order to give some context, we briefly recall that if we replace \eqref{lambdone} with the first order sharp Poincar\'e constant 
	\begin{equation}
	\label{lambdino}
	\lambda(\Omega)=\inf_{\varphi\in C^\infty_0(\Omega)}\left\{\|\nabla \varphi\|^2_{L^2(\Omega)}\, :\, \|\varphi\|_{L^2(\Omega)}=1\right\},
	\end{equation}
this same problem is well-studied. In the plane, i.e. for $N=2$, it is known that
\begin{equation}
\label{base2}
\frac{1}{C}\,\left(\frac{1}{r_\Omega}\right)^2\le \lambda(\Omega),
\end{equation}
under very mild topological assumptions. More precisely, such a lower bounds holds:
\begin{itemize}
\item for {\it simply connected sets}. This result sometimes goes under the name of  {\it Makai-Hayman inequality}, see \cite{Ha} and \cite{Ma} for the original result. We also point out \cite{An, BC, vanBu} for some refinements or alternative proofs of this result;
\vskip.2cm
\item more generally, for {\it open sets with a prescribed number of ``holes''}, see \cite{Cr, Os, Tay}. We also refer to \cite[Theorem 1.2]{BozBra} for an extension of this result to more general first order Poincar\'e constant. In this case, the constant $C$ in \eqref{lambdino} must depend on the number of ``holes'' and eventually blows-up, as the latter diverges. 
\end{itemize}  
On the contrary, for dimensions $N\ge 3$ the lower bound \eqref{base2} can not hold, unless some stronger geometric assumptions are imposed on the open sets. For example, the lower bound \eqref{base2} holds for  convex sets (see \cite[Th\'eor\`eme 8.1]{He} and \cite[Theorem 2.1]{Ka}) or for open sets satisfying a suitable ``fatness condition'' of the complement $\mathbb{R}^N\setminus\Omega$ (see \cite[Corollary 5.6]{BozBra2}). We refer to the introduction of \cite{BozBra} for further comments on this problem.
\vskip.2cm\noindent
We also recall that the aforementioned results in dimension $N=2$ have been recently extended to the {\it fractional case} by the first and third authors. More precisely, let us introduce the following fractional sharp Poincar\'e constant 
\[
	\lambda^s(\Omega)=\inf_{\varphi\in C^\infty_0(\Omega)}\left\{[\varphi]^2_{W^{s,2}(\mathbb{R}^N)}\, :\, \|\varphi\|_{L^2(\Omega)}=1\right\},\qquad \text{for}\ 0<s<1,
	\]
	where\footnote{Functions in $C^\infty_0(\Omega)$ are always considered extended by zero to the complement $\mathbb{R}^N\setminus\Omega$.} 
\[
[\varphi]^2_{W^{s,2}(\mathbb{R}^N)}=\iint_{\mathbb{R}^N\times\mathbb{R}^N} \frac{|\varphi(x)-\varphi(y)|^2}{|x-y|^{N+2\,s}}\,dx\,dy,\qquad \text{for every}\ \varphi\in C^\infty_0(\Omega),
\]	
is the {\it Gagliardo-Slobodecki\u{\i} seminorm}. Then in \cite{BiaBra22} and \cite{BiaBra24} it has been proved that for $1/2<s<1$ it holds
\[
\frac{1}{C_s}\,\left(\frac{1}{r_\Omega}\right)^{2\,s}\le \lambda^s(\Omega),
\]
for all $\Omega\subseteq\mathbb{R}^2$ open set satisfying the same topological assumptions as in the case of \eqref{base2}. The range of $1/2<s<1$ it optimal, since for $s\le 1/2$ the lower bound can not hold (see \cite[Theorem 1.3]{BiaBra22}). Moreover, the constant $C_s$ is such that\footnote{Throughout the whole paper, the writing ``$f\sim g$ as $x\to x_0$'', means that
	\[
	0<\liminf_{x\to x_0} \frac{f(x)}{g(x)}\le \limsup_{x\to x_0} \frac{f(x)}{g(x)}<+\infty.
	\]
} 
\[
C_s\sim \frac{1}{s-\dfrac{1}{2}},\quad \text{as}\ s\searrow \frac{1}{2}\qquad\text{and}\qquad C_s\sim 1-s,\quad \text{as}\ s\nearrow 1,
\]
and this asymptotic behavior is optimal (see \cite[Theorem 1.2]{BiaBra24}). 
\subsection{Scope of the paper (reprise)}
In light of this discussion, it is quite natural to extend our initial goal of proving \eqref{base}. Namely, we will consider the same issue for the whole family of Poincar\'e--type constants interpolating between $\lambda(\Omega)$ and $\Lambda(\Omega)$, so to admit intermediate higher orders of fractional differentiability. In other words, if for $1 < s < 2$ we define the following quantities
	\begin{equation}\label{lambda_BR}
		\Lambda^s(\Omega):=\inf_{\varphi\in C^\infty_0(\Omega)}\left\{[\nabla \varphi]^2_{W^{s-1,2}(\mathbb{R}^N)} : \|\varphi\|_{L^2(\Omega)}=1\right\},
	\end{equation}	
we want to discuss whether it is possible to bound from below this constant in terms of $r_\Omega$ only or not. Thus, by observing that
\begin{equation}
\label{scaling_law}
\Lambda^s(t\,\Omega) = \frac{\Lambda^s(\Omega)}{t^{2\,s}},\qquad \text{for every}\ t>0,
		\end{equation}
we inquire about the validity of the inequality
\[
\frac{1}{C}\,\left(\frac{1}{r_\Omega}\right)^{2\,s}\le \Lambda(\Omega),
\]
for some constant $C=C(N,s)>0$, not depending on $\Omega$. Again, we would like to admit the largest possible class of open sets.
\begin{remark}[Echoes of real interpolation]
The interpolative nature of $\Lambda^s$ is better appreciated by observing that for every $\varphi\in C^\infty_0(\Omega)$ we have
\[
[\nabla\varphi]^2_{W^{s-1,2}(\mathbb{R}^N)}\sim \frac{1}{2-s}\,\|\Delta \varphi\|^2_{L^2(\Omega)},\qquad \text{as}\ s\nearrow 2,
\]
and
\[
[\nabla\varphi]^2_{W^{s-1,2}(\mathbb{R}^N)}\sim \frac{1}{s-1}\,\|\nabla \varphi\|^2_{L^2(\Omega)},\qquad \text{as}\ s\searrow 1,
\]
see Proposition \ref{prop:BBM} below. These are plain consequences of the so-called {\it Bourgain-Brezis-Mironescu} and {\it Maz'ya-Shaposhnikova formulas}, respectively. 
\end{remark}
\begin{remark}
Even if we will not use this fact, we point out that both the constants $\lambda^s(\Omega)$ and $\Lambda^s(\Omega)$ are connected with suitable eigenvalue problems, for fractional powers of the Laplacian $-\Delta$. Indeed, when such infima are attained on the relevant Sobolev spaces obtained as closures of $C^\infty_0(\Omega)$, they coincide with the first eigenvalue of 
\[
(-\Delta )^s u=\lambda\,u,\qquad \text{in}\ \Omega,
\]
with homogeneous nonlocal Dirichlet conditions.
We refer the reader to the survey paper \cite{Sa} for a gentle introduction to the properties of the operator $(-\Delta)^s$, when $1<s<2$. The case $0<s<1$ is nowadays much more studied, we refer for example to \cite[Chapter 2]{BuVa} for a friendly introduction.
\end{remark}
	\subsection{Main results}
	After this premise, we can present the main results of our paper.
	We will show that, for dimension $N=2$ and $N=3$, it is possible to prove the desired geometric lower bound for both $\Lambda$ and $\Lambda^s$. This holds for {\it every open set} having finite inradius.
In the fractional case, we will need the restriction
\[
\frac{N}{2}<s<2.
\]
This range is optimal, i.e. for $s\le N/2$ the lower bound can not hold in general, without imposing further topological/geometric assumptions on the sets (see Section \ref{sec:8}). The core of this result is the fact that points have positive capacity in this regime. This means that the following quantities
\[
\mathrm{Cap}_2(\{0\};Q_1):=\inf_{\varphi \in C^\infty_0(Q_1)}\left\{\|\Delta\varphi\|^2_{L^2(\mathbb{R}^N)} : \varphi\equiv 1\ \text{in a neighborhood of}\ \{0\} \right\},	
\]
and (for $1<s<2$)
\[
\mathrm{Cap}_s(\{0\};Q_1):=\inf_{\varphi \in C^\infty_0(Q_1)}\left\{[\nabla\varphi ]^2_{W^{s-1, 2}(\mathbb{R}^N)} : \varphi\equiv 1\ \text{in a neighborhood of}\ \{0\} \right\},	
\]
{\it are both positive}, for $N\in\{2,3\}$ and $N/2<s< 2$. Here $Q_1$ is the $N-$dimensional open cube $(-1,1)^N$.
\vskip.2cm\noindent
For ease of presentation, we will give two separate statements for the clamped case and for the fractional case $1<s<2$. 
		\begin{theorem}\label{thm:local}
			Let $N\in\{2, 3\}$. For every open set $\Omega \subseteq \mathbb{R}^N$ with $r_\Omega<+\infty$, we have 
			\begin{equation} \label{rossino}
				\frac{1}{\mathcal{C}_1}\,\left(\frac{1}{r_\Omega}\right)^4 \le \Lambda(\Omega),
			\end{equation}
			for a constant $\mathcal{C}_1 = \mathcal{C}_1(N) >0$.
		\end{theorem}
In the fractional case, it is useful to better precise the nature of the constant entering in the lower bound. Our result reads as follows.
		\begin{theorem} \label{thm:frac}
			Let $N\in\{2, 3\}$ and $N/2 < s < 2$. 
			For every open set $\Omega \subseteq \mathbb{R}^N$ with $r_\Omega<+\infty$, we have 
			\begin{equation}
			\label{higher}
				\frac{\mathrm{Cap}_s(\{0\}; Q_1)}{\mathcal{C}_2}\,\left(\frac{1}{r_\Omega}\right)^{2\,s} \le \Lambda^s(\Omega),
							\end{equation}
			for a constant $\mathcal{C}_2 = \mathcal{C}_2(N,s) > 0$ having the following asymptotic behavior
		\[
				\mathcal{C}_2 \sim 1, \qquad \text{as}\ s \searrow N/2\ \text{and}\ s \nearrow 2.
		\]
		\end{theorem}
We will show in Section \ref{sec:8} that the dependence on $s$ of the left-hand side of \eqref{higher} is sharp, in the asymptotic regimes $s\searrow N/2$ and $s\nearrow 2$. This will be an important feature of our estimate.
	
	\subsection{Some comments on the proofs}
The idea of the proof is the same, for both Theorem \ref{thm:local} and Theorem \ref{thm:frac}. 
As already said, the key fact is that points have positive capacity. The latter, together with the finiteness of the inradius, imply that our open sets have a ``uniformly fat complement''. We will be more precise on this point in a moment. 
\par
At first, we wish to explain our choice of presenting the two results separately. There are two good reasons, in our opinion. Firstly, a reader interested in the classical clamped plate eigenvalue problem could not be interested in all the technicalities and subtle points that go with the fractional case. Secondly, the overall strategy of the proof is better appreciated by starting with the case of $\Lambda$, which is quite simple. In this way, the reader interested in the fractional case, should better appreciate the kind of novelties and complications introduced by this generalization.
\vskip.2cm\noindent
We now proceed to explain the proof in the local case, i.e. for Theorem \ref{thm:local}. 
By following \cite[Theorem 15.4.2.1]{Maz}, we tile the space with cubes having an edge length $L$. By calling $\{Q_\mathbf{i}\}_{\mathbf{i}\in \mathbb{Z}^N}$ these cubes, for each function $\varphi\in C^\infty_0(\Omega)$ we can obviously write
\[
\|\varphi\|_{L^2(\Omega)}^2=\sum_{\mathbf{i}\in\mathbb{Z}^N} \|\varphi\|^2_{L^2(Q_\mathbf{i})}.
\]
If we choose the edge length $L>r_\Omega$, by definition of inradius we know that in each $Q_{\mathbf{i}}$ there is a point $x_\mathbf{i}\in \mathbb{R}^N\setminus \Omega$. Consequently, our function vanishes in a neighborhood of such a point. We can then exploit a suitable {\it Maz'ya-Poincar\'e inequality} on each cube (see Lemma \ref{lm:lemmasega}), thanks to the fact that in dimensions $N\in\{2,3\}$ points have positive capacity. For second order derivatives, the relevant inequality takes the following form 
\begin{equation}
\label{ammazzapoincare}
\frac{\mathrm{Cap}_2(\{x_\mathbf{i}\};Q_\mathbf{i})}{L^N}\,\int_{Q_\mathbf{i}}|\varphi|^2\,dx\lesssim \int_{Q_\mathbf{i}} |\Delta \varphi|^2\,dx+\frac{1}{L^2}\,\int_{{Q_\mathbf{i}}} |\nabla \varphi|^2\,dx.
\end{equation}
We observe that the ``undesired'' term containing $\nabla \varphi$ must be there at this stage, i.e. for smooth functions defined on a cube and vanishing only in a neighborhood of a point, one {\it can not} have the previous inequality \eqref{ammazzapoincare} without such a term on the right-hand side.
\par
By summing up all the contributions, one then obtains
\[
\begin{split}
\sum_{\mathbf{i}\in\mathbb{Z}^N}\frac{\mathrm{Cap}_2(\{x_\mathbf{i}\};Q_\mathbf{i})}{L^N}\, \|\varphi\|^2_{L^2(Q_\mathbf{i})}&\lesssim \sum_{\mathbf{i}\in\mathbb{Z}^N}\left[\int_{Q_\mathbf{i}} |\Delta \varphi|^2\,dx+\frac{1}{L^2}\,\int_{{Q_\mathbf{i}}} |\nabla \varphi|^2\,dx\right]\\
&=\int_\Omega |\Delta \varphi|^2\,dx+\frac{1}{L^2}\,\int_\Omega |\nabla \varphi|^2\,dx.
\end{split}
\]		
This is still not the end of the story: we need to bound uniformly the capacity of the points $x_\mathbf{i}$, on the left-hand side; more importantly, we need to remove the $L^2$ norm of $\nabla \varphi$ on the right-hand side. The first issue is easily handled, it is sufficient to choose the edge length to be $L=\alpha\,r_\Omega$, for $\alpha$ sufficiently large. This permits to say that
\[
\frac{1}{r_\Omega^4}\lesssim \frac{\mathrm{Cap}_2(\{x_\mathbf{i}\};Q_\mathbf{i})}{L^N},
\]
by using some basic properties of the capacity. As for the second issue, i.e. the presence of the $L^2$ norm of $\nabla \varphi$, we can observe that $\varphi$ is compactly supported on $\Omega$, thus we have the classical identity
\[
\int_\Omega |\nabla \varphi|^2\,dx=-\int_\Omega \varphi\,\Delta \varphi\,dx.
\]
Thus, we can simply interpolate $\nabla \varphi$ between  $\varphi$ and $\Delta \varphi$ (with a universal constant independent of $\Omega$). A standard use of Young's inequality then leads to 
\[
\frac{1}{r_\Omega^4}\,\|\varphi\|_{L^2(\Omega)}^2\lesssim \int_\Omega |\Delta \varphi|^2\,dx,
\]
as desired. This gives the desired lower bound on $\Lambda(\Omega)$.
\vskip.2cm\noindent
When coming to the fractional case, i.e Theorem \ref{thm:frac}, apparently (almost) nothing changes. It would be sufficient to repeat the above arguments, by replacing the Maz'ya-Poincar\'e inequality \eqref{ammazzapoincare} with a suitable fractional counterpart, i.e. something of the type
\[
\frac{\mathrm{Cap}_s(\{x_\mathbf{i}\};Q_\mathbf{i})}{L^{N}}\,\int_{Q_\mathbf{i}}|\varphi|^2\,dx\lesssim \iint_{Q_\mathbf{i}\times Q_{\mathbf{i}}} \frac{|\nabla \varphi(x)-\nabla \varphi(y)|^2}{|x-y|^{N+2\,(s-1)}}\,dx+\frac{1}{L^{2\,s-2}}\,\int_{{Q_\mathbf{i}}} |\nabla\varphi|^2\,dx.
\] 
Now, the main difficulty is carefully keeping track of the dependence on $s$ of the relevant constants and assure that they exhibit
the correct asymptotic behavior, in the two regimes $s\nearrow 2$ and $s\searrow N/2$. This guarantees that, even if we do not identify the sharp constants, we produce an estimate with the correct order of magnitude in terms of $s$.
\par
We thus try to proceed as in the local case and use the previous fractional Maz'ya-Poincar\'e inequality: after summing up all the contributions and using that, with a judicious choice of $L$, we have
\[
\frac{\mathrm{Cap}_s(\{0\};Q_1)}{r_\Omega^{2\,s}}\lesssim \frac{\mathrm{Cap}_s(\{x_\mathbf{i}\};Q_\mathbf{i})}{L^N},
\] 
we can obtain this time
\begin{equation}
\label{BRPG}
\frac{\mathrm{Cap}_s(\{0\};Q_1)}{L^{2\,s}}\,\|\varphi\|_{L^2(\Omega)}^2\lesssim \sum_{\mathbf{i}\in\mathbb{Z}^N}  \iint_{Q_\mathbf{i}\times Q_{\mathbf{i}}} \frac{|\nabla\varphi(x)-\nabla\varphi(y)|^2}{|x-y|^{N+2\,(s-1)}}\,dx+\frac{1}{L^{2\,s-2}}\,\int_{\Omega} |\nabla\varphi|^2\,dx.
\end{equation}
Here we encounter a subtle point: the first term on the right-hand side does not coincide with the full Gagliardo-Slobodecki\u{\i} seminorm of $\nabla u$. Actually, it is a smaller quantity keeping into account only the short-range interactions in $Q_{\mathbf{i}}$. In other words, we are discarding all the contributions coming from
\[
 \iint_{Q_\mathbf{i}\times (\mathbb{R}^N\setminus Q_{\mathbf{i}})} \frac{|\nabla \varphi(x)-\nabla \varphi(y)|^2}{|x-y|^{N+2\,(s-1)}}\,dx.
\]
Of course, one could use the following brutal estimate
\[
\sum_{\mathbf{i}\in\mathbb{Z}^N}  \iint_{Q_\mathbf{i}\times Q_{\mathbf{i}}} \frac{|\nabla \varphi(x)-\nabla \varphi(y)|^2}{|x-y|^{N+2\,(s-1)}}\,dx\le \iint_{\mathbb{R}^N\times\mathbb{R}^N} \frac{|\nabla \varphi(x)-\nabla \varphi(y)|^2}{|x-y|^{N+2\,(s-1)}}\,dx.
\]
However, in dimension $N=2$, when $s$ comes closer and closer to $N/2=1$, we see that for $\varphi\in C^\infty_0(\Omega)$
\[
\sum_{\mathbf{i}\in\mathbb{Z}^N}\iint_{Q_\mathbf{i}\times Q_{\mathbf{i}}} \frac{|\nabla \varphi(x)-\nabla \varphi(y)|^2}{|x-y|^{N+2\,(s-1)}}\,dx\sim \int_{\Omega} |\nabla \varphi|^2\,dx,\qquad \text{as}\ s\searrow 1,
\]
while 
\[
\iint_{\mathbb{R}^N\times\mathbb{R}^N} \frac{|\nabla \varphi(x)-\nabla \varphi(y)|^2}{|x-y|^{N+2\,(s-1)}}\,dx\sim \frac{1}{s-1}\,\int_{Q_{\mathbf{i}}} |\nabla \varphi|^2\,dx,\qquad \text{as}\ s\searrow 1.
\]
This difference is due to the non-integrability of the singular kernel $|x-y|^N$ ``at infinity'', i.e. when $x$ and $y$ are far away (such a possibility is ruled out in the first case, since $x,y\in Q_\mathbf{i}$).
This comparison shows that there may be a first ``scale defect'', in the resulting estimate.
\par
Unfortunately, this is not the unique delicate point: also the second term in the right-hand side of \eqref{BRPG} may be problematic. Indeed, if we wish to reabsorb the $L^2$ norm of $\nabla \varphi$ in the left-hand side through an interpolation estimate, we have to take into account the presence of the factor $\mathrm{Cap}_s(\{0\};Q_1)$ in the left-hand side of \eqref{BRPG}. Thus, after absorption, this will let appear on the right-hand side a term of the type
\[
\frac{1}{\mathrm{Cap}_s(\{0\};Q_1)},
\]
coupled with the Gagliardo-Slobodecki\u{\i} seminorm of $\nabla\varphi$. The final outcome would be again an estimate with a ``scale defect'', as $s$ approaches the critical regime $N/2$, where points start to be removable.
\vskip.2cm\noindent
In order to fix these issues, we need to proceed much more carefully and prove an {\it ad hoc} Maz'ya-Poincar\'e inequality, which prevents the formation of these ``scale defects''. In particular, we prove an estimate of the following type
\[
\frac{\mathrm{Cap}_s(\{x_\mathbf{i}\};Q_\mathbf{i})}{L^{2\,s}}\,\int_{Q_\mathbf{i}}|\varphi|^2\,dx\lesssim \iint_{Q_\mathbf{i}\times \mathbb{R}^N} \frac{|\nabla \varphi(x)-\nabla \varphi(y)|^2}{|x-y|^{N+2\,(s-1)}}\,dx+\frac{\mathrm{Cap}_s(\{x_\mathbf{i}\};Q_\mathbf{i})}{L^{2\,s}}{L^{2\,s-2}}\,\int_{{Q_\mathbf{i}}} |\nabla \varphi|^2\,dx,
\] 
see Corollary \ref{coro:francesco} below.
Observe the presence of an ``asymmetric'' Gagliardo-Slobodecki\u{\i} seminorm of $\nabla\varphi$ (an idea taken from \cite{BozTal}), as well as the appearing of the modulating factor $\mathrm{Cap}_s(\{x_\mathbf{i}\};Q_\mathbf{i})$ in front of the $L^2$ norm of the gradient. This inequality permits to circumvent the problems highlighted above and provide the desired lower bound on $\Lambda^s$, with the correct asymptotic behavior of the relevant constant.

	\subsection{Plan of the paper}
We start by presenting the main notation, definitions and basic results in Section \ref{sec:2}. We then present in Section \ref{sec:3} a collection of Poincar\'e--type inequalities, which will be useful throughout the paper. As usual, we pay due attention to the quality of the constants involved, i.e. on their dependence on the fractional parameter $s$. Section \ref{sec:4} contains some results about (fractional) capacities, in particular we discuss the removability issues for points, according to the value of $s$. This is complemented by the results of Section \ref{sec:5}, which prove in particular a quantitative lower bound on the capacity of points. This is obtained by means of a higher order Morrey inequality (see Proposition \ref{prop: morrey ineq} below). With Section \ref{sec:6} we arrive at the core of the matter: we present the Maz'ya-Poincar\'e--type inequalities which are the cornerstones lying at the basis of Theorems \ref{thm:local} and \ref{thm:frac}. In particular, the crucial result will be Theorem \ref{thm:mp_frac} (and expecially its Corollary \ref{coro:francesco}). The proof of the main results can then be found in Section \ref{sec:7}. Finally, in Section \ref{sec:8} we give some examples showing the optimality of our estimates, in the asymptotic regimes $s\nearrow 2$ and $s\searrow N/2$. The paper is concluded by a technical appendix, containing a technical fact obtained through real interpolation techniques.

\begin{ack}
This project started as a spin-off of the PhD Thesis of the first author. The initial input was a question raised by Davide Buoso, during the special session on Spectral Optimization of the ``{\it AMS-SMF-EMS Joint International Meeting\,}'', held in Grenoble in July 2022. We wish to thank him. 
After a period of neglect, the second author revitalized the project. 
\par
We thank Paolo Baroni for sharing with us a preliminary version of his paper \cite{AnBa}.
\par
F.\, Bozzola is supported by Funda\c{c}\~ao para a Ci\^encia e a Tecnologia (FCT),
through national funds (FCT/MCTES, PIDDAC), within the framework of GFM,
reference UID/00208/2025 - IST-ID.
\par
F.\, Bozzola is member of the Gruppo
Nazionale per l'Analisi Matematica, la Probabilit\`a e le loro
Applicazioni (GNAMPA) of the Istituto Nazionale di Alta Matematica
(INdAM) and partially funded by the INdAM-GNAMPA project ``{\it Problemi di
ottimizzazione di forma in contesti anisotropi e non-locali\,}'', CUP:
E53C25002010001.
\par
L.\,Brasco has been financially supported by the {\it Fondo di Ateneo per la Ricerca} FAR 2024 of the University of Ferrara.
\end{ack}
\begin{disclosure}
This paper is 100\% AI free. 
\end{disclosure}	
	\section{Preliminaries}
	\label{sec:2}
	\subsection{Notation}
	The euclidean open ball of $\mathbb{R}^N$ of radius $r$ and centered at $x_0$ will be denoted by
	\[
	B_r(x_0)=\Big\{x\in\mathbb{R}^N\, :\, |x-x_0|<r\Big\}.
	\]
	When $x_0 = 0$, we omit to specify the center. The $N-$dimensional open hypercube of radius $r$ and centered at $x_0$ will be denoted by  
	\[
	Q_r(x_0)=\prod_{i=1}^N \left(x_0^i-r,x_0^i+r\right),\qquad \text{where}\ x_0=(x_0^1,\dots,x_0^N).
	\]
	Also in this case, when $x_0 = 0$ we omit to specify the center.
	\par
	If $E\subseteq\mathbb{R}^N$ is a measurable set, the notation $|E|$ stands for its $N-$dimensional Lebesgue measure. As usual, the symbol $\omega_N$ indicates the volume of the unit ball. For a couple of vectors $x,\,y\in\mathbb{R}^N$, their scalar product is $\langle x,y\rangle$. For $\varphi\in L^1_{\rm loc}(\mathbb{R}^N)$ and a measurable set $E\subseteq\mathbb{R}^N$ with positive finite measure, we set
	\[
	\mathrm{av}(\varphi;E):=\fint_E \varphi\,dx=\frac{1}{|E|}\,\int_E \varphi\,dx.
	\]
The partial derivatives of a differentiable function $\varphi$ will be denoted by $\varphi_{x_i}$, for $i\in\{1,\dots,N\}$. 
	Accordingly, we will write $\varphi_{x_ix_j}$ to indicate the components of the Hessian matrix $D^2 \varphi$ of $\varphi$. 
\subsection{Some functional spaces}
For an open set $E\subseteq\mathbb{R}^N$, we denote by $W^{2,2}(E)$ the Sobolev space
\[
\Big\{\varphi\in L^2(E)\, :\, \nabla \varphi\in L^2(E)\ \text{and}\  D^2 \varphi\in L^2(E)\Big\},
\] 
endowed with the norm
\[
\|\varphi\|_{W^{2,2}(E)}=\|\varphi\|_{L^2(E)}+\|\nabla \varphi\|_{L^2(E)}+\|D^2 \varphi\|_{L^2(E)},
\]
where 
\[
\|D^2 \varphi\|_{L^2(E)}^2=\sum_{i,j=1}^N \int_E |\varphi_{x_ix_j}|^2\,dx.
\]
We also define $W^{2,2}_0(E)$ to be the closure of $C^\infty_0(E)$ in $W^{2,2}(E)$. The following facts are well-known.
\begin{lemma}
For every $\varphi\in C^\infty_0(\mathbb{R}^N)$ we have that
 \begin{equation}
 \label{HessLap}
 \|D^2 \varphi\|_{L^2(\mathbb{R}^N)}=\|\Delta \varphi\|_{L^2(\mathbb{R}^N)},
\end{equation}
and
\begin{equation}
\label{interpol2}
\|\nabla \varphi\|_{L^2(\mathbb{R}^N)}\le \sqrt{\|\varphi\|_{L^2(\mathbb{R}^N)}\, \|\Delta \varphi\|_{L^2(\mathbb{R}^N)}}.
\end{equation}
In particular, for every $E\subseteq\mathbb{R}^N$ open set the following norm 
\[
\varphi\mapsto \|\varphi\|_{L^2(E)}+\|\Delta\varphi\|_{L^2(E)},
\]
is an equivalent one on $W^{2,2}_0(E)$.
\end{lemma}
\begin{proof}
The identity \eqref{HessLap} easily follows by exchanging the order of differentiation and integrating by parts twice.
As for the interpolation inequality: since $u \in C^\infty_0(\mathbb{R}^N)$, by integrating by parts and using H\"older's inequality, we have
		\[		
\int_{\mathbb{R}^N} |\nabla \varphi|^2\,dx = -\int_{\mathbb{R}^N} \varphi\,\Delta \varphi\,dx \le \|\varphi\|_{L^2(\mathbb{R}^N)}\, \|\Delta \varphi\|_{L^2(\mathbb{R}^N)},
		\]
which gives the desired conclusion.		
\end{proof}
In what follows, we will use the following notation
\[
[\varphi]_{W^{1,2}(E)}=\left(\int_E |\nabla \varphi|^2\,dx\right)^\frac{1}{2}\qquad \text{and}\qquad [\varphi]_{W^{2,2}(E)}=\left(\int_E |\Delta \varphi|^2\,dx\right)^\frac{1}{2}. 
\]
We now introduce our main notation in the context of fractional Sobolev spaces. We refer to \cite{EE} and \cite{Leo_frac}  for a general overview and for further references. Let $0<\sigma<1$ and let $E\subseteq\mathbb{R}^N$ be an open set. We recall the definition of {\it Gagliardo-Slobodecki\u{\i} seminorm}, given by
	\[
	[\varphi]_{W^{\sigma,2}(E)}:=\left(\iint_{E\times E}\frac{|\varphi(x)- \varphi(y)|^2}{|x-y|^{N+2\,\sigma}}\,dx\,dy\right)^\frac{1}{2},
	\]
	for a (possibly vector valued) $\varphi\in L^2_{\rm loc}(E)$.
For $1<s<2$, we then consider the Sobolev-Slobodecki\u{\i} space defined by 
	\[
	W^{s,2}(E)=\Big\{\varphi\in L^2(E)\, :\, \nabla \varphi \in L^2(E)\ \text{and}\ [\nabla \varphi]_{W^{s-1,2}(E)}<+\infty \Big\},
	\]
	endowed with the norm
	\[
	\|\varphi\|_{W^{s,2}(E)}=\|\varphi\|_{L^2(E)}+\|\nabla\varphi\|_{L^2(E)}+[\nabla\varphi]_{W^{s-1,2}(E)}.
	\]
We also set $\widetilde{W}^{s,2}_0(E)$ to be the closure of $C^\infty_0(E)$ in the space $W^{s,2}(\mathbb{R}^N)$. As in the local case, one can prove that
\[
\varphi\mapsto \|\varphi\|_{L^2(E)}+[\nabla\varphi]_{W^{s-1}(\mathbb{R}^N)},
\]
is an equivalent norm on $\widetilde{W}^{s,2}_0(E)$ (this is a consequence of Proposition \ref{lm:interpolazione} below).
Finally, for later reference, we observe that we have
\[
[\nabla\varphi]_{W^{s-1,2}(E)}^2=\sum_{i=1}^N [\varphi_{x_i}]^2_{W^{s-1,2}(E)}.
\]
	\subsection{Basic facts on the Fourier transform}

	We recall some properties on the Fourier transform, for which we refer to \cite[Chapter VII, Section 9]{Hor}. 
	\par For every $\varphi\in C^\infty_0(\mathbb{R}^N)$, the symbol $\mathcal{F}[\varphi]$ stands for the {\it Fourier transform of $\varphi$} given by
\[
		\mathcal{F}[\varphi](\xi)=\frac{1}{(2\,\pi)^\frac{N}{2}}\,\int_{\mathbb{R}^N}\varphi(x)\,e^{-\text{i}\,\langle x,\xi\rangle}\,dx,\qquad \mbox{ for } \xi\in\mathbb{R}^N.
\]
We recall that 
\[
\mathcal{F}[\varphi_{x_i}](\xi)=\text{i}\,\xi_i\,\mathcal{F}[\varphi](\xi),\qquad \text{for}\ \xi\in \mathbb{R}^N,\ i\in\{1,\dots,N\},
\]
and
\[
\mathcal{F}[\varphi_{x_i x_k}](\xi)=-\xi_i\,\xi_k\,\mathcal{F}[\varphi](\xi),\qquad \text{for}\ \xi\in \mathbb{R}^N,\ i,k\in\{1,\dots,N\}.
\]	
By the inversion formula, we have
\[
			\varphi(x)=\frac{1}{(2\,\pi)^\frac{N}{2}}\,\int_{\mathbb{R}^N}\mathcal{F}[\varphi](\xi)\,e^{\text{i}\,\langle x,\xi\rangle}\,d\xi,\qquad \mbox{ for } x\in\mathbb{R}^N.
\]
	We also recall the following {\it Plancherel identities} 
	\begin{equation}\label{eq: plancherel 1}
		\int_{\mathbb{R}^N}|\varphi|^2\,dx=\frac{1}{(2\,\pi)^N}\,\int_{\mathbb{R}^N}\Big|\mathcal{F}[\varphi]\Big|^2\,d\xi,\qquad \int_{\mathbb{R}^N}|\nabla \varphi|^2\,dx=\frac{1}{(2\,\pi)^N}\,\int_{\mathbb{R}^N}|\xi|^{2}\,\Big|\mathcal{F}[\varphi]\Big|^2\,d\xi. 
	\end{equation}
and
\[
		\int_{\mathbb{R}^N}|\Delta \varphi|^2\,dx=\frac{1}{(2\,\pi)^N}\,\int_{\mathbb{R}^N}|\xi|^4\,\Big|\mathcal{F}[\varphi]\Big|^2\,d\xi.
\]	For the fractional seminorm, we have
	\begin{equation}\label{eq: plancherel s-1}
		\begin{split}
			\iint_{\mathbb{R}^N\times\mathbb{R}^N}\frac{| \varphi(x)- \varphi(y)|^2}{|x-y|^{N+2\,s}}\,dx\,dy&=\frac{1}{(2\,\pi)^N}\,\frac{1}{A_{s}}\,\int_{\mathbb{R}^N}|\xi|^{2\,s}\,\Big|\mathcal{F}[\varphi]\Big|^2\,d\xi,\qquad 0<s<1,
		\end{split}
	\end{equation}
	where \begin{equation}\label{eq: def A(s-1)}
		{A}_{s}=\left(\int_{\mathbb{R}^N}\frac{\left|e^{\text{i}\,\langle x,\xi\rangle}-1\right|^2}{|x|^{N+2\,s}}\,dx\right)^{-1}.
	\end{equation}
	In particular, 
	\begin{equation} \label{asym-Aesse}
			A_s\sim 1-s,\ \text{for}\ s\nearrow 1\qquad\text{and}\qquad A_s\sim s,\ \text{for}\ s\searrow 0.
	\end{equation}
 	By \eqref{eq: plancherel s-1}, the following Plancherel identity holds true for $1<s<2$
	\[	\begin{split}
			\iint_{\mathbb{R}^N\times\mathbb{R}^N}\frac{|\varphi_{x_i}(x)-\varphi_{x_i}(y)|^2}{|x-y|^{N+2\,(s-1)}}\,dx\,dy&
			=\frac{1}{(2\,\pi)^N}\,\frac{1}{A_{s-1}}\,\int_{\mathbb{R}^N}|\xi|^{2\,(s-1)}\,\left|\mathcal{F}[\varphi_{x_i}]\right|^2\,d\xi \\
			&=\frac{1}{(2\,\pi)^N}\,\frac{1}{A_{s-1}}\,\int_{\mathbb{R}^N}|\xi|^{2\,(s-1)}\,\xi_i^2\,\Big|\mathcal{F}[ \varphi]\Big|^2\,d\xi .
		\end{split}
		\]
		In particular, by summing over $i\in\{1,\dots,N\}$, we get
		\begin{equation}
	\label{eq: plancherel s}
[\varphi]^2_{W^{s,2}(\mathbb{R}^N)}=\frac{1}{(2\,\pi)^N}\,\frac{1}{A_{s-1}}\,\int_{\mathbb{R}^N}|\xi|^{2\,s}\,\Big|\mathcal{F}[\varphi]\Big|^2\,d\xi,\qquad 1<s<2.
	\end{equation}
The following interpolation inequalities, which are analogous to \eqref{interpol2}, will be useful.	
	\begin{proposition}\label{lm:interpolazione}
		Let $1<s<2$. For every $\varphi\in C^\infty_0(\mathbb{R}^N)$ and every $\delta>0$ the following inequalities hold
		\begin{equation}\label{eq: interpolaz 1}
			\|\nabla \varphi\|^2_{L^2(\mathbb{R}^N)}\le \delta\,\left(\frac{s-1}{s}\right)\,\| \varphi\|^2_{L^2(\mathbb{R}^N)}+2\,\frac{A_{s-1}}{\delta^{s-1}}\,[\varphi]^2_{W^{s,2}(\mathbb{R}^N)},
		\end{equation}
				and 
		\begin{equation}\label{eq: interpolaz 2}
			[\varphi]^2_{W^{s-1,2}(\mathbb{R}^N)}\le \frac{\delta^{s-1}}{A_{s-1}}\,\| \varphi\|^2_{L^2(\mathbb{R}^N)}
			+\frac{1}{\delta}\,\left(\frac{s}{s-1}\right)\,[\varphi]^2_{W^{s,2}(\mathbb{R}^N)}.
		\end{equation}
The constant $A_{s-1}$ is given by \eqref{eq: def A(s-1)}.
		In particular, 
	\begin{equation} \label{asintotico-fourier}
			A_{s-1}\sim 2-s,\quad \text{for}\ s\nearrow 2,\qquad A_{s-1}\sim s-1\quad \text{for}\ s\searrow 1.
	\end{equation}
	\end{proposition}
	\begin{proof} 
		By \eqref{eq: plancherel 1}, we can write
		\[\begin{split}
			{(2\,\pi)^N}\,\int_{\mathbb{R}^N}|\nabla \varphi|^2\,dx&=\int_{\mathbb{R}^N}|\xi|^{2}\,|\mathcal{F}[\varphi](\xi)|^2\,d\xi\\
			&
			=\int_{\{|\xi|\le M\}}|\xi|^{2}\,|\mathcal{F}[\varphi](\xi)|^2\,d\xi+\int_{\{|\xi|> M\}}|\xi|^{2}\,|\mathcal{F}[\varphi](\xi)|^2\,d\xi\\
			&\le M^2\,\int_{\{|\xi|\le M\}}|\mathcal{F}[\varphi](\xi)|^2\,d\xi+M^{2-2\,s}\,\int_{\{|\xi|> M\}}|\xi|^{2\,s}\,|\mathcal{F}[\varphi](\xi)|^2\,d\xi,
		\end{split}\]
		for $M > 0$. In particular, by taking $\delta>0$ and choosing
		\[
		M=\sqrt{\delta\,\frac{s-1}{s}},
		\]
	we can infer that
		\[
		{(2\,\pi)^N}\,\int_{\mathbb{R}^N}|\nabla \varphi|^2\,dx\le \delta\,\left(\frac{s-1}{s}\right)\,\int_{\mathbb{R}^N}|\mathcal{F}[\varphi](\xi)|^2\,d\xi+\frac{1}{\delta^{s-1}}\,\left(\frac{s}{s-1}\right)^{s-1}\,\int_{\mathbb{R}^N}|\xi|^{2\,s}\,|\mathcal{F}[\varphi](\xi)|^2\,d\xi,
		\]
		for every $\delta>0$. By observing that
	\[
	s\mapsto \left(\frac{s}{s-1}\right)^{s-1}\quad \text{is increasing and converges to $2$ as $s\nearrow 2$},
	\]	
		this yields \eqref{eq: interpolaz 1}, by \eqref{eq: plancherel 1} and \eqref{eq: plancherel s}. 
\vskip.2cm\noindent
		To prove \eqref{eq: interpolaz 2}, we proceed similarly. By \eqref{eq: plancherel s-1} we can write
		\[
		\begin{split}
			{(2\,\pi)^N\,A_{s-1}}\,\iint_{\mathbb{R}^N\times\mathbb{R}^N}\frac{|\varphi(x)-\varphi(y)|^2}{|x-y|^{N+2\,(s-1)}}\,dx\,dy
			&=\int_{\mathbb{R}^N}|\xi|^{2\,(s-1)}\,|\mathcal{F}[\varphi](\xi)|^2\,d\xi\\
			&\le M^{2\,s-2}\,\int_{\{|\xi|\le M\}}|\mathcal{F}[\varphi](\xi)|^2\,d\xi\\
			&+M^{-2}\,\int_{\{|\xi|> M\}}|\xi|^{2\,s}\,|\mathcal{F}[\varphi](\xi)|^2\,d\xi\\
		\end{split}
		\]
By taking $\delta>0$ and choosing $M$ as before, we easily get the conclusion.
	\end{proof}
	\subsection{Asymptotics}
	For the higher order seminorms previously introduced, we will need the following result. This is an easy consequence of the so-called {\it Bourgain-Brezis-Mironescu} and {\it Maz'ya-Shaposhnikova formulas} for fractional seminorms of order $0<\sigma<1$.
	\begin{proposition}\label{prop:BBM}
		Let $\Omega \subseteq \mathbb{R}^N$ be an open set. For every $u\in C^\infty_0(\Omega)$, we have 
		\begin{equation}\label{eq:BBM}
		 \lim_{s\nearrow2}(2-s)\,[\varphi]^2_{W^{s,2}(\mathbb{R}^N)}= \frac{\omega_N}{2}\int_{\Omega}|\Delta \varphi|^2\,dx,
		\end{equation}
		and
		\begin{equation}\label{eq:MS}
		 \lim_{s\searrow 1}(s-1)\,[\varphi]^2_{W^{s,2}(\mathbb{R}^N)}=N\,\omega_N\,\int_{\Omega}|\nabla \varphi|^2\,dx.
		\end{equation}
	\end{proposition}
	\begin{proof}
		Let $\varphi \in C^\infty_0(\Omega)$, we extend it by zero to the whole $\mathbb{R}^N$. For $i\in\{1,\ldots,N\}$, we apply to each partial derivative $\varphi_{x_i}$ the Bourgain-Brezis-Mironescu's Theorem, see \cite{BBM}, \cite[Corollary 3.20]{EE} or \cite[Theorem 7.39]{Leo_frac} for instance. This yields 
\[
			\begin{split}
				\lim_{s \nearrow 2} (2-s)\, [\varphi]^2_{W^{s, 2}(\mathbb{R}^N)} &=  \sum_{i=1}^N\lim_{s \nearrow 2}(2-s)\,\iint_{\mathbb{R}^N\times\mathbb{R}^N} \frac{|\varphi_{x_i}(x)- \varphi_{x_i} (y)|^2}{|x-y|^{N+2\,(s-1)}}\,dx\,dy \\
				&=  \sum_{i=1}^N\lim_{\sigma \nearrow 1}(1-\sigma)\,\iint_{\mathbb{R}^N\times\mathbb{R}^N} \frac{|\varphi_{x_i}(x)- \varphi_{x_i} (y)|^2}{|x-y|^{N+2\,\sigma}}\,dx\,dy \\
				 &=\frac{\omega_N}{2}\,\sum_{i=1}^N \int_{\mathbb{R}^N}|\nabla \varphi_{x_i}|^2\,dx. 
			\end{split}
\]
By recalling the identity \eqref{HessLap}, we get \eqref{eq:BBM}.
\par
Similarly, for the limit as $s$ goes to $1$ it is sufficient to appeal to \cite[Theorem 3]{MS}
			\[
			\begin{split}
				\lim_{s \searrow 1} (s-1)\, [\varphi]^2_{W^{s, 2}(\mathbb{R}^N)} &=  \sum_{i=1}^N\lim_{s \searrow 1}(s-1)\,\iint_{\mathbb{R}^N\times\mathbb{R}^N} \frac{|\varphi_{x_i}(x)- \varphi_{x_i} (y)|^2}{|x-y|^{N+2\,(s-1)}}\,dx\,dy \\ &=N\,\omega_N\,\sum_{i=1}^N \int_{\mathbb{R}^N}|\varphi_{x_i}|^2\,dx. 
			\end{split}
			\]
This establishes \eqref{eq:MS}, as well.			
\end{proof}
\begin{corollary}
\label{coro:BBM}
	Let $\Omega \subseteq \mathbb{R}^N$ be an open set. There holds 
		\begin{equation} \label{asym-autovalore}
			\limsup_{s\nearrow2} (2-s)\,\Lambda^s(\Omega)\le\dfrac{\omega_N}{2}\,\Lambda(\Omega)\qquad \text{and}\qquad			\limsup_{s\searrow 1} (s-1)\,\Lambda^s(\Omega)\le N\,\omega_N\,\lambda(\Omega).
		\end{equation}
		\end{corollary}
\begin{proof}
By using \eqref{eq:BBM} and recalling the definition \eqref{lambda_BR} of $\Lambda^s$, for every $\varphi\in C^\infty_0(\Omega)$ with unit $L^2$ norm we can infer that 
		\[
		\limsup_{s\nearrow2} (2-s)\,\Lambda_1^s(\Omega)\le \limsup_{s\nearrow2} (2-s)\,[\varphi]^2_{W^{s,2}(\mathbb{R}^N)}=\frac{\omega_N}{2}\,\|\Delta \varphi\|^2_{L^2(\Omega)}.
		\]
This yields  the first limit in \eqref{asym-autovalore}, by the arbitrariness of $\varphi$. The proof of the other limit is similar and left to the reader.
	\end{proof}
	
	\section{A collection of Poincar\'e inequalities}
\label{sec:3}	
We start with the following consequence of Proposition \ref{lm:interpolazione}, which will be useful. These inequalities can be seen as ``intermediate'' Poincar\'e inequalities, where the maximal order of (fractional) differentiability controls some lower order derivatives.
\begin{proposition}
\label{prop:mezzapoin}
Let $1<s<2$ and let $\Omega\subseteq\mathbb{R}^N$ be an open set. Then for every $\varphi\in C^\infty_0(\Omega)$ we have
\[
\Big(\Lambda^s(\Omega)\Big)^\frac{s-1}{s}\,\|\nabla \varphi\|_{L^2(\Omega)}^2\le \frac{s+1}{s}\,\Big(2\,s\,A_{s-1}\Big)^\frac{1}{s}\,[\varphi]^2_{W^{s,2}(\mathbb{R}^N)},
\]
and
\begin{equation}
\label{luna}
\Big(A_{s-1}\,\Lambda^s(\Omega)\Big)^\frac{1}{s}\,[\varphi]^2_{W^{s-1,2}(\mathbb{R}^N)}\le \,\frac{s^{2-\frac{1}{s}}}{(s-1)^{2\,(1-\frac{1}{s})}}\,[\varphi]^2_{W^{s,2}(\mathbb{R}^N)},
\end{equation}
where the constant $A_{s-1}$ is given by \eqref{eq: def A(s-1)}.
\end{proposition}
\begin{proof}
We can suppose that $\Lambda^s(\Omega)>0$, otherwise there is nothing to prove.
In order to prove the first inequality, it is sufficient to use the estimate
\begin{equation}
\label{prapra}
\|\varphi\|^2_{L^2(\Omega)}\le \frac{1}{\Lambda^s(\Omega)}\,[\varphi]^2_{W^{s,2}(\mathbb{R}^N)},
\end{equation}
in \eqref{eq: interpolaz 1} and then choose
\[
\delta=\Big(2\,s\,A_{s-1}\,\Lambda^s(\Omega)\Big)^\frac{1}{s},
\]
there. 
\par
The proof of the second inequality is similar, we just need to use \eqref{prapra} in \eqref{eq: interpolaz 2} and choose this time
\[
\delta=\left(\frac{s}{(s-1)^2}\,A_{s-1}\,\Lambda^s(\Omega)\right)^\frac{1}{s}.
\]
We leave the details to the reader.
\end{proof}
The next result permits to relate the two constants $\lambda$ and $\Lambda^s$.
\begin{proposition}
\label{lm:lowerbound}
Let $1<s<2$ and let $\Omega\subseteq\mathbb{R}^N$ be an open set. Then we have
\[
A_{s-1}\,\Lambda^s(\Omega)\ge \frac{1}{8}\,\Big(\lambda(\Omega)\Big)^s,
\]
where $A_{s-1}$ is still given by \eqref{eq: def A(s-1)}. 
\end{proposition}	
\begin{proof}
Let us suppose that $\lambda(\Omega)>0$, otherwise there is nothing to prove.
Let $\varphi\in C^\infty_0(\Omega)$ and combine the first order Poincar\'e inequality
\[
\lambda(\Omega)\,\int_\Omega |\varphi|^2\,dx\le \int_\Omega |\nabla \varphi|^2\,dx,
\]
with the inequality \eqref{eq: interpolaz 1}. This gives		
\[
\lambda(\Omega)\,\|\varphi\|^2_{L^2(\Omega)}\le \delta\,\left(\frac{s-1}{s}\right)\,\| \varphi\|^2_{L^2(\Omega)}+2\,\frac{A_{s-1}}{\delta^{s-1}}\,[\varphi]^2_{W^{s,2}(\mathbb{R}^N)},
\]
which holds for every $\delta>0$. We choose 
\[
\delta=\frac{\lambda(\Omega)}{2}\,\frac{s}{s-1},
\]
so to get
\[
\frac{\lambda(\Omega)}{2}\,\|\varphi\|^2_{L^2(\Omega)}\le 2\,\left(\frac{\lambda(\Omega)}{2}\right)^{1-s}\,\left(\frac{s-1}{s}\right)^{s-1}\,A_{s-1}\,[\varphi]^2_{W^{s,2}(\mathbb{R}^N)}.
\]
We can simplify the common term $\lambda(\Omega)$ on both sides. 
By arbitrariness of $\varphi$, recalling the definition \eqref{lambda_BR} of $\Lambda^s(\Omega)$, after some manipulations we get
\[
\Big(\lambda(\Omega)\Big)^s\le 2^{1+s}\,\left(\frac{s-1}{s}\right)^{s-1}\,A_{s-1}\,\Lambda^s(\Omega).
\]
We can now simply bound for every $1<s<2$
\[
\left(\frac{s-1}{s}\right)^{s-1}\le \lim_{x\searrow 1} \left(\frac{x-1}{x}\right)^{x-1}=1,
\]
in order to get the estimate.
This concludes the proof.
\end{proof}
The next simple Poincar\'e--type inequality will be used somewhere in the paper. The proof can be found for example in \cite[Lemma 3.2]{AnBa}.	
	\begin{lemma}[Poincar\'e-type inequality]
	\label{lm:rubbato}
		Let $0 < \sigma < 1$ and let $x_0 \in \mathbb{R}^N$ and $r > 0$. For every $\varphi \in W^{1,2}(B_r(x_0))$, we have 
\[
			[\varphi]^2_{W^{\sigma, 2}(B_r(x_0))} \le \left(\frac{C}{1-\sigma}\right)\,r^{2\,(1-\sigma)}\,\|\nabla \varphi\|^2_{L^2(B_r(x_0))},
\]
for some $C=C(N)>0$.
	\end{lemma}
The following (possibly) asymmetric Poincar\'e-Wirtinger inequality will play a crucial role. It generalizes \cite[Lemma 3.5]{BiaBra24}, the latter being concerned with the case $d=D$. The idea of using asymmetric seminorm is taken from \cite[Section 4]{BozTal}. 
	\begin{proposition}[Fractional Poincar\'e-Wirtinger inequality] \label{lm:pw_frac}
		Let $0 < \sigma < 1$ and $0 < d \le D < +\infty$. For every $\varphi \in C^\infty(\overline{Q_D(x_0)})$, we have
		\begin{equation*}
			\begin{split}
				\|\varphi-{\rm av}(\varphi; Q_d(x_0))\|^2_{L^2(Q_d(x_0))} \le \mathcal{W}\,\sigma\,(1-\sigma)\,d^{2\,\sigma}\,\iint_{Q_d(x_0) \times Q_D(x_0)} \frac{|\varphi(x)-\varphi(y)|^2}{|x-y|^{N+2\,\sigma}}\,dx\,dy,
			\end{split}
		\end{equation*}
		where  
		\begin{equation}
		\label{vudoppia}
			\mathcal{W} = \mathcal{W}\left(N,\sigma, d/D\right) =\frac{W_N}{\sigma+(1-\sigma)\, \left(1 - ((2\,d)/D)^{2\,\sigma}\right)_+},
		\end{equation}
		for a constant $W_N > 0$ depending on the dimension only.
	\end{proposition}
	\begin{proof}
	As in \cite{BiaBra24}, we will need the following $K-$functional
	\[
				K(t, \varphi) := \min_{v \in W^{1,2}(Q_d(x_0))} \Big[\|\varphi-v\|_{L^2(Q_d(x_0))} + \|\nabla v\|_{L^2(Q_d(x_0))}\Big].
		\]
		Without loss of generality, we may assume that $x_0=0$. Let $v \in W^{1,2}(Q_d)$, by Minkowski and Jensen inequalities, for every $0 < t \le d$ we have 
		\begin{equation*}
			\begin{split}
				t\,\|\varphi-{\rm av}(\varphi; Q_d)\|_{L^2(Q_d)} &\le t\,\left(\|\varphi-v\|_{L^2(Q_d)} + \|v-{\rm av}(v; Q_d)\|_{L^2(Q_d)}\right)\\ 
				&+t\, \|{\rm av}(v; Q_d) - {\rm av}(\varphi; Q_d)\|_{L^2(Q_d)}\\
				&\le 2\,t\,\|\varphi-v\|_{L^2(Q_d)} + t\,\|v-{\rm av}(v;Q_d)\|_{L^2(Q_d)}\\
				&\le 2\,d\,\max\left\{\frac{1}{\sqrt{\mu_N}},1\right\}\,\left(\|\varphi-v\|_{L^2(Q_d)} + t\,\|\nabla v\|_{L^2(Q_d)}\right)
			\end{split}
		\end{equation*}
where in the last line we used the classical Poincar\'e-Wirtinger inequality. In particular, by arbitrariness of $v\in W^{1,2}(Q_d)$ we get 
		\begin{equation*}
			t^2\,\|\varphi-{\rm av}(u; Q_d)\|^2_{L^2(Q_d)} \le a_N\,d^2\,K(t, \varphi)^2, \qquad \text{for}\ 0 < t \le d,
		\end{equation*}
		where $a_N$ is a dimensional constant.
		Thus, we can infer
		\begin{equation*}
			\begin{split}
				\|\varphi-{\rm av}(\varphi; Q_d)\|^2_{L^2(Q_d)}\, \frac{d^{2\,(1-\sigma)}}{2\,(1-\sigma)} &= \|\varphi-{\rm av}(\varphi; Q_d)\|^2_{L^2(Q_d)}\, \int_0^{d}t^{1-2\,\sigma}\, dt \\
				&\le a_N\,{d}^2\,\int_{0}^{d} \left(\frac{K(t, \varphi)}{t^\sigma}\right)^2 \frac{dt}{t},
			\end{split}
		\end{equation*}
	which in turn implies
		\begin{equation} \label{1}
				 \|\varphi-{\rm av}(\varphi; Q_d)\|^2_{L^2(Q_d)}\le 2\,a_N\,(1-\sigma)\,d^{2\,\sigma}\,\int_{0}^{d} \left(\frac{K(t, \varphi)}{t^\sigma}\right)^2 \frac{dt}{t}. 
		\end{equation}
If $D\le 2\,d$, by joining the previous estimate and Proposition \ref{prop:Kfunz} we get the claimed inequality, with constant
\[
\mathcal{W}=\frac{2\,a_N\,A_N}{\sigma},
\]		
where $A_N>0$ is the constant in \eqref{claim:k-funz}.
\par
If $D>2\,d$, with a similar reasoning as above, for every $d<t\le D/2$ we have 
		\begin{equation*}
			\|\varphi-{\rm av}(\varphi; Q_d)\|^2_{L^2(Q_d)} \le a_N\,K(t, \varphi)^2,
		\end{equation*}
with the same dimensional constant $a_N>0$.
This yields 
		\begin{equation*} \label{2}
			\begin{split}
				\frac{\|\varphi-{\rm av}(\varphi; Q_d)\|^2_{L^2(Q_d)}}{2\,\sigma\,d^{2\sigma}} \left(1 - \left(\frac{2\,d}{D}\right)^{2\,\sigma}\right) &= \|\varphi-{\rm av}(\varphi; Q_d)\|^2_{L^2(Q_d)} \int_{d}^{D/2} \frac{dt}{t^{1+2\,\sigma}} \\
				&\le a_N\, \int_{d}^{D/2} \left(\frac{K(t, \varphi)}{t^\sigma}\right)^2 \frac{dt}{t}. 
			\end{split}
		\end{equation*}
		We rewrite it as follows
\begin{equation} \label{2}
\|\varphi-{\rm av}(\varphi; Q_d)\|^2_{L^2(Q_d)}\, \left(1 - \left(\frac{2\,d}{D}\right)^{2\,\sigma}\right)\le 2\,a_N\,\sigma\,d^{2\sigma}\,\int_{d}^{D/2} \left(\frac{K(t, \varphi)}{t^\sigma}\right)^2 \frac{dt}{t}.
\end{equation}
We now multiply \eqref{1} by $\sigma$, multiply \eqref{2} by $1-\sigma$ and then we sum up. We thus get
\[
\begin{split}
\left[\sigma+(1-\sigma)\, \left(1 - \left(\frac{2\,d}{D}\right)^{2\,\sigma}\right)\right]&\,\|\varphi-{\rm av}(\varphi; Q_d)\|^2_{L^2(Q_d)}\\
 &\le 4\,a_N\,\sigma\,(1-\sigma)\,d^{2\,\sigma}\,\int_{0}^{D/2} \left(\frac{K(t, \varphi)}{t^\sigma}\right)^2 \frac{dt}{t}.
\end{split}
\]
We then get the claimed estimate from Proposition \ref{prop:Kfunz}, this time with constant
\[
\mathcal{W}=4\,a_N\,A_N\,\left[\sigma+(1-\sigma)\, \left(1 - \left(\frac{2\,d}{D}\right)^{2\,\sigma}\right)\right]^{-1}.
\]
This concludes the proof.
	\end{proof}
\begin{remark}
In the particular case $D=d$, by observing that 
\[
\mathcal{W} = \mathcal{W}\left(N,\sigma,1\right) =\frac{W_N}{\sigma},
\]
the previous inequality boils down to
\[
\|\varphi-{\rm av}(\varphi; Q_d(x_0))\|^2_{L^2(Q_d(x_0))} \le W_N\,(1-\sigma)\,d^{2\,\sigma}\,\iint_{Q_d(x_0) \times Q_d(x_0)} \frac{|\varphi(x)-\varphi(y)|^2}{|x-y|^{N+2\,\sigma}}\,dx\,dy,
\]
which coincides with \cite[Lemma 3.5]{BiaBra24}.
\end{remark}	
On the other hand, by letting $D$ diverge to $+\infty$ in Proposition \ref{lm:pw_frac} and observing that
\[
\lim_{D\to+\infty} \mathcal{W}=W_N,
\] 
one can easily get the following result, which is a particular case of \cite[Lemma 4.2]{BozTal}, with cubes in place of balls.	
\begin{corollary}
		Let $0 < \sigma < 1$, $d>0$ and $x_0\in\mathbb{R}^N$. For every $\varphi \in C^\infty_0(\mathbb{R}^N)$ we have
	\[
			\begin{split}
				\|\varphi-{\rm av}(\varphi; Q_d(x_0))\|^2_{L^2(Q_d(x_0))} \le W_N\,\sigma\,(1-\sigma)\,d^{2\,\sigma}\,\iint_{Q_d(x_0) \times \mathbb{R}^N} \frac{|\varphi(x)-\varphi(y)|^2}{|x-y|^{N+2\,\sigma}}\,dx\,dy,
			\end{split}
\]
where $W_N>0$ is the same dimensional constant of Proposition \ref{lm:pw_frac}.	
\end{corollary}

\section{Capacities}
\label{sec:4}
\begin{definition}
	Let $E \subseteq \mathbb{R}^N$ be an open set and $1 < s \le 2$. For every compact set $\Sigma \subseteq E$, we define
	\begin{equation*} 
		\mathrm{Cap}_s(\Sigma; E) := \inf_{\varphi \in C^\infty_0(E)}\left\{[ \varphi ]^2_{W^{s, 2}(\mathbb{R}^N)} : \varphi\equiv 1\ \text{in a neighborhood of}\ \Sigma \right\}.
	\end{equation*}
Observe that in the limite case $s=2$, with the notation introduced in Section \ref{sec:2} we have
\[
[\varphi]^2_{W^{2,2}(\mathbb{R}^N)}=\int_{\mathbb{R}^N} |\Delta \varphi|^2\,dx=\int_E |\Delta \varphi|^2\,dx,\qquad \text{for every}\ \varphi\in C^\infty_0(E).
\]
Thus, in this case we obtain the classical relative capacity of order $2$, as in \cite[Chapter 13]{Maz}.	
	\end{definition}
It is not difficult to see that the following monotonicity relation holds
\begin{equation}
\label{monot1}
\mathrm{Cap}_s(\Sigma; E')\ge \mathrm{Cap}_s(\Sigma; E),\qquad \text{if}\ \Sigma\subseteq E'\subseteq E,
\end{equation}
and
\begin{equation}
\label{monot2}
\mathrm{Cap}_s(\Sigma; E)\ge \mathrm{Cap}_s(\Sigma'; E),\qquad \text{if}\ \Sigma'\subseteq \Sigma\subseteq E.
\end{equation}
We start with a simple comparison result among different capacities.
	\begin{proposition} \label{prop:capVScap}
		Let $1 < s < 2$ and let $E \subseteq \mathbb{R}^N$ be an open set. For every $\Sigma \subseteq E$ compact set, we have
		\begin{equation} \label{ineq:cap1}
			\Big(\lambda(E)\Big)^{2-s}\,(s-1)\,(2-s)\,\mathrm{Cap}_s\left(\Sigma; E\right) \le C_N\,\mathrm{Cap}_2(\Sigma; E),
		\end{equation}
		where $C_N>0$ depends on the dimension $N$, only.
	\end{proposition}
	\begin{proof}
		Let $\varphi \in C^\infty_0(E)$ be such that $\varphi \equiv 1$ in a neighborhood of $\Sigma$. In order to prove the first claimed estimate, we can assume that $\lambda(E) > 0$. By using the interpolation inequality of \cite[Corollary 2.2]{BPS} applied to each $\varphi_{x_i}$, we obtain 
		\begin{equation*}
			\begin{split}
				[\varphi]^2_{W^{s,2}(\mathbb{R}^N)} = \sum_{i=1}^N [\varphi_{x_i}]^2_{W^{s-1, 2}(\mathbb{R}^N)}
				&\le \frac{C_N}{(s-1)\,(2-s)} \sum_{i=1}^N \|\varphi_{x_i}\|^{2\,(2-s)}_{L^2(\mathbb{R}^N)}\, \|\nabla \varphi_{x_i}\|^{2\,(s-1)}_{L^2(\mathbb{R}^N)}\\
				&\le \frac{C_N}{(s-1)\,(2-s)}\,\lambda(E)^{s-2}\,\sum_{i=1}^N \|\nabla\varphi_{x_i}\|^2_{L^2(\mathbb{R}^N)}\\
				&= \frac{C_N}{(s-1)\,(2-s)}\,\lambda(E)^{s-2}\,\|\Delta \varphi\|^2_{L^2(E)}.
			\end{split}
		\end{equation*}
In the second inequality we used the definition of $\lambda(E)$,	while in the last equality we used again \eqref{HessLap}. By the arbitrariness of $\varphi$, we get \eqref{ineq:cap1}. 
	\end{proof}
	Points have zero $s-$capacity for $1<s\le N/2$. This is the content of the next result.
\begin{proposition}
\label{prop:capacitapunto}
Let $N\ge 3$. For every $1<s\le \min\{2,\,N/2\}$, $x_0\in\mathbb{R}^N$ and $r>0$ we have
\[
\mathrm{Cap}_s(\{x_0\};Q_r(x_0))=0.
\]
\end{proposition}
	\begin{proof}
		Without loss of generality, we can assume $x_0=0$ and $r=1$. Let $\varphi\in C^\infty_0(Q_1)$ be such that 
		\[
			0 \le \varphi \le 1,\qquad \varphi\equiv 1\ \text{on}\ B_\frac{1}{2}, \qquad |\nabla \varphi|^2+|\Delta \varphi|\le C,
		\]
			then we set $\varphi_k(x)=\varphi(k\,x)$. Thus, we have $\varphi_k\in C^\infty_0(Q_{1/k})$ and $\varphi\equiv 1$ on $B_{1/(2k)}$. Moreover, by a change of variable we easily get
			\begin{equation}
			\label{riscalo}
			[\varphi_k]^2_{W^{s,2}(\mathbb{R}^N)}=k^{2\,s-N}\,[\varphi]^2_{W^{s,2}(\mathbb{R}^N)}.
			\end{equation}
Below, we discuss separately the cases $1<s<N/2$ and $s=N/2$. Observe that, under our assumptions, the latter can only occur if $N=3$ (and consequently $s=3/2$) or $N=4$ (and consequently $s=2$).
		\vskip.2cm\noindent
		{\bf Case $1<s<N/2$.} This case is simpler. Each $\varphi_k$ is admissible in the definition of capacity of $\{0\}$. Thus, by \eqref{riscalo} we directly have
	\[
	\mathrm{Cap}_s(\{0\};Q_1)\le		\lim_{k\to+\infty}[\varphi_k]_{W^{s,2}(\mathbb{R}^N)}=0.
	\]
		\noindent 
		{\bf Borderline case $N=3$ and $s=3/2$.} 
		In this case, by \eqref{riscalo} the sequence $\{\varphi_k\}_{k \in \mathbb{N}}$ is only bounded in $W^{s,2}(\mathbb{R}^3)$. Thus, we can only infer the weak convergence of $\{\varphi_k\}_{k \in \mathbb{N}}$ to $0$ and we cannot directly apply the previous argument. In order to upgrade the convergence, we can proceed in an elementary way as in the proof of \cite[Theorem 1.3]{BiaBra22}: we will use {\it Mazur's Lemma}, see for example \cite[Theorem B.3.3 \& Remark B.3.3]{Brasco_book}. We set
		\[
		F_k(x,y) :=\frac{\nabla\varphi_k(x)-\nabla\varphi_k(y)}{|x-y|^{2}}, \qquad \text{for}\ k \in \mathbb{N},
		\]
		and observe that
		\[
		\left\|F_k\right\|_{L^2(\mathbb{R}^3\times\mathbb{R}^3)}=[\nabla \varphi_k]_{W^{\frac{1}{2},2}(\mathbb{R}^3)}, \qquad \text{for}\ k \in \mathbb{N}.
		\]
		By construction,  
		\[
		F_k\rightharpoonup0 \quad \text{in}\ L^2(\mathbb{R}^3\times\mathbb{R}^3; \mathbb{R}^3), \qquad \text{as}\ k\nearrow+\infty.
		\]
		By Mazur's Lemma, for every $k\in\mathbb{N} $, there exist 
		\[
		\big\{\alpha_m(k)\big\}_{m=1}^k\subseteq[0,1],\qquad
		\sum_{m=1}^k\alpha_m(k)=1,
		\]
		such that the sequence of convex combinations
		\begin{equation*} \label{weak-conv}
			\widetilde{F}_k(x,y):=\sum_{m=1}^k\alpha_m(k)\,F_m(x,y),
		\end{equation*}
		strongly converges to the null function in $L^2(\mathbb{R}^3\times\mathbb{R}^3; \mathbb{R}^3)$, as $k$ diverges to $+\infty$. By defining
		\begin{equation*}
			\Phi_k(x) := \sum_{m=1}^k \alpha_m(k)\,\varphi_m(x), \quad \mbox{ for } x \in \Omega,\ k \in \mathbb{N},
		\end{equation*}
		we then obtain a sequence of smooth functions compactly supported in $\Omega$ and such that 
		\[
		\Phi_k(x) = \sum_{m=1}^k \alpha_m(k)\,\varphi_m(x)=\sum_{m=1}^k \alpha_m(k)=1,\qquad \text{for}\ x\in B_\frac{1}{k}. 
		\]
		By construction, we also have  
		\[
		\begin{split}
			\lim_{k \to \infty} [\Phi_k]_{W^{\frac{3}{2},2}(\mathbb{R}^N)}=\lim_{k\to\infty} [\nabla \Phi_k]_{W^{\frac{1}{2},2}(\mathbb{R}^N)} &=\lim_{k\to\infty}\left[\sum_{m=1}^k \alpha_m(k)\,\nabla \varphi_m\right]_{W^{\frac{1}{2},2}(\mathbb{R}^3)}\\
			&=\lim_{k\to\infty} \left\|\sum_{m=1}^k\alpha_m(k)\,F_m(x,y)\right\|_{L^2(\mathbb{R}^3\times\mathbb{R}^3)}= 0.
		\end{split}
		\]
We can now use the functions $\Phi_k$ to test the definition of capacity of $\{0\}$. This gives
\[
	\mathrm{Cap}_\frac{3}{2}(\{0\};Q_1)\le		\lim_{k\to+\infty}[\varphi_k]_{W^{\frac{N}{s},2}(\mathbb{R}^3)}=0.
\]
This concludes the proof, in this case.
	\vskip.2cm\noindent 
		{\bf Borderline case $N=4$ and $s=2$.} For every $\varepsilon>0$, by the monotonicity relations \eqref{monot1} and \eqref{monot2} we have 
		\[
		\mathrm{Cap}_2(\{0\};Q_1)\le \mathrm{Cap}_2(\overline{B_\varepsilon};Q_1)\le \mathrm{Cap}_2(\overline{B_\varepsilon};B_1)\le C\,\left(\log \frac{1}{\varepsilon}\right)^{-1}.
		\]
		The rightmost estimate follows from \cite[Proposition 13.1.2.2]{Maz}. By letting $\varepsilon$ go to $0$, we get the desired conclusion.
		\end{proof}
\begin{remark}
The previous result is also contained in \cite[Example 2.5]{AFN}, but the borderline case $s=N/2$ is not considered there.
\end{remark}				
The following simple result concerns the asymptotic behavior of the fractional capacity of a point, in the ``critical'' regime $s\searrow N/2$. Due to our restriction $1<s<2$, this may happen only for dimension $N=2$ and $N=3$.
\begin{lemma}
\label{lm:puntolimite}
Let $N\in\{2,3\}$, for every $x_0\in\mathbb{R}^N$ and $r>0$ we have:
\begin{itemize}
\item if $N=3$
\[
\lim_{s\searrow \frac{3}{2}} \mathrm{Cap}_s(\{x_0\};Q_r(x_0))=0;
\]
\item if $N=2$
\[
\lim_{s\searrow 1} (s-1)\,\mathrm{Cap}_s(\{x_0\};Q_r(x_0))=0.
\]
\end{itemize}
\end{lemma}
\begin{proof}
We can suppose for simplicity that $x_0=0$ and $r=1$.
We take $\varphi\in C^\infty_0(Q_1)$ which coincides with $1$ in a neighborhood of $\{0\}$. We thus get
\[
 \mathrm{Cap}_s(\{0\};Q_1)\le [\varphi]_{W^{s,2}(\mathbb{R}^N)}^2.
\]
In dimension $N=2$, by multiplying the previous inequality by $(s-1)$ and then using Proposition \ref{prop:BBM} we get
\[
\limsup_{s\searrow 1}  (s-1)\,\mathrm{Cap}_s(\{0\};Q_1)\le \lim_{s\searrow 1}(s-1)\, [\varphi]_{W^{s,2}(\mathbb{R}^N)}^2=N\,\omega_N\, \int_{Q_1} |\nabla \varphi|^2\,dx.
\]
By taking the infimum over $\varphi$, we then obtain
\[
\limsup_{s\searrow 1}  (s-1)\,\mathrm{Cap}_s(\{0\};Q_1)\le N\,\omega_N\,\mathrm{Cap}_1(\{0\};Q_1)=0.
\]
Observe that we used that points have zero relative capacity in $W^{1,2}_0$, for dimension $N\ge 2$ (see for example \cite[Section 2.2.4, page 148]{Maz}).
\par
The case $N=3$ is similar, we just have to observe that 
\[
\lim_{s\searrow \frac{3}{2}}[\varphi]_{W^{s,2}(\mathbb{R}^3)}^2= [\varphi]_{W^{\frac{3}{2},2}(\mathbb{R}^N)}^2,
\]
for every $\varphi\in C^\infty_0(\mathbb{R}^3)$. This just follows from the Dominated Convergence Theorem. We thus obtain this time
\[
\limsup_{s\searrow \frac{3}{2}}  \mathrm{Cap}_s(\{0\};Q_1)\le \lim_{s\searrow \frac{3}{2}} [\varphi]_{W^{s,2}(\mathbb{R}^N)}^2= [\varphi]_{W^{\frac{3}{2},2}(\mathbb{R}^N)}^2.
\]
By taking the infimum over $\varphi$, we now obtain
\[
\limsup_{s\searrow \frac{3}{2}}  \mathrm{Cap}_s(\{0\};Q_1)\le \mathrm{Cap}_\frac{3}{2}(\{0\};Q_1).
\]
By observing that the last capacity is zero, see Proposition \ref{prop:capacitapunto}, we conclude.
\end{proof}	
\begin{remark}
We notice that the case $N=2$ is a little bit delicate: when $s$ approaches $N/2=1$ we have a superposition of two different effects. On the one hand, the fractional capacity should tend to $0$, because we are approaching a situation where points are negligible; on the other hand, we are boosting the non-integrable (at infinity) singularity in the kernel $|x-y|^{N+2\,(s-1)}$ of the Gagliardo-Slobodecki\u{\i} seminorm. Thus, the renormalizing factor $(s-1)$ is needed (recall Proposition \ref{prop:BBM}).
\end{remark}
The next removability result would be a consequence of \cite[Proposition 2.6]{AFN}. We provide a different proof, which is more direct and does not rely on Caffarelli-Silvestre--type formulas. 
We also include the limit case $s=N/2$, which is not considered in \cite{AFN}.
\begin{proposition}[Removability of points]\label{prop: punti removibili nello spazio}
		Let $N\ge 3$ and let $1<s\le \min\{2,\,N/2\}$, with $s\not=2$. Let $\Omega\subseteq\mathbb{R}^N$ be an open set, for every $x_0\in\Omega$ we have 
\[
\widetilde{W}^{s,2}_0(\Omega)=\widetilde{W}^{s,2}_0(\Omega\setminus\{x_0\}).
\]
In particular, we have
\[
			\Lambda^s(\Omega)=\Lambda^s(\Omega\setminus \{x_0\}).
\]
\end{proposition}
	\begin{proof}
By definition, we clearly have $\widetilde{W}^{s,2}_0(\Omega\setminus\{x_0\})\subseteq \widetilde{W}^{s,2}_0(\Omega)$. In order to get the reverse inclusion, thanks to the density of smooth functions, it is sufficient to prove that for every $u\in C^\infty_0(\Omega)$ there exists a sequence $\{u_n\}_{n\in\mathbb{N}}\subseteq C^\infty_0(\Omega\setminus\{x_0\})$ such that
\begin{equation}
\label{nap}
\lim_{n\to\infty}\|u_n-u\|_{L^2(\Omega)}=\lim_{n\to\infty}[\nabla u_n-\nabla u]_{W^{s-1}(\mathbb{R}^N)}=0.
\end{equation}
Indeed, we recall that
\[
\varphi\mapsto \|\varphi\|_{L^2(\Omega)}+[\nabla\varphi]_{W^{s-1}(\mathbb{R}^N)},
\]
is an equivalent norm on $\widetilde{W}^{s,2}_0(\Omega)$.
\par
Since $\Omega$ is open, there exists $r>0$ such that $Q_r(x_0)\Subset \Omega$. By Proposition \ref{prop:capacitapunto}, there exists a sequence $\{\varphi_n\}_{n\in\mathbb{N}}\subseteq C^\infty_0(Q_r(x_0))$ such that each $\varphi_n$ coincides with $1$ in a neighborhood of $\{x_0\}$ and 
\begin{equation}
\label{plop}
\lim_{n\to\infty}[\varphi_n]_{W^{s,2}(\mathbb{R}^N)}=0.
\end{equation}
By using the definition of $\Lambda^s$ and Proposition \ref{prop:mezzapoin} on the bounded set $Q_r(x_0)$, we have 
\[
\Lambda^s(Q_r(x_0))\,\|\varphi_n\|^2_{L^2(\Omega)}\le [\varphi_n]_{W^{s,2}(\mathbb{R}^N)}^2,
\]
and
\[
\Big(\Lambda^s(Q_r(x_0))\Big)^\frac{s-1}{s}\,\|\nabla \varphi_n\|_{L^2(\Omega)}^2\le \frac{s+1}{s}\,\Big(2\,s\,A_{s-1}\Big)^\frac{1}{s}\,[\varphi_n]^2_{W^{s,2}(\mathbb{R}^N)},
\]
\[
\Big(A_{s-1}\,\Lambda^s(Q_r(x_0))\Big)^\frac{1}{s}\,[\varphi]^2_{W^{s-1,2}(\mathbb{R}^N)}\le \,\frac{s^{2-\frac{1}{s}}}{(s-1)^{2\,(1-\frac{1}{s})}}\,[\varphi]^2_{W^{s,2}(\mathbb{R}^N)}.
\]
Thus, we have 
\begin{equation}
\label{tuttazzero}
\lim_{n\to\infty}\|\varphi_n\|_{L^2(\Omega)}=\lim_{n\to\infty} [\varphi_n]_{W^{s-1,2}(\mathbb{R}^N)}=\lim_{n\to\infty} \|\nabla \varphi_n\|_{L^2(\Omega)}=0,
\end{equation}
as well.
We then define
\[
u_n=u\,(1-\varphi_n)\in C^\infty_0(\Omega\setminus\{x_0\}).
\]
We start by estimating the difference $u_n-u$ in $L^2$ norm: we have
\[
\int_\Omega |u_n-u|^2\,dx=\int_\Omega |\varphi_n|^2\,|u|^2\,dx\le \|u\|^2_{L^\infty(\Omega)}\,\int_\Omega |\nabla \varphi_n|^2\,dx.
\]
By recalling that $u\in C^\infty_0(\Omega)$ and relying on \eqref{tuttazzero}, we get that the previous difference goes to $0$, as $n$ diverges to $\infty$. This establishes the first limit in \eqref{nap}.
\par
We now estimate the seminorm of the difference $u_n-u$ as follows:
\begin{equation}
\label{prot!}
\begin{split}
[u_n-u]^2_{W^{s,2}(\mathbb{R}^N)}=[u\,\varphi_n]_{W^{s,2}(\mathbb{R}^N)}^2&=[u\,\nabla \varphi_n+\varphi_n\,\nabla u]^2_{W^{s-1,2}(\mathbb{R}^N)}\\
&\le 2\,[\varphi_n\,\nabla u]_{W^{s-1,2}(\mathbb{R}^N)}^2+2\,[u\,\nabla \varphi_n]^2_{W^{s-1,2}(\mathbb{R}^N)}.
\end{split}
\end{equation}
Through some lengthy, yet standard, manipulations, we can separately estimate both terms as follows:
\[
\begin{split}
[\varphi_n\,\nabla u]_{W^{s-1,2}(\mathbb{R}^N)}^2&\le 2\,\int_{\mathbb{R}^N} |\varphi_n(x)|^2\,\left(\int_{\mathbb{R}^N} \frac{|\nabla u(x)-\nabla u(y)|^2}{|x-y|^{N+2\,(s-1)}}\right)\,dx\\
&+2\,\|\nabla u\|_{L^\infty(\Omega)}^2\,[\varphi_n]_{W^{s-1,2}(\mathbb{R}^N)}^2\\
&\le \frac{C}{(2-s)\,(s-1)}\,\left(\int_\Omega |\varphi_n|^2\,dx\right)\,\|D^2 u\|_{L^\infty(\Omega)}^{2\,(s-1)}\,\|\nabla u\|_{L^\infty(\Omega)}^{2\,(2-s)}\\
&+2\,\|\nabla u\|_{L^\infty(\Omega)}^2\,[\varphi_n]_{W^{s-1,2}(\mathbb{R}^N)}^2.
\end{split}
\]
Observe that we used \cite[Lemma 2.6]{BiaBraZag} to estimate the norm of the incremental ratio of $\nabla u$. By recalling that $u\in C^\infty_0(\Omega)$ and using both \eqref{plop} and \eqref{tuttazzero}, we obtain
\[
\lim_{n\to\infty} [\varphi_n\,\nabla u]_{W^{s,2}(\mathbb{R}^N)}=0.
\]
The other term is quite similar. We have
\[
\begin{split}
[u\,\nabla \varphi_n]^2_{W^{s-1,2}(\mathbb{R}^N)}&\le 2\,\|u\|^2_{L^\infty(\Omega)}\,[\nabla \varphi_n]^2_{W^{s-1,2}(\mathbb{R}^N)}\\
&+2\,\int_{\mathbb{R}^N} |\nabla \varphi_n(x)|^2\,\left(\int_{\mathbb{R}^N} \frac{|u(x)-u(y)|^2}{|x-y|^{N+2\,(s-1)}}\,dy\right)\,dx\\
&\le 2\,\|u\|^2_{L^\infty(\Omega)}\,[\nabla \varphi_n]^2_{W^{s-1,2}(\mathbb{R}^N)}\\
&+\frac{C}{(2-s)\,(s-1)}\,\left(\int_\Omega |\nabla \varphi_n|^2\,dx\right)\,\|\nabla u\|_{L^\infty(\Omega)}^{2\,(s-1)}\,\|u\|_{L^\infty(\Omega)}^{2\,(2-s)}.
\end{split}
\]
We used again \cite[Lemma 2.6]{BiaBraZag}, this time applied directly to $u$. Still by \eqref{plop} and \eqref{tuttazzero}, we get
\[
\lim_{n\to\infty}[u\,\nabla \varphi_n]^2_{W^{s-1,2}(\mathbb{R}^N)}=0,
\]
as well. By using these informations in \eqref{prot!}, we get
\[
\lim_{n\to\infty} [u_n-u]_{W^{s,2}(\mathbb{R}^N)}=0.
\]
This concludes the proof of \eqref{nap}, thus proving the desired equality of the spaces.
\par
The equality of the Poincar\'e constants $\Lambda^s(\Omega\setminus\{x_0\})$ and $\Lambda^s(\Omega)$ now easily follows from the equality of the spaces. 
	\end{proof}

	\section{Some expedient higher order Morrey inequalities}
	\label{sec:5}
	
The main result of this section is a lower bound on the relevant capacity of a singleton, in the case $s>N/2$. This will be obtained as a consequence of a Morrey--type inequality, that we are now going to present.
\par
	For $0 < \alpha \le 1$, we introduce the {\it Zygmund seminorm}
	\[
	[\varphi]_{\dot{\mathcal{C}}^\alpha(\mathbb{R}^N)}:=\sup_{x\in\mathbb{R}^N}\sup_{h\in\mathbb{R}^N\setminus\{0\}}\frac{|\delta_h^2\varphi(x)|}{|h|^\alpha}, \qquad \text{for}\ \varphi \in C^\infty_0(\mathbb{R}^N),
	\]
	where we set 
	\[
	\delta^2_h\varphi(x):=\varphi(x)-2\,\varphi(x+h)+\varphi(x+2\,h).
	\]
We then have the following Morrey inequalities.	 As usual, we want to track the dependence on $s$ of the relevant constants.
	\begin{proposition}\label{prop: morrey ineq}
		Let $N\in\{2,3\}$ and $N/2<s\le2$.  There exists two constants $\theta_{N,s}>0$ and $C_N>0$ such that for every $\varphi\in C^\infty_0(\mathbb{R}^N)$ we have
		\begin{equation}\label{eq:Mozzy}
			[\varphi]^2_{\dot{\mathcal{C}}^{s-\frac N2}(\mathbb{R}^N)}\le \theta_{N,s}\,[\varphi]^2_{W^{s,2}(\mathbb{R}^N)},
		\end{equation}
		and
		\[
		\|\varphi\|^2_{L^\infty(\mathbb{R}^N)}\le C_N\,\left(\theta_{N,s}\,[\varphi]^2_{W^{s,2}(\mathbb{R}^N)}\right)^{\frac{N}{2\,s}}\,\left(\|\varphi\|^2_{L^2(\mathbb{R}^N)}\right)^{1-\frac{N}{2\,s}}.
		\]
		Moreover, the constant $\theta_{N,s}$ has the following asymptotic behaviors
	\begin{equation}\label{eq: asintotica cost}
		\theta_{N,s}\sim 2-s,\mbox{ as }s\nearrow 2,\qquad\mbox{and}\qquad \begin{cases}
			\theta_{2,s}\sim1,&\mbox{ as }s\searrow 1,\\&\\
		\theta_{3,s}\sim \dfrac{1}{2\,s-3},&\mbox{ as }s\searrow\dfrac 32.
		\end{cases}
	\end{equation}
	\end{proposition}
	
	\begin{proof} We split the proof by discussing separately the case $1<s<2$ and the case $s=2$. \vskip.2cm\noindent
		{\bf Case $1<s<2$.}
		In view of establishing \eqref{eq:Mozzy}, we first observe that 
		\[
		\begin{split}
		\delta^2_h\varphi(x)&= \frac{1}{(2\,\pi)^\frac{N}{2}}\,\int_{\mathbb{R}^N}\mathcal{F}[\varphi](\xi)\,\left[e^{\text{i}\,\langle x,\xi\rangle}-2\,e^{\text{i}\,\langle x+h,\xi\rangle}+e^{\text{i}\,\langle x+2\,h,\xi\rangle}\right]\,d\xi\\
		&=\frac{1}{(2\,\pi)^\frac{N}{2}}\,\int_{\mathbb{R}^N}\mathcal{F}[\varphi](\xi)\,e^{\text{i}\,\langle x,\xi\rangle}\,\left[1-2\,e^{\text{i}\,\langle h,\xi\rangle}+e^{\text{i}\,\langle 2\,h,\xi\rangle}\right]\,d\xi,
		\end{split}
		\]
		thanks to the inversion formula.
	Thus,	by H\"older's inequality, we have
	\[
				\begin{split}
				\left|\delta^2_h\varphi(x)\right|
				&\le \frac{1}{(2\,\pi)^\frac{N}{2}}\,\left(\int_{\mathbb{R}^N}|\xi|^{2\,s}\,\Big|\mathcal{F}[\varphi](\xi)\Big|^2\,d\xi\right)^{\frac12}\,\left(\int_{\mathbb{R}^N}\frac{\left|1-2\,e^{\text{i}\,\langle h,\xi\rangle}+e^{\text{i}\,\langle 2\,h,\xi\rangle}\right|^2}{|\xi|^{2\,s}}\,d\xi\right)^{\frac12},
			\end{split}
	\]
which holds		for every $x,\,h\in\mathbb{R}^N$. By \eqref{eq: plancherel s}, this yields
		\begin{equation}\label{eq: stima differenza}\begin{split}
				|\delta^2_h\varphi(x)|&\le	\left[ \varphi\right]_{W^{s,2}(\mathbb{R}^N)}\,\left(A_{s-1}\,\int_{\mathbb{R}^N}\frac{\left|1-2\,e^{\text{i}\,\langle h,\xi\rangle}+e^{\text{i}\,\langle 2\,h,\xi\rangle}\right|^2}{|\xi|^{2\,s}}\,d\xi\right)^\frac{1}{2},
		\end{split}\end{equation}
		where the constant $A_{s-1}$ is given by \eqref{eq: def A(s-1)}. In order to estimate the right-hand side of the previous inequality, for every $a > 0$ we write
		\[
		\begin{split}
			\int_{\mathbb{R}^N}\frac{\left|1-2\,e^{\text{i}\,\langle h,\xi\rangle}+e^{\text{i}\,\langle 2\,h,\xi\rangle}\right|^2}{|\xi|^{2\,s}}\,d\xi&=\int_{\{|\xi|\le a\}}\frac{\left|1-2\,e^{\text{i}\,\langle h,\xi\rangle}+e^{\text{i}\,\langle 2\,h,\xi\rangle}\right|^2}{|\xi|^{2\,s}}\,d\xi\\
			&+\int_{\{|\xi|>a\}}\frac{\left|1-2\,e^{\text{i}\,\langle h,\xi\rangle}+e^{\text{i}\,\langle 2\,h,\xi\rangle}\right|^2}{|\xi|^{2\,s}}\,d\xi\\
			&\le N\,\omega_N\left(|h|^4\int_0^a\rho^{N+3-2\,s}\,d\rho+16\,\int_a^{+\infty}\rho^{N-1-2\,s}\,d\rho\right) \\
			&= N\,\omega_N \left(\frac{|h|^4\,a^{N+4-2\,s}}{N+4-2\,s} + \frac{16\,a^{N-2\,s}}{2\,s-N}\right).
		\end{split}
		\]
		For the integral over $\{|\xi|\le a\}$ we used that 
		\[
		\left|1-2\,e^{\text{i}\,\langle h,\xi\rangle}+e^{\text{i}\,\langle 2\,h,\xi\rangle}\right|\le |h|^2\,|\xi|^2,
		\]
		while for the integral over $\{|\xi|>a\}$ we used that
		\[
		\left|1-2\,e^{\text{i}\,\langle h,\xi\rangle}+e^{\text{i}\,\langle 2\,h,\xi\rangle}\right|\le 4.
		\]
		By choosing $a=|h|^{-1}$, we infer that
		\[
		\int_{\mathbb{R}^N}\frac{\left|1-2\,e^{\text{i}\,\langle h,\xi\rangle}+e^{\text{i}\,\langle 2\,h,\xi\rangle}\right|^2}{|\xi|^{2\,s}}\,d\xi\le  \frac{c_N}{2\,s-N}\,|h|^{2\,s-N}.
		\]
		By \eqref{eq: stima differenza}, this yields  
		\begin{equation*}
		\frac{|\delta_h^2\varphi(x)|^2}{|h|^{2\,s-N}}\le \theta_{N,s}\,[\varphi]^2_{W^{s,2}(\mathbb{R}^N)}, \qquad \text{with}\ \theta_{N,s}:=\frac{c_N\,A_{s-1}}{2\,s-N},
		\end{equation*}
		for every $x, h \in \mathbb{R}^N$ with $h \neq 0$, and so \eqref{eq:Mozzy} follows. The claimed asymptotic behavior of the constant $\theta_{N,s}$ now follows from \eqref{asym-Aesse}.
		\vskip.2cm\noindent
		We still need to prove the interpolation inequality for the $L^\infty$ norm. By \eqref{eq:Mozzy} and the triangle inequality, we have
		\[
		\begin{split}
		|\varphi(x)|&\le |\varphi(x)-2\,\varphi(x+h)+\varphi(x+2\,h)|+2\,|\varphi(x+h)|+|\varphi(x+2\,h)|\\
		&\le |h|^{s-\frac{N}{2}}\,\sqrt{\theta_{N,s}}\,[\varphi]_{W^{s,2}(\mathbb{R}^N)}+2\,|\varphi(x+h)|+|\varphi(x+2\,h)|,
		\end{split}\]
		for every $x,h\in \mathbb{R}^N$. We fix $\ell>0$ and we integrate the previous inequality with respect to $h\in B_\ell$. We then obtain
		\[
		\omega_N\,\ell^N\,|\varphi(x)|\le \frac{N\,\omega_N}{\dfrac{N}{2}+s}\,\ell^{\frac{N}{2}+s}\,\sqrt{\theta_{N,s}}\,[\varphi]_{W^{s,2}(\mathbb{R}^N)}+2\,\int_{B_\ell(x)} |\varphi(y)|\,dy+\frac{1}{2^N}\,\int_{B_{2\ell}(x)} |\varphi(y)|\,dy.
		\]
	By using simple manipulations and H\"older's inequality, we can in turn infer
	\[
	|\varphi(x)|\le C\,\sqrt{\theta_{N,s}}\,\ell^{s-\frac{N}{2}}\,[\varphi]_{W^{s,2}(\mathbb{R}^N)}+C\,\frac{1}{\ell^\frac{N}{2}}\,\|\varphi\|_{L^2(\mathbb{R}^N)},
	\]	
	for a constant $C=C(N)>0$. The previous estimate holds for every $x\in\mathbb{R}^N$ and every $\ell>0$. If we choose 
	\[
	\ell=\left(\frac{\|\varphi\|_{L^2(\mathbb{R}^N)}}{\sqrt{\theta_{N,s}}\,[\varphi]_{W^{s,2}(\mathbb{R}^N)}}\right)^\frac{1}{s},
	\]
	we can then obtain the desired inequality.
		\vskip.2cm\noindent
		{\bf Case $s=2$.} We could repeat verbatim the previous arguments. Otherwise, we can multiply the previous inequalities valid for $0<s<2$ by the factor $(2-s)$ and take the limit as $s \nearrow 2$. By Proposition \ref{prop:BBM} and by \eqref{eq: asintotica cost}, we get the claimed estimates for a positive constant $\theta_N$ depending only on $N$.
	\end{proof}
\begin{remark}
As a straightforward consequence of the previous Morrey-type estimates, we get that for every open set $\Omega\subseteq\mathbb{R}^N$ the following continuous embedding holds
\[
\widetilde{W}^{s,2}_0(\Omega)\hookrightarrow \mathcal{C}^{s-\frac{N}{2}}_0(\overline\Omega),
\]
for $N\in\{2,3\}$ and $N/2<s\le 2$. Here, we indicated by $\mathcal{C}^{s-N/2}_0(\overline\Omega)$ the following non-homogeneous Zygmund space
\[
\mathcal{C}^{s-\frac{N}{2}}_0(\overline\Omega)=\Big\{\varphi\in C_{\rm bound}(\overline\Omega)\, :\, [\varphi]_{\dot{\mathcal{C}}^{s-\frac{N}{2}}(\mathbb{R}^N)}<+\infty,\ \varphi=0\ \text{on}\ \partial\Omega\Big\},
\]
and $ C_{\rm bound}(\overline\Omega)$ is the set of continuous and bounded functions on $\overline\Omega$. The latter is a Banach space, when endowed with the norm
\[
\|\varphi\|_{\mathcal{C}^{s-\frac{N}{2}}(\overline\Omega)}:=\|\varphi\|_{L^\infty(\Omega)}+[\varphi]_{\dot{\mathcal{C}}^{s-\frac{N}{2}}(\mathbb{R}^N)},\qquad \mbox{for every}\ \varphi\in \mathcal{C}^{s-\frac{N}{2}}_0(\overline\Omega).
\]
We refer to \cite{Tri} for more details on the Zygmund space.
\end{remark}	
	
As a straightforward consequence of the previous result, we have the following lower bound for the relative capacity of a point.	
	\begin{corollary}\label{lm: capacità punto positiva}
		Let $N\in\{2,3\}$ and $N/2<s\le 2$. Let $E\subseteq\mathbb{R}^N$ be an open bounded set. 
For every $x_0\in E$, we have 
		\begin{equation} \label{cap_punto}
			\frac{{\mathrm{diam}(E)^{N-2\,s}}}{\theta_{N,s}} \le \mathrm{Cap}_s(\{x_0\};E),
		\end{equation}
		where $\theta_{N,s}$ is the same constant as in \eqref{eq:Mozzy}
	\end{corollary}
	\begin{proof}
		For every $\varphi\in C^\infty_0(E)$ such that $\varphi(x_0) \equiv 1$ in a neighborhood of $x_0$, let us take $h\in\mathbb{R}^N$ such that $|h|>\mathrm{diam}(E)$. In particular, we get that $x_0+h\not\in E$ and $x_0+2\,h\not\in E$. From \eqref{eq:Mozzy}, we thus get
		\[
		1\le|\varphi(x_0)|^2\le |h|^{2\,s-N}\,\theta_{N,s}\,[\varphi]^2_{W^{s,2}(\mathbb{R}^N)}.
		\]
		By the arbitrariness of $\varphi$ and $|h|>\mathrm{diam}(E)$, we get \eqref{cap_punto}.
	\end{proof}

\section{Maz'ya-Poincar\'e inequalities}
\label{sec:6}	
\subsection{The case $s=2$}
In this case, the following simple Poincar\'e inequality for functions vanishing in a neighborhood of a compact set will be the key tool to prove the main theorem. The estimate is a bit rough, but it will be largely sufficient. This is a particular case of \cite[Theorem 14.1.2]{Maz}.
\begin{lemma}
\label{lm:lemmasega}
Let $\Sigma\subseteq\mathbb{R}^N$ be a compact set. Let $Q_d(y_0)\subseteq\mathbb{R}^N$ be a cube containing $\Sigma$ and set
\[
\delta:=\mathrm{dist}(\Sigma,\partial Q_d(y_0)).
\] 
Then for every $u\in C^\infty(\overline{Q_d(y_0)})$ such that $u\equiv0$ in a neighborhood of $\Sigma$, we have 
\[
\frac{\mathrm{Cap}_2(\Sigma;Q_d(y_0))}{d^N}\,\int_{Q_d(y_0)}|u|^2\,dx\le C\,\int_{Q_d(y_0)} |\Delta u|^2\,dx+\frac{C}{\delta^2}\,\left(\frac{d}{\delta}\right)^2\,\int_{Q_d(y_0)} |\nabla u|^2\,dx,
\]
where $C=C(N)>0$. 
\end{lemma}
\begin{proof}
Without loss of generality, we can assume that 
\begin{equation}
\label{normale}
\int_{Q_{d}(y_0)} |u|^2\,dx=|Q_d(y_0)|\qquad \mbox{and}\qquad \mathrm{av}(u;Q_d(y_0))\ge 0.
\end{equation}
Let $\eta\in C^\infty_0(Q_{d-\delta/4}(y_0))$ be a cut-off function such that
\[
0\le \eta\le 1,\qquad \eta\equiv 1 \quad \text{on}\ Q_{d-\frac{\delta}{2}}(y_0),\qquad |\nabla \eta|^2+|\Delta \eta|\le \frac{C}{\delta^2},
\]
for some universal constant $C>0$. Then we can use $\eta\,(1-u)$ as a test function for the capacity of $\Sigma$ relative to $Q_d(y_0)$. We get
\[
\begin{split}
\mathrm{Cap}_2(\Sigma;Q_d(y_0))&\le \int_{Q_d(y_0)} |\Delta\,(\eta\,(1-u))|^2\,dx\\
&=\int_{Q_d(y_0)} |(1-u)\,\Delta \eta-2\,\langle \nabla \eta,\nabla u\rangle-\eta\,\Delta u|^2\,dx\\
&\le 3\,\int_{Q_d(y_0)} |1-u|^2\,|\Delta \eta|^2\,dx+6\,\int_{Q_d(y_0)} |\nabla \eta|^2\,|\nabla u|^2\,dx\\
&+3\,\int_{Q_d(y_0)} \eta^2\,|\Delta u|^2\,dx.
\end{split}
\]
In particular, by using the properties of the cut-off function $\eta$, we obtain
\begin{equation}
\label{maurizio}
\mathrm{Cap}_2(\Sigma;Q_d(y_0))\le \frac{C}{\delta^4}\,\int_{Q_{d}(y_0)} |1-u|^2\,dx+\frac{C}{\delta^2}\,\int_{Q_d(y_0)} |\nabla u|^2\,dx+C\,\int_{Q_d(y_0)} |\Delta u|^2\,dx.
\end{equation}
In order to conclude, we just need to control the $L^2$ norm of $1-u$. We proceed as follows
\begin{equation}
\label{pap}
\begin{split}
\int_{Q_{d}(y_0)} |1-u|^2\,dx&\le 2\,\int_{Q_{d}(y_0)} |\mathrm{av}(u;Q_d(y_0))-u|^2\,dx+2\,\int_{Q_{d}(y_0)} |1-\mathrm{av}(u;Q_d(y_0))|^2\,dx\\
&\le 2\,\int_{Q_{d}(y_0)} |\mathrm{av}(u;Q_d(y_0))-u|^2\,dx+2\,|Q_{d}(y_0)|\, |1-\mathrm{av}(u;Q_d(y_0))|^2.\\
\end{split}
\end{equation}
Then we observe that 
\[
\begin{split}
\sqrt{|Q_{d}(y_0)|}\,|1-\mathrm{av}(u;Q_d(y_0))|=\Big|\|u\|_{L^2(Q_d(y_0))}-\|\mathrm{av}(u;Q_d(y_0))\|_{L^2(Q_d(y_0))}\Big|,
\end{split}
\]
thanks to the normalization condition \eqref{normale}. By using the triangle inequality, we thus get
\[
\sqrt{|Q_{d}(y_0)|}\,|1-\mathrm{av}(u;Q_d(y_0))|\le \|u-\mathrm{av}(u;Q_d(y_0))\|_{L^2(Q_d(y_0))}.
\]
We can use this estimate in \eqref{pap} and obtain
\[
\int_{Q_{d}(y_0)} |1-u|^2\,dx\le 4\,\|u-\mathrm{av}(u;Q_d(y_0))\|^2_{L^2(Q_d(y_0))}.
\]
A further application of Poincar\'e-Wirtinger inequality gives
\[
\int_{Q_{d}(y_0)} |1-u|^2\,dx\le \frac{4}{\mu_N}\,d^2\,\int_{Q_d(y_0)} |\nabla u|^2\,dx.
\]
By finally inserting this estimate in \eqref{maurizio} and recalling the normalization condition \eqref{normale}, we eventually get the desired inequality.
\end{proof}

\subsection{The case $1<s< 2$}
This case will be much more delicate. We start with the following auxiliary result: this is nothing but a fractional variant of (a particular case of) \cite[Theorem 14.1.3]{Maz}. We give all the details, in order to carefully keep track of the quality of the relevant constants.
	\begin{lemma} \label{lm:pag673_frac}
	Let $1 < s < 2$ and let $\Sigma\subseteq\mathbb{R}^N$ be a compact set. Let $Q_d(y_0)\subseteq\mathbb{R}^N$ be a cube containing $\Sigma$ and set
\[
\delta:=\mathrm{dist}(\Sigma,\partial Q_d(y_0)).
\] 
Let	 $u\in C^\infty(\overline{Q_d(y_0)})$ be such that 
	\begin{equation} \label{hp1_frac}
		{\rm av}(u; Q_d(y_0)) \ge 0, \qquad \max_{x\in\Sigma}u(x) \le \varepsilon\,\frac{\|u\|_{L^2(Q_d(y_0))}}{\sqrt{|Q_d(y_0)|}},
	\end{equation}
	for some  $\varepsilon \in (0, 1)$. Then, the following Poincar\'e--type inequality holds
	\begin{equation*}
		\begin{split}
			(1-\varepsilon)^2\,\frac{\mathrm{Cap}_s\left(\Sigma; Q_d(y_0)\right)}{d^N}\,\int_{Q_d(y_0)}|u|^2\,dx &\le C\,[ u]^2_{W^{s, 2}(Q_d(y_0))}\\& + C\,\left(\frac{d}{\delta}\right)^2\,\frac{1}{(s-1)\,(2-s)}\,\frac{1}{\delta^{2\,(s-1)}}\,\int_{Q_d(y_0)}|\nabla u|^2\,dx,
		\end{split}
	\end{equation*}
for a constant $C=C(N)>0$.
\end{lemma}
\begin{proof}
	Without loss of generality, we can assume that $y_0 = 0$ and 
	\begin{equation} \label{hp0_frac}
		\|u\|^2_{L^2(Q_d)} =|Q_d|=(2\,d)^N.
	\end{equation}
 Let $\psi \in C^\infty_0(\mathbb{R}^N)$ be such that $0 \le \psi \le 1$ and 
\begin{equation} \label{cut_off_frac}
	\psi=1\ \text{on}\ Q_{d-{\frac{\delta}{2}}}, \qquad \psi = 0\ \text{on}\ \mathbb{R}^N \setminus Q_{d-\frac{\delta}{4}}, \qquad |\nabla \psi|^2+|D^2\psi|\le \frac{a_N^2}{\delta^2},
\end{equation}
	Let $\varepsilon \in (0, 1)$ be as in the statement, we set $V:= (u-\varepsilon)/(1-\varepsilon)$. Then, $V \in C^\infty(\overline{Q_d})$ and $V \le 0$ on $\Sigma$, thanks to \eqref{hp1_frac} (recall the normalization \eqref{hp0_frac}). This entails that the function $(1-V)\,\psi$ is admissible in the definition of $\mathrm{Cap}_s(\Sigma; Q_d). $ By the triangle inequality, we get 
	\begin{equation} \label{cap_prima_stima}
		\begin{split}
			\mathrm{Cap}_s(\Sigma; Q_d)&\le 2\,\big[(1-V)\,\nabla \psi\big]^2_{W^{s-1,2}(\mathbb{R}^N)} + 2\,\big[\psi\,\nabla (1-V)\big]^2_{W^{s-1,2}(\mathbb{R}^N)} =:2\,\mathcal{I}_1+2\,\mathcal{I}_2,   
		\end{split}
	\end{equation}  and we estimate each term separately. For $\mathcal{I}_1$, we add and subtract $(1-V(y))\,\nabla \psi(x)$ and split the seminorm, thus we get
\[		\begin{split}
			\mathcal{I}_1&\le \iint_{Q_d \times Q_d} \frac{|(V(y) - V(x))\,\nabla \psi(x) + (1-V(y))\,(\nabla \psi(x) - \nabla \psi(y))|^2}{|x-y|^{N+2\,(s-1)}}\,dx\,dy+\\
			&+ 2\,\iint_{\left(\mathbb{R}^N\setminus Q_d\right) \times Q_{d-\frac{\delta}{4}}} \frac{| (1-V(y))\,\nabla \psi(y)|^2}{|x-y|^{N+2\,(s-1)}}\,dx\,dy
			\\
			&=: \mathcal{I}_{1,1} + \mathcal{I}_{1,2}.\\ 
		\end{split}
\]	
We also used that $(1-V)\,\nabla \psi$ identically vanishes on $Q_{d}\setminus Q_{d-\delta/4}$, by the properties of $\psi$.
	By the elementary inequality $(a+b)^2\le 2\,(a^2+b^2)$ and \cite[Lemma 2.6]{BiaBraZag}, we have 
	\begin{equation*} 
		\begin{split}
			\mathcal{I}_{1,1} &\le 2\,\|\nabla \psi\|^2_{L^\infty(\mathbb{R}^N)}\,[V]^2_{W^{s-1, 2}(Q_d)}\\
			& + \frac{C}{(s-1)\,(2-s)}\,\|1-V\|^2_{L^2(Q_d)}\,\|{\rm Hess}(\psi) \|^{2\,(s-1)}_{L^\infty(\mathbb{R}^N)}\,\|\nabla \psi\|^{2\,(2-s)}_{L^\infty(\mathbb{R}^N)} \\
			&\le \frac{C}{\delta^2\,(1-\varepsilon)^2}\,[u]^2_{W^{s-1, 2}(Q_d)} +\frac{C}{(s-1)\,(2-s)}\,\frac{1}{\delta^{2s}}\,\|1-V\|^2_{L^2(Q_d)},
		\end{split}
	\end{equation*}
	for a dimensional constant $C>0$, where in the last line we used  the properties of $\psi$ and the definition of $V$. 
\par	
To estimate $\mathcal{I}_{1,2}$, we observe that
\[
\mathbb{R}^N\setminus Q_d\subseteq \mathbb{R}^N\setminus B_{\delta/4}(y),\qquad \text{for every}\ y\in Q_{d-\frac{\delta}{4}}.
\]	
By using spherical coordinates,	this yields
	\begin{equation} \label{conto_utile}
		\begin{split}
			\int_{\mathbb{R}^N \setminus Q_d} \frac{dx}{|x-y|^{N+2\,(s-1)}} &\le \int_{\mathbb{R}^N \setminus B_{\delta/4}(y)} \frac{dx}{|x-y|^{N+2\,(s-1)}}\\
		&=		 \int_{\mathbb{S}^{N-1}}\left(\int_{\frac{\delta}{4}}^{\infty} \varrho^{-2\,(s-1)-1}\,d\varrho\right)\,d\mathcal{H}^{N-1}(\omega)\\
			&= \frac{N\,\omega_N}{2\,(s-1)}\,\left(\frac{4}{\delta}\right)^{2\,(s-1)},
		\end{split}
	\end{equation}
	for every $y \in Q_{d-\delta/4}$. By using this estimate and the properties of $\psi$, we then infer that 
\[
	\begin{split}
		\mathcal{I}_{1,2}&\le2\,\frac{N\,\omega_N}{2\,(s-1)}\,\left(\frac{4}{\delta}\right)^{2\,(s-1)}\iint_{Q_{d-\frac{\delta}{4}}} \frac{| (1-V(y))\,\nabla \psi(y)|^2}{|x-y|^{N+2\,(s-1)}}\,dy\\
		& \le \frac{C}{\delta^{s}}\,\left(\frac{1}{s-1}\right)^{\frac{1}{2}}\,\|1-V\|_{L^2(Q_d)},
	\end{split}
\]
	for a suitable constant $C=C(N) > 0$. By keeping the two estimates together, we get
\begin{equation}
\label{I1}
\mathcal{I}_1\le \mathcal{I}_{1,1}+\mathcal{I}_{1,2}\le\frac{C}{\delta^2\,(1-\varepsilon)^2}\,[u]^2_{W^{s-1, 2}(Q_d)} +\frac{C}{(s-1)\,(2-s)}\,\frac{1}{\delta^{2s}}\,\|1-V\|^2_{L^2(Q_d)},
\end{equation}
possibly for a different constant $C=C(N)>0$.
We still need to estimate the $L^2$ norm of
	\[
	1-V=\frac{1-u}{1-\varepsilon}.
	\] 
We proceed as follows
\[
		\begin{split}
			\|1- V\|_{L^2(Q_d)}&\le \frac{1}{1-\varepsilon}\left(\|1-\mathrm{av}( u;Q_d)\|_{L^2(Q_d)}
			+\|\mathrm{av}( u;Q_d)-u\|_{L^2(Q_d)}\right)\\
			&=\frac{1}{1-\varepsilon}\left(\sqrt{|Q_d|}\,|1-\mathrm{av}( u;Q_d)|+\|\mathrm{av}( u;Q_d)-u\|_{L^2(Q_d)}\right)\\
			&=\frac{1}{1-\varepsilon}\,\left(\Big|\|u\|_{L^2(Q_d)}-\|\mathrm{av}( u;Q_d)\|_{L^2(Q_d)}\Big|+\|\mathrm{av}( u;Q_d)-{u}\|_{L^2(Q_d)} \right)\\
			&\le \frac{2}{1-\varepsilon}\,\|\mathrm{av}( u;Q_d)-u\|_{L^2(Q_d)},
		\end{split}
\]
where in the second equality we used \eqref{hp1_frac} and  \eqref{hp0_frac}. The last $L^2$ norm can be finally estimated by the $L^2$ norm of $\nabla u$, thanks to the classical Poincar\'e-Wirtinger inequality.
By spending the resulting estimate in \eqref{I1}, we get
	\begin{equation} \label{stima_I1}
		\mathcal{I}_1 \le \frac{C}{(1-\varepsilon)^2}\,\left(\frac{[u]^2_{W^{s-1, 2}(Q_d)}}{\delta^2}+ \frac{1}{(s-1)\,(2-s)}\,\left(\frac{d}{\delta}\right)^2\,\frac{1}{\delta^{2\,(s-1)}}\,\|\nabla u\|^2_{L^2(Q_d)}\right),
	\end{equation}
	for a constant $C=C(N)> 0$. 
	\par 
Going back to \eqref{cap_prima_stima}, we still need to estimate $\mathcal{I}_2$. We proceed similarly as before. We add and subtract $\psi(y)\,\nabla V(x)$, split the double integral and use the properties of the cut-off $\psi$, so to infer that   
	\begin{equation} \label{I_2}
		\begin{split}
			\mathcal{I}_2 &= \iint_{Q_d \times Q_d} \frac{|(\psi(y)-\psi(x))\,\nabla V(x) + \psi(y)\,(\nabla V(y)-\nabla V(x))|^2}{|x-y|^{N+2\,(s-1)}}\,dx\,dy\\
			&+ 2\,\int_{Q_{d-\frac{\delta}{4}}} |\psi(y)|^2\,|\nabla V(y)|^2\left(\int_{\mathbb{R}^N\setminus Q_d} \frac{dx}{|x-y|^{N+2\,(s-1)}}\right)\,dy \\
			&=: \mathcal{I}_{2,1} + \mathcal{I}_{2,2}.
		\end{split} 
	\end{equation}
	Again by elementary inequalities and \cite[Lemma 2.6]{BiaBraZag}, we get
	\begin{equation*}
		\begin{split}
			\mathcal{I}_{2,1}&\le \frac{2}{(1-\varepsilon)^2}\left(\frac{C}{(s-1)\,(2-s)}\,\frac{1}{\delta^{2\,(s-1)}}\,\|\nabla u\|^2_{L^2(Q_d)}  + [\nabla u]^2_{W^{s-1, 2}(Q_d)}\right),
		\end{split}
	\end{equation*}
	where we used \eqref{cut_off_frac} and the fact that $\nabla V=\nabla u/(1-\varepsilon)$. For the ``strongly nonlocal'' term $\mathcal{I}_{2,2}$, we use again \eqref{conto_utile} and the fact that $|\psi|\le 1$, so to obtain
	\begin{equation*}
		\begin{split}
			\mathcal{I}_{2,2} \le \frac{2}{(1-\varepsilon)^2}\,\frac{C}{(s-1)\,(2-s)}\,\frac{1}{\delta^{2\,(s-1)}}\,\|\nabla u\|^2_{L^2(Q_d)}.
		\end{split}
	\end{equation*}
	By spending the previous estimates in \eqref{I_2}, we get 
	\begin{equation} \label{stima_I2}
		\mathcal{I}_2 \le  \frac{C}{(1-\varepsilon)^2}\,\left(\frac{1}{(s-1)\,(2-s)}\,\frac{1}{\delta^{2\,(s-1)}}\,\|\nabla u\|^2_{L^2(Q_d)}  + [\nabla u]^2_{W^{s-1, 2}(Q_d)}\right),
	\end{equation}
	for some $C=C(N)>0$.
We can now use \eqref{stima_I1} and \eqref{stima_I2} in \eqref{cap_prima_stima}. We also multiply both sides by $(1-\varepsilon)^2$. This yields 
	\begin{equation} \label{vai}
		\begin{split}
			(1-\varepsilon)^2\,\mathrm{Cap}_s\left(\Sigma; Q_d\right)& \le C\,\frac{[u]^2_{W^{s-1, 2}(Q_d)}}{\delta^2} \\&+C\,\left( \frac{d}{\delta}\right)^2\,\frac{1}{(s-1)\,(2-s)}\,\frac{1}{\delta^{2\,(s-1)}}\,\|\nabla u\|^2_{L^2(Q_d)}+C\, [\nabla u]^2_{W^{s-1, 2}(Q_d)},
		\end{split}
	\end{equation}
	for a suitable constant $C=C(N) > 0$. We also used that $d/\delta\ge 1$.
	\par
	In order to conclude, we need to (carefully) estimate the first term on the right-hand side, i.e. the fractional $(s-1)$ seminorm of $u$.
	Here, we can use Lemma \ref{lm:rubbato}. More precisely, we proceed as follows: we extend $u$ to the whole space, by means of the extension operator $\mathcal{E}_{Q_d}$ in \cite[Corollary 3.2]{BiaBra24}. For notational simplicity, we set $\widetilde{u}:= \mathcal{E}_{Q_d}(u)$. We recall that $\widetilde{u}$ is locally Lipschitz and it satisfies in particular the following estimate 
	\begin{equation} \label{ineq:estesa}
		[\widetilde{u}]^2_{W^{\sigma,2}(Q_D)} \le \alpha_N\,\left(\frac{D}{d}\right)^{4\,N}\, [u]^2_{W^{\sigma, 2}(Q_d)},
	\end{equation}
	for every $0 < \sigma \le 1$ and $D > d$. The dimensional constant $\alpha_N$  can be made explicit.
	Since $Q_d \subseteq B_{\sqrt{N}\,d} \subseteq Q_{\sqrt{N}\,d}$, by using Lemma \ref{lm:rubbato} for $\widetilde{u}$ in the ball $B_{\sqrt{N}\,d}$ and by using \eqref{ineq:estesa} with $D=\sqrt{N}\,d$ we obtain 
	\begin{equation*}
		\begin{split}
			\frac{1}{\delta^2}\,[u]^2_{W^{s-1, 2}(Q_d)} \le \frac{1}{\delta^2}\,[\widetilde{u}]^2_{W^{s-1, 2}(B_{\sqrt{N}\,d})} &\le \frac{1}{\delta^2}\,\frac{C_N}{2-s}\, (\sqrt{N}\,d)^{2\,(2-s)}\, \|\nabla \widetilde{u}\|^2_{L^2(Q_{\sqrt{N}\,d})}\\ &\le \left(\frac{d}{\delta}\right)^2\,\frac{C_N}{2-s}\,\alpha_N\,N^{2N+2}\,\frac{1}{d^{2\,(s-1)}}\,\|\nabla u\|^2_{L^2(Q_d)}.
		\end{split}
	\end{equation*}
	By spending this estimate in \eqref{vai} and using that 
	\[
	\frac{1}{2-s}\le \frac{1}{(s-1)\,(2-s)}\qquad \text{and}\qquad \frac{1}{\delta^{2\,(s-1)}}+\frac{1}{d^{2\,(s-1)}}\le \frac{2}{\delta^{2\,(s-1)}},
	\]
	we get 
	\begin{equation*}
		\begin{split}
		(1-\varepsilon)^2\,\mathrm{Cap}_s\left(\Sigma; Q_d\right)&\le C\,[\nabla u]^2_{W^{s-1, 2}(Q_d)} + C\,\left(\frac{d}{\delta}\right)^2\,\frac{1}{(s-1)(2-s)}\,\frac{1}{\delta^{2\,(s-1)}}\,\|\nabla u\|^2_{L^2(Q_d)},
		\end{split}
	\end{equation*}
for a constant $C> 0$, depending on the dimension only.
	This eventually yields the desired conclusion, by taking into account \eqref{hp0_frac} and possibly redefining the constant $C$.
	\end{proof}
We can now prove the Maz'ya-Poincar\'e inequality that will be needed to get the main result, in the case $1<s<2$. This is the fractional version of Lemma \ref{lm:lemmasega}: the crucial difference is the constant in front of the $L^2$ norm of the gradient, which now depends on the capacity of the Dirichlet region. This will be essential in order to get the sharp dependence on the parameter $s$, in the result of Theorem \ref{thm:frac}. For future reference, we prove it for a general compact set $\Sigma$, even if we will use it only for singletons.	
	\begin{theorem}[Refined fractional Maz'ya-Poincar\'e inequality] 
	\label{thm:mp_frac}
		Let $1 < s < 2$ and let $\Sigma\subseteq\mathbb{R}^N$ be a compact set. Let $Q_d(y_0)\subseteq\mathbb{R}^N$ be a cube containing $\Sigma$ and set
\[
\delta:=\mathrm{dist}(\Sigma,\partial Q_d(y_0)).
\] 
Let $0 < d \le D < +\infty$ and let $u \in C^\infty(\overline{Q_D(y_0)})$ be such that $u \equiv 0$ in a neighborhood of $\Sigma$. Then, we have 
		\begin{equation*}
			\begin{split}
				\frac{\mathrm{Cap}_s\left(\Sigma; Q_d(y_0)\right)}{d^N}\, \int_{Q_d(y_0)}|u|^2\,dx &\le C\,\frac{\mathrm{Cap}_s\left(\Sigma; Q_d(y_0)\right)}{d^{N-2}}\, \int_{Q_d(y_0)}|\nabla u|^2\,dx\\&+ C\,(\mathcal{W}+1)\,\left(\frac{d}{\delta}\right)^{2\,s}\,\iint_{Q_d(y_0) \times Q_D(y_0)} \frac{|\nabla u(x)-\nabla u(y)|^2}{|x-y|^{N+2\,(s-1)}}\,dx\,dy,  
			\end{split}
		\end{equation*}
		where $C=C(N)>0$ and $\mathcal{W}=\mathcal{W}\left(N,s,d/D\right) > 0$ is the constant given by \eqref{vudoppia}.
	\end{theorem}
	\begin{proof}
	We can suppose that $\Sigma$ has positive $s-$capacity, otherwise there is nothing to prove.
		Again, without loss of generality, we can assume that $y_0 = 0$ and that $u$ satisfies  
		\begin{equation} \label{hp_di_lavoro}
			\|\nabla u\|_{L^2(Q_d)} < \frac{1}{8\,\sqrt{N}\,d}\,\|u\|_{L^2(Q_d)}.
		\end{equation}
		Indeed, if 
		\[
		\|\nabla u\|_{L^2(Q_d)} \ge \frac{1}{8\,\sqrt{N}\,d}\,\|u\|_{L^2(Q_d)},
		\]
		then the desired estimate would trivially follows with $C=64\,N$.
		For $x \in \mathbb{R}^N$, we set 
		\begin{equation*}
			S[u](x) := \langle {\rm av}(\nabla u; Q_d); x \rangle \qquad\text{and}\qquad v:= u-S[u]. 
		\end{equation*}
		We can assume that 
		\begin{equation} \label{media_v}
			{\rm av}(v; Q_d) \ge 0,
		\end{equation}
		up to replace $u$ with $-u$. We want to show that $v$ satisfies the assumptions of Lemma \ref{lm:pag673_frac}.
		\par
	We need some preliminary estimates.	Observe that by Cauchy-Schwarz and Jensen inequalities, we have 
		\begin{equation} \label{stima_polinomio_gradiente}
			\begin{split}
				\sup_{x\in Q_d}\big|S[u](x)\big| \le 2\,\sqrt{N}\,d\,\left(\sum_{i=1}^N \left| \fint_{Q_d} u_{x_i}\,dx\right|^2\right)^\frac{1}{2} \le \frac{2\,\sqrt{N}\,d}{\sqrt{|Q_d|}}\,\|\nabla u\|_{L^2(Q_d)}.
			\end{split}
		\end{equation} We used that 
		\[
		\mathrm{diam}(Q_d)=2\,\sqrt{N}\,d.
		\]
By using \eqref{stima_polinomio_gradiente}, together with the normalization \eqref{hp_di_lavoro}, we also get that
	\[
	\big\|S[u]\big\|_{L^2(Q_d)}\le \sqrt{|Q_d|}\,\sup_{x\in Q_d}\big|S[u](x)\big| \le 2\,\sqrt{N}\,d\,\|\nabla u\|_{L^2(Q_d)}<\frac{1}{4}\,\|u\|_{L^2(Q_d)}. 
	\]	 
Thus, by the triangle inequality, this in particular gives
\[
\|u\|_{L^2(Q_d)}\le \|v\|_{L^2(Q_d)}+\big\|S[u]\big\|_{L^2(Q_d)}<\|v\|_{L^2(Q_d)}+\frac{1}{4}\,\|u\|_{L^2(Q_d)},
\]
which in turn implies
\begin{equation}
\label{manigoldo}
\|u\|_{L^2(Q_d)}<\frac{4}{3}\,\|v\|_{L^2(Q_d)}.
\end{equation}
We are ready to show that $v$ verify the assumptions of Lemma \ref{lm:pag673_frac}. By using \eqref{stima_polinomio_gradiente} and \eqref{manigoldo}, the normalization \eqref{hp_di_lavoro} and the fact that $u=0$ on $\Sigma$, we infer that 
\[
			\max_{x\in \Sigma}|v(x)| = \max_{x\in \Sigma}\big|S[u](x)\big|\le \sup_{x\in Q_d}\big|S[u](x)\big|\le \frac{1}{4}\,\frac{\|u\|_{L^2(Q_d)}}{\sqrt{|Q_d|}}<\frac{1}{3}\,\frac{\|v\|_{L^2(Q_d)}}{\sqrt{|Q_d|}}.
\]		By recalling also \eqref{media_v}, we can finally infer that $v$ satisfies the assumptions of Lemma \ref{lm:pag673_frac}, with $\varepsilon=1/3$. This yields
		\begin{equation} \label{perdono}
			\begin{split}
				\frac{4}{9}\,\frac{\mathrm{Cap}_s(\Sigma; Q_d)}{d^N}&\big\|u-S[u]\big\|^2_{L^2(Q_d)}\\&\le C\,\big[u-S[u]\big]^2_{W^{s, 2}(Q_d)}\\
				 &+C\,\left(\frac{d}{\delta}\right)^2\,\frac{1}{(s-1)\,(2-s)}\,\frac{1}{\delta^{2\,(s-1)}}\,\big\|\nabla\left(u-S[u]\right)\big \|^2_{L^2(Q_d)}.
			\end{split} 
		\end{equation}
		We now observe that 
		\begin{equation*}
			\big[u-S[u]\big]_{W^{s, 2}(Q_d)} = \big[\nabla(u-S[u])\big]_{W^{s-1, 2}(Q_d)} = [\nabla u]_{W^{s-1, 2}(Q_d)} = [u]_{W^{s,2}(Q_d)},
		\end{equation*}
		thanks to the fact that $S[u]$ is a linear function, thus $\nabla S[u]$ is a constant vector. Thus, from \eqref{perdono}  we easily get
	\begin{equation} \label{xdono}		
	\begin{split}
	\frac{4}{9}\,\frac{\mathrm{Cap}_s(\Sigma; Q_d)}{d^N}&\big\|u-S[u]\big\|^2_{L^2(Q_d)}\\&\le C\,\iint_{Q_d \times Q_D} \frac{|\nabla u(x)- \nabla u(y)|^2}{|x-y|^{N+2\,(s-1)}}\,dx\,dy\\
				 &+C\,\left(\frac{d}{\delta}\right)^2\,\frac{1}{(s-1)\,(2-s)}\,\frac{1}{\delta^{2\,(s-1)}}\,\big\|\nabla\left(u-S[u]\right)\big \|^2_{L^2(Q_d)}.
				 \end{split}
				 \end{equation}
The constant $C$ only depends on the dimension $N$.				 
		Moreover, by applying the ``asymmetric'' Poincar\'e-Wirtinger inequality of Proposition \ref{lm:pw_frac} with the choice $\sigma=s-1$, we can estimate the last term as follows
		\begin{equation*}
			\begin{split}
				\big\|\nabla\left(u-S[u]\right) \big\|^2_{L^2(Q_d)} &=\sum_{i=1}^N \|u_{x_i}-\mathrm{av}(u_{x_i};Q_d)\|^2\\
				&\le \mathcal{W}\,(s-1)\,(2-s)\,d^{2\,(s-1)}\, \iint_{Q_d \times Q_D} \frac{|\nabla u(x)- \nabla u(y)|^2}{|x-y|^{N+2\,(s-1)}}\,dx\,dy.
			\end{split}
		\end{equation*}
Here $\mathcal{W}=\mathcal{W}(N,s,d/D)>0$ is the same constant defined in \eqref{vudoppia}.		
		By spending this information in \eqref{xdono} and using that $d/\delta\ge 1$, we get
		\begin{equation} \label{qui}
			\begin{split}
				\frac{4}{9}\,\frac{\mathrm{Cap}_s(\Sigma; Q_d)}{d^N}\,\big\|u-S[u]\big\|^2_{L^2(Q_d)} \le C\,(\mathcal{W}+1)\,\left(\frac{d}{\delta}\right)^{2\,s}\,\iint_{Q_d \times Q_D} \frac{|\nabla u(x)- \nabla u(y)|^2}{|x-y|^{N+2\,(s-1)}}\,dx\,dy.
			\end{split}
		\end{equation}
We are ready to conclude: by the triangle inequality and the two estimates \eqref{stima_polinomio_gradiente} and \eqref{qui}, we get  
		\begin{equation*}
			\begin{split}
				\|u\|^2_{L^2(Q_d)} &\le 2\,\big\|S[u]\big\|^2_{L^2(Q_d)} +2\, \big\|u-S[u]\big\|_{L^2(Q_d)}^2\\
				&\le 2\,|Q_d|\,\sup_{x\in Q_d} \big|S[u](x)\big|^2+2\,\big\|u-S[u]\big\|^2_{L^2(Q_d)}\\	
				& \le 8\,N\,d^2\,\|\nabla u\|_{L^2(Q_d)}\\
				& + \frac{9}{2}\,\frac{d^N}{\mathrm{Cap}_s(\Sigma; Q_d)}\, C\,(\mathcal{W}+1)\,\left(\frac{d}{\delta}\right)^{2\,s}\,\iint_{Q_d \times Q_D} \frac{|\nabla u(x)-\nabla u(y)|^2}{|x-y|^{N+2\,(s-1)}}\,dxdy,
			\end{split}
		\end{equation*}
		which yields the claimed estimate. 
	\end{proof}
We recall that the constant $\mathcal{W}$ entering in the previous Maz'ya-Poincar\'e inequality is given by
\[
 \mathcal{W}\left(N,\sigma, d/D\right) =\frac{W_N}{\sigma+(1-\sigma)\, \left(1 - ((2\,d)/D)^{2\,\sigma}\right)_+},
\]	
		 where $W_N>0$ only depends on the dimension. In particular, observe that
\[
\lim_{D\to+\infty}\mathcal{W}\left(N, s, \frac{d}{D}\right)=\frac{W_N}{\sigma+(1-\sigma)}=W_N.
\] 
Thus, by taking the limit as $D$ goes to $+\infty$ in Theorem \ref{thm:mp_frac} and using the above observation,
we easily get the following 
\begin{corollary}
\label{coro:francesco}
		Let $1 < s < 2$ and let $\Sigma\subseteq\mathbb{R}^N$ be a compact set. Let $Q_d(y_0)\subseteq\mathbb{R}^N$ be a cube containing $\Sigma$ and set
\[
\delta:=\mathrm{dist}(\Sigma,\partial Q_d(y_0)).
\] 
For every $u \in C^\infty_0(\mathbb{R}^N)$ such that $u \equiv 0$ in a neighborhood of $\Sigma$, we have
		\begin{equation*}
			\begin{split}
				\frac{\mathrm{Cap}_s\left(\Sigma; Q_d(y_0)\right)}{d^N} 
				\,\int_{Q_d(y_0)} |u|^2\,dx &\le C\,\frac{\mathrm{Cap}_s\left(\Sigma; Q_d(y_0)\right)}{d^{N-2}}\, \int_{Q_d(y_0)}|\nabla u|^2\,dx \\&+ C\,\left(\frac{d}{\delta}\right)^{2\,s}\,\iint_{Q_d(y_0) \times \mathbb{R}^N} \frac{|\nabla u(x)-\nabla u(y)|^2}{|x-y|^{N+2\,(s-1)}}\,dx\,dy,
			\end{split}
		\end{equation*}
		for a constant $C=C(N)>0$.
\end{corollary}

\section{Proof of the main results}
\label{sec:7}
	\subsection{Proof of Theorem \ref{thm:local}}
We assume that $r_\Omega < +\infty$, otherwise there is nothing to prove. We take a tiling of the space (up to a set of measure zero) made of disjoint hypercubes $\left\{Q_{\bf i}\right\}_{{\bf i} \in \mathbb{Z}^N}$ with sides $3\,r_\Omega$. More precisely, we define
		\begin{equation}
		\label{tiling}
			Q_{\bf i} := Q_{3\,r_\Omega}(6\,\mathbf{i}),\qquad \text{for}\ {\bf i} \in \mathbb{Z}^N.  
		\end{equation}
By construction, for every $\mathbf{i}\in\mathbb{Z}^N$ we have $B_{2\,r_\Omega}(6\,\mathbf{i})\subseteq Q_{\bf i}$ and, thanks to the fact that $2\,r_\Omega>r_\Omega$, by definition of inradius there exists a point $x_\mathbf{i}\in B_{2\,r_\Omega}(6\,\mathbf{i})\setminus\Omega$. Moreover, still by construction we clearly have 
\[
\delta_\mathbf{i}:=\mathrm{dist}(x_\mathbf{i};\partial Q_{\bf i})\ge r_\Omega.
\]		
Let $u \in C^\infty_0(\Omega) \setminus \left\{0\right\}$, extended by zero outside $\Omega$. According to the previous discussion, for every $\mathbf{i}\in\mathbb{Z}^N$ by Lemma \ref{lm:lemmasega} we have
\[
\frac{\mathrm{Cap}_2(\{x_\mathbf{i}\};Q_\mathbf{i})}{(3\,r_\Omega)^N}\,\int_{Q_\mathbf{i}}|u|^2\,dx\le C\,\int_{Q_\mathbf{i}} |\Delta u|^2\,dx+\frac{C}{r_\Omega^2}\,\int_{Q_\mathbf{i}} |\nabla u|^2\,dx,
\]
for a constant $C=C(N)>0$. Recall that the capacity of a point is positive (see Lemma \ref{lm: capacità punto positiva}). 
Moreover, we can uniformly bound it from below as follows
		\begin{equation*}
		\mathrm{Cap}_2(\{x_\mathbf{i}\};Q_\mathbf{i}) \ge \mathrm{Cap}_2( \{x_{\bf i}\}  ; Q_{5\,r_\Omega}(x_{\bf i}))=(5\,r_\Omega)^{N-4}\,\mathrm{Cap}_2(\{0\};Q_1),
		\end{equation*} 
		since $Q_{\mathbf{i}} \subseteq Q_{5\,r_\Omega}(\{x_{\bf i}\})$, for every $\mathbf{i} \in \mathbb{Z}^N$. We can thus obtain
\[
\left(\frac{1}{r_\Omega}\right)^{4}\,\int_{Q_\mathbf{i}}|u|^2\,dx\le C\,\int_{Q_\mathbf{i}} |\Delta u|^2\,dx+\frac{C}{r_\Omega^2}\,\int_{Q_\mathbf{i}} |\nabla u|^2\,dx,
\]
possibly for a different constant $C=C(N)>0$. By summing all the contributions, we infer that 
		\begin{equation*}
			\begin{split}
			\left(\frac{1}{r_\Omega}\right)^{4}\,	\|u\|^2_{L^2(\Omega)} = \left(\frac{1}{r_\Omega}\right)^{4}\,\sum_{{\bf i} \in \mathbb{Z}^{N}} \|u\|^2_{L^{2}(Q_{\bf i})} &\le C\, \sum_{{\bf i} \in \mathbb{Z}^{N}} \|\Delta u\|^2_{L^{2}(Q_{\bf i})}+ \frac{C}{r_\Omega^2}\,\sum_{{\bf i} \in \mathbb{Z}^{N}}\|\nabla u\|^2_{L^2(Q_{\bf i})}\\
			&=C\,\|\Delta u\|^2_{L^2(\Omega)}+\frac{C}{r_\Omega^2}\,\|\nabla u\|^2_{L^2(\Omega)}.
			\end{split}
		\end{equation*}
We now need to reabsorb the inradius contained in the right-hand side. This can be done by using the interpolation inequality \eqref{interpol2} and Young's inequality. More precisely, we can infer that
		\[		
	C\,	\int_{\Omega} |\nabla u|^2\,dx \le C\,\|u\|_{L^2(\Omega)}\, \|\Delta u\|_{L^2(\Omega)}\le \frac{\|u\|^2_{L^2(\Omega)}}{2\,r_\Omega^2} + \frac{C^2}{2}\,r_\Omega^2\,\|\Delta u\|^2_{L^2(\Omega)}.
		\]
	This entails that 
		\begin{equation*}
			\begin{split}
				\frac{1}{2}\,\left(\frac{1}{r_\Omega}\right)^{4}\,\|u\|^2_{L^2(\Omega)} \le \left(C + \frac{C^2}{2}\right)\, \|\Delta u\|^2_{L^2(\Omega)}.
			\end{split}
		\end{equation*}
By arbitrariness of $u\in C^\infty_0(\Omega)$, we get the desired result \eqref{rossino}. 

\begin{corollary}
\label{coro:clamped}
Let $N\in\{2, 3\}$. For every open set $\Omega \subseteq \mathbb{R}^N$, we have 
\begin{equation*}
		\Lambda(\Omega) > 0 \quad \Longleftrightarrow \quad r_\Omega < +\infty.
\end{equation*}
\end{corollary}	
\begin{proof}
The implication $\Longleftarrow$ simply follows from Theorem \ref{thm:local}.
\par
For the converse implication, let $B_r(x_0)\subseteq\Omega$ be a ball. By the monotonicity of $\Lambda$ with respect to inclusion of sets and the scaling law \eqref{scaling_law}, we obtain
\[
\Lambda(\Omega)\le \frac{\Lambda(B_1)}{r^4}.
\]	
Since we are now assuming $\Lambda(\Omega)>0$, this can be rewritten as
\[
r\le \left(\frac{\Lambda(B_1)}{\Lambda(\Omega)}\right)^\frac{1}{4}.
\]
By arbitrariness of $r$, we get that $r_\Omega<+\infty$. More precisely, we have the following estimate
\[
r_\Omega\le \left(\frac{\Lambda(B_1)}{\Lambda(\Omega)}\right)^\frac{1}{4}.
\]
This concludes the proof.
\end{proof}
	\subsection{Proof of Theorem \ref{thm:frac}}
Again, we assume that $r_\Omega<+\infty$, otherwise the estimate is trivially true.	We then consider the same tiling $\{Q_\mathbf{i}\}_{\mathbf{i}\in\mathbb{Z}^N}$ given by \eqref{tiling}. Exactly as before, for every $\mathbf{i}\in\mathbb{Z}^N$  
there exists a point $x_\mathbf{i}\in B_{2\,r_\Omega}(6\,\mathbf{i})\setminus\Omega$ such that
\[
\delta_\mathbf{i}:=\mathrm{dist}(x_\mathbf{i};\partial Q_{\bf i})\ge r_\Omega.
\]		
Thus, if $u \in C^\infty_0(\Omega)\setminus\{0\}$ we can apply this time Corollary \ref{coro:francesco} on each hypercube $Q_{\bf i}$ and get 
		\begin{equation*}
			\begin{split}
				\frac{\|u\|^2_{L^2(Q_{\bf i})}}{r_\Omega^N} \le \frac{C}{r_\Omega^{N-2}}\|\nabla u\|^2_{L^2(Q_{\bf i})} + \frac{C}{\mathrm{Cap}_s\left(\left\{x_{\bf i}\right\}; Q_{\bf i}\right)}\iint_{Q_{\bf i} \times \mathbb{R}^N} \frac{|\nabla u(x) - \nabla u(y)|^2}{|x-y|^{N+2(s-1)}}dxdy,
			\end{split}
		\end{equation*}
		for every ${\bf i} \in \mathbb{Z}^N$, where $C=C(N)>0$. Observe that points have positive relative capacity, thanks to the fact that $s>N/2$ (see again Corollary \ref{lm: capacità punto positiva}). Moreover, as before we can estimate it from below
		\[
			\mathrm{Cap}_s(\{x_\mathbf{i}\};Q_\mathbf{i}) \ge \mathrm{Cap}_s( \{x_{\bf i}\}  ; Q_{5\,r_\Omega}(x_{\bf i}))=(5\,r_\Omega)^{N-2\,s}\,\mathrm{Cap}_s(\{0\};Q_1).
\]		
By using this estimate, we get
		\begin{equation*}
			\begin{split}
		\left(\frac{1}{r_\Omega}\right)^{2\,s}\,		\|u\|^2_{L^2(Q_{\bf i})} \le C\,r_\Omega^{2-2\,s}\,\|\nabla u\|^2_{L^2(Q_{\bf i})} +\frac{C}{\mathrm{Cap}_s\left(\left\{0\right\}; Q_{1}\right)}\iint_{Q_{\bf i} \times \mathbb{R}^N} \frac{|\nabla u(x) - \nabla u(y)|^2}{|x-y|^{N+2(s-1)}}\,dx\,dy,
			\end{split}
		\end{equation*}
		for every ${\bf i} \in \mathbb{Z}^N$, possibly for a different constant $C=C(N)>0$. We now sum up all the contributions of the hypercubes, so to get
		\begin{equation}
		\label{camillo}
			\begin{split}
		\left(\frac{1}{r_\Omega}\right)^{2\,s}\,	\|u\|^2_{L^2(\Omega)} &\le C\,r_\Omega^{2-2\,s}\,\|\nabla u\|^2_{L^2(\Omega)} +  \frac{C}{\mathrm{Cap}_{s}\left(\left\{0\right\}; Q_{1}\right)} [u]^2_{W^{s,2}(\mathbb{R}^N)}.
			\end{split}
		\end{equation} 
Observe that we crucially exploited the fact that
\[
\sum_{\mathbf{i}\in\mathbb{Z}^N}\iint_{Q_{\bf i} \times \mathbb{R}^N} \frac{|\nabla u(x) - \nabla u(y)|^2}{|x-y|^{N+2(s-1)}}\,dx\,dy=[u]^2_{W^{s,2}(\mathbb{R}^N)},
\]		
thanks to the fact that the hypercubes are tiling the space and that 
\[
\left(Q_{\bf i} \times \mathbb{R}^N\right)\cap \left(Q_{\bf j} \times \mathbb{R}^N\right)=\emptyset,\qquad \text{for}\ \mathbf{i}\not=\mathbf{j}.
\]
We now apply the interpolation inequality \eqref{eq: interpolaz 1} in the right-hand side of \eqref{camillo}, so to get for every $\delta>0$
\[
			\begin{split}
		\left(\frac{1}{r_\Omega}\right)^{2\,s}\,	\|u\|^2_{L^2(\Omega)} \le C\,r_\Omega^{2-2\,s}\,\delta\,\left(\frac{s-1}{s}\right)\|u\|^2_{L^2(\Omega)} +\left(\frac{2\,C\,A_{s-1}}{(r_\Omega^2\,\delta)^{s-1}}+\frac{C}{\mathrm{Cap}_s\left(\left\{0\right\}; Q_{1}\right)}\right) [u]^2_{W^{s,2}(\mathbb{R}^N)}.
			\end{split}
\]
In order to reabsorb the $L^2$ norm, we choose
\[
\delta=\frac{1}{2\,C\,r_\Omega^2}\,\frac{s}{s-1}.
\]
This permits to infer that 
		\begin{equation}
		\label{bestemmioni}
			\begin{split}
	\frac{1}{2}\,\left(\frac{1}{r_\Omega}\right)^{2\,s}\,\|u\|^2_{L^2(\Omega)} \le \frac{\Theta_{N,s}}{{\mathrm{Cap}_s\left(\left\{0\right\}; Q_{1}\right)}}\, [u]^2_{W^{s,2}(\mathbb{R}^N)}.
			\end{split}
		\end{equation}
where we set for simplicity
\[
\Theta_{N,s}:=\left((2\,C)^{2\,s}\,\left(\frac{s-1}{s}\right)^{s-1}\,A_{s-1}\,\mathrm{Cap}_s\left(\left\{0\right\}; Q_{1}\right)+C\right).
\]
In order to get the desired estimate with a good asymptotic control on the relevant constant, we can estimate from above the $s-$capacity of $\{0\}$ by Proposition \ref{prop:capVScap}. Thus, 
 we have
\[
\Theta_{N,s}\le \left((2\,C)^{s-1}\,\left(\frac{s-1}{s}\right)^{s-1}\,\frac{A_{s-1}}{(2-s)\,(s-1)}\,\frac{C_N\,\mathrm{Cap}_{2}\left(\left\{0\right\}; Q_{1}\right)}{\Big(\lambda(Q_1)\Big)^{2-s}}+C\right)=:\mathcal{C}_{N,s}.
\]
This is enough to get from \eqref{bestemmioni} the seeked lower bound on $\Lambda^s(\Omega)$, with constant given by
\[
\frac{{\mathrm{Cap}_s\left(\left\{0\right\}; Q_{1}\right)}}{2\,\mathcal{C}_{N,s}}.
\]
The claimed asymptotic behavior now follows from that of $A_{s-1}$, given by \eqref{asintotico-fourier}. This easily gives  
\[
			\mathcal{C}_{N,s} \sim 1, \quad \mbox{ as }\,\, s \searrow N/2\,\, \mbox{ and }\,\, s \nearrow 2,
\]
The proof is concluded.
\begin{corollary}
Let $N\in\{2, 3\}$ and $N/2<s<2$. For every open set $\Omega \subseteq \mathbb{R}^N$, we have 
\begin{equation*}
		\Lambda^s(\Omega) > 0 \quad \Longleftrightarrow \quad r_\Omega < +\infty.
\end{equation*}
\end{corollary}	
\begin{proof}
The proof is exactly the same as that of Corollary \ref{coro:clamped}.
\end{proof}	
	\section{Asymptotic optimality}
	\label{sec:8}
	This section is devoted to discuss the optimality of the lower bound of Theorem \ref{thm:frac}, with respect to the fractional differentiability parameter $s$. 
	\subsection{Failure for $N=3$ and $s\le 3/2$}	In dimension $N=3$, if the order of fractional differentiation $s$ is below or equal to the critical threshold $3/2$, then the inequality \eqref{higher} fails, for general open sets. This is due to the fact that points are removable, as shown in Proposition \ref{prop: punti removibili nello spazio}.
\par	
By using this last result, it is now very easy to produce a counter-example to the validity of the lower bound of Theorem \ref{thm:frac}, for $s\le 3/2$ and $N=3$. We can simply take
\[
\Omega_k=Q_k\cap (\mathbb{R}^3\setminus \mathbb{Z}^3),\qquad \text{for every}\ k\in\mathbb{N}.
\]
Then we observe that
\[
r_{\Omega_k}=\frac{\sqrt{3}}{2},\qquad \text{for every}\ k\ge 2,
\]
while by using iteratively Proposition \ref{prop: punti removibili nello spazio} and the scaling law \eqref{scaling_law}, for every $1<s\le 3/2$ we have
\[
\Lambda^s(\Omega_k)=\Lambda^s(Q_k)=k^{-2\,s}\,\Lambda^2(Q_1).
\] 
Thus, we easily get
\[
\lim_{k\to\infty} r_{\Omega_k}^{2\,s}\,\Lambda^s(\Omega_k)=0.
\]
\subsection{Asymptotics for $s\nearrow 2$} We show here that the dependence on $s$ in the lower bound of \eqref{higher} is optimal, in the regime ``$s$ converging to $2$''. 
\par
Indeed, let $\Omega\subseteq\mathbb{R}^N$ be an open set with finite inradius, where $N\in\{2,3\}$. We multiply both sides of \eqref{higher} by $(2-s)$, so to get
\[
\frac{(2-s)\,\mathrm{Cap}_s(\{0\}; Q_1)}{\mathcal{C}_2}\,\left(\frac{1}{r_\Omega}\right)^{2\,s} \le (2-s)\,\Lambda^s(\Omega).
\]
We observe that by Corollary \ref{lm: capacità punto positiva} we have 
\[
(2-s)\,\mathrm{Cap}_s(\{0\}; Q_1)\ge \frac{2-s}{\theta_{N,s}}\,(2\,\sqrt{N})^{N-2\,s}.
\]
Moreover, by \eqref{eq: asintotica cost} we recall that
\[
\frac{2-s}{\theta_{N,s}}\sim 1,\qquad \text{as}\ s\nearrow 2,
\]
so that 
\[
\liminf_{s\nearrow 2} \frac{(2-s)\,\mathrm{Cap}_s(\{0\}; Q_1)}{\mathcal{C}_2}\,\left(\frac{1}{r_\Omega}\right)^{2\,s}\ge \frac{1}{C}\,\frac{1}{r_\Omega^2}.
\]
On the other hand, by Corollary \ref{coro:BBM} we have
\[
\limsup_{s\nearrow 2}(2-s)\,\Lambda^s(\Omega)\le \frac{\omega_N}{2}\,\Lambda(\Omega).
\]
By using these two limits, we retrieve the lower bound of Theorem \ref{thm:local} from \eqref{higher}.

	\subsection{Asymptotics for $s\searrow N/2$}
Finally, we prove that the asymptotic behavior of the lower bound is optimal, also in the regime ``$s$ converging to $N/2$''. This is much more delicate.	
		\begin{theorem}
		Let $N\in\{2,3\}$ and let us consider the open set $\Omega = \mathbb{R}^N \setminus (2\,\mathbb{Z}^{N})$, having inradius $r_\Omega=\sqrt{N}$.
		Then we have
		\begin{equation} \label{claim:opt}
			\limsup_{s \searrow N/2} \frac{\Lambda^s(\Omega)}{\mathrm{Cap}_s(\{0\}; Q_1)} < +\infty. 
		\end{equation}
		In particular, the dependence on $s$ of the inequality \eqref{higher} is sharp.  
	\end{theorem}
	\begin{proof}
		 In order to prove the claimed asymptotics, we will need to judiciously estimate from above $\Lambda^s(\Omega)$. This will be done by constructing a suitable sequence of test functions (depending on $s$), along the lines of what done in the proof of \cite[Theorem 1.2, point (3)]{BiaBra24}. However, the fact that now $s>1$ will cause some (nontrivial) additional issues.
\vskip.2cm\noindent
{\bf 1. Set-up}.
By definition of capacity as a infimum, for every $N/2<s<2$ there exists a function $\zeta_s\in C^\infty_0(Q_1)$ such that 
\begin{equation}
\label{zetas}
\zeta_s\equiv 1\ \text{in a neighborhood of}\ \{0\}\qquad\text{and}\qquad [\zeta_s]^2_{W^{s,2}(\mathbb{R}^N)}<2\,\mathrm{Cap}_s(\{0\};Q_1).
\end{equation}
Observe that we used that points has positive capacity, in this regime. 
For every $n\in\mathbb{N}\setminus\{0\}$, we define the following smooth multi-funnel function 
\[
			\psi_{n, s, \varepsilon}(x) := 1-\sum_{{\bf i} \in  2\,\mathbb{Z}^N \cap Q_n} \zeta_{s}\left(\frac{x-\mathbf{i}}{\varepsilon}\right).
\]
By construction $\psi_{n,s,\varepsilon}\in C^\infty(\mathbb{R}^N)$ and it identically vanishes in a neighborhood of every point $\mathbf{i}\in (2\,\mathbb{Z})^N\cap Q_n$.		
Moreover, by construction and by Proposition \ref{prop: morrey ineq}, we have 
\[
\|\psi_{n, s, \varepsilon}\|_{L^\infty(\mathbb{R}^N)} \le 1+\|\zeta_s\|_{L^\infty(Q_1)}\le 1+C_N\,\left(\sqrt{\theta_{N,s}}\,[\zeta_s]_{W^{s,2}(\mathbb{R}^N)}\right)^{\frac{N}{2\,s}}\,\left(\|\zeta_s\|_{L^2(Q_1)}\right)^{1-\frac{N}{2\,s}},
\]
where $\theta_{N,s}$ has the asymptotic behaviors \eqref{eq: asintotica cost}. We report it here for the reader's convenience
\[
			\theta_{N,s}\sim 2-s,\ \text{as}\ s\nearrow 2\qquad\text{and}\qquad \begin{cases}
			\theta_{2,s}\sim1,&\mbox{ as }s\searrow 1,\\&\\
		\theta_{3,s}\sim \dfrac{1}{2\,s-3},&\mbox{ as }s\searrow\dfrac{3}{2}.
		\end{cases}
			\] 
			By using Poincar\'e inequality to estimate the $L^2$ norm of $\zeta_s$ and the second property in \eqref{zetas}, we get
\begin{equation} \label{bound-l-infty}
\|\psi_{n, s, \varepsilon}\|_{L^\infty(\mathbb{R}^N)}\le 1+C_N\,\left(\sqrt{\theta_{N,s}}\,\right)^{\frac{N}{2\,s}}\,\left(\frac{1}{\sqrt{\Lambda^s(Q_1)}}\right)^{1-\frac{N}{2\,s}}\,\sqrt{\mathrm{Cap}_s(\{0\}; Q_1)}.
\end{equation}
We also take a function $u \in C^\infty_0(Q_1)$ with $\|u\|_{L^2(Q_1)}=1$ and define 
\[
u_n(x) := u\left(\frac{x}{n}\right),\qquad \text{for}\ n \in \mathbb{N} \setminus \{0\}.
\] 
Then the function 
		\[
		u_n\,\psi_{n, s, \varepsilon} \in C^\infty_0(Q_n \cap \Omega) \subseteq C^\infty_0(\Omega),
		\] 
		and so it is a feasible competitor to test the definition of $\Lambda^s(\Omega)$. In particular, by Minkowski's inequality we get
		\begin{equation} \label{conto-esempio}
			\begin{split}
				\Lambda^s(\Omega) &\le \frac{2\,[u_n\,\nabla \psi_{n, s, \varepsilon}]^2_{W^{s-1, 2}(\mathbb{R}^N)} + 2\,[\psi_{n, s, \varepsilon}\,\nabla u_n]^2_{W^{s-1, 2}(\mathbb{R}^N)}}{\|u_n\,\psi_{n, s,\varepsilon}\|^2_{L^2(Q_n)}} \\
				&=: \frac{\mathcal{I}_1 + \mathcal{I}_2}{\|u_n\,\psi_{n, s,\varepsilon}\|^2_{L^2(Q_n)}}.
			\end{split}
		\end{equation} 
We need to estimate $\mathcal{I}_1, \mathcal{I}_2$ and the denominator separately. 
\vskip.2cm\noindent
{\bf 2. Estimate of $\mathcal{I}_1$.} For the first one, by adding and subtract $\nabla \psi_{n,s,\varepsilon}(x)\,u_n(y)$, we get
			\begin{equation*}
			\begin{split}
				\mathcal{I}_1 &\le 4\,\|u_n\|^2_{L^\infty(Q_n)}\,[\nabla \psi_{n, s, \varepsilon}]^2_{W^{s-1, 2}(\mathbb{R}^N)} \\
				&+4\, \int_{Q_n} |\nabla \psi_{n, s, \varepsilon}(x)|^2 \left(\int_{\mathbb{R}^N}\frac{|(u_n(x) - u_n(y))|^2}{|x-y|^{N+2\,(s-1)}}\,dy\right)\,dx\\
				&\le 4\,\|u\|^2_{L^\infty(Q_1)}\,[\nabla \psi_{n, s, \varepsilon}]^2_{W^{s-1, 2}(\mathbb{R}^N)}\\
				& +  \frac{C}{(s-1)\,(2-s)} \|\nabla \psi_{n, s, \varepsilon}\|^2_{L^2(Q_n)}\,\|\nabla u_n\|^{2\,(s-1)}_{L^\infty(\mathbb{R}^N)}\,\|u\|_{L^\infty(Q_1)}^{2\,(2-s)},
			\end{split}
		\end{equation*}
		with $C=C(N)>0$.
We used
\[
\|u\|^2_{L^\infty(Q_1)}=\|u_n\|^2_{L^\infty(Q_n)},
\]
and \cite[Lemma 2.7]{BiaBraZag} to estimate the incremental quotient of $u_n$. In order to estimate the $L^2$ norm of $\nabla \psi_{n,s,\varepsilon}$, we can use a change of variable to write
		\begin{equation*}
			\begin{split}
				\|\nabla \psi_{n, s, \varepsilon}\|^2_{L^2(Q_n)} &= \sum_{{\bf i} \in  \mathbb{Z}^N \cap Q_n} \int_{Q_\varepsilon({\bf i})} \left|\nabla \left(\zeta_{s}\left(\frac{x-{\bf i}}{\varepsilon}\right)\right)\right|^2\,dx \\
				&= (2n-1)^N\, \varepsilon^{N-2}\, \int_{Q_1}\left|\nabla \zeta_{s}\right|^2\,dx.
			\end{split}
		\end{equation*} 
Then the last term can be estimated by using Proposition \ref{prop:mezzapoin}, which gives
\[
\|\nabla \zeta_s\|_{L^2(Q_1)}^2\le \frac{s+1}{s\,\Lambda^s(Q_1)}\,\Big(2\,s\,A_{s-1}\,\Lambda^s(Q_1)\Big)^\frac{1}{s}\,[\zeta_s]^2_{W^{s,2}(\mathbb{R}^N)}.
\]	
Finally, by recalling \eqref{zetas}, we obtain
\[
\|\nabla \psi_{n, s, \varepsilon}\|^2_{L^2(Q_n)}\le  (2n-1)^N\, \varepsilon^{N-2}\,\frac{s+1}{s\,\Lambda^s(Q_1)}\,\Big(2\,s\,A_{s-1}\,\Lambda^s(Q_1)\Big)^\frac{1}{s}\,2\,\mathrm{Cap}_s(\{0\};Q_1).
\]	
By further observing that
		\[
	\|\nabla u_n\|^{2\,(s-1)}_{L^\infty(\mathbb{R}^N)}=n^{2-2\,s}\,\|\nabla u\|_{L^\infty(Q_1)},
		\]
and using that  $s>N/2$,	we get 
		\begin{equation*}
			\begin{split}
			\mathcal{I}_1&\le C\,[\nabla \psi_{n, s, \varepsilon}]^2_{W^{s-1, 2}(\mathbb{R}^N)}\\
					&+C\, (2n-1)^N\, \varepsilon^{N-2}\,\frac{n^{2-2\,s}}{\Lambda^s(Q_1)}\,\Big(A_{s-1}\,\Lambda^s(Q_1)\Big)^\frac{1}{s}\,\mathrm{Cap}_s(\{0\};Q_1),
			\end{split}
		\end{equation*}
		for some $C=C(N,\|u\|_{C^1(\overline{Q_1})})>0$.
		\par
		We still need to estimate the first term in the right-hand side term of the previous inequality. We first rewrite it as follows
		\begin{equation*}
			\begin{split}
			[\nabla \psi_{n, s, \varepsilon}]^2_{W^{s-1, 2}(\mathbb{R}^N)}&=\sum_{i=1}^N \Big[\big(\psi_{n, s, \varepsilon}\big)_{x_i}\Big]^2_{W^{s-1,2}(E)}=\sum_{i=1}^N\left[\sum_{{\bf i} \in  \mathbb{Z}^N \cap Q_n} \left(\zeta_{s}\left(\frac{\cdot-{\bf i}}{\varepsilon}\right)\right)_{x_i}\right]^2_{W^{s-1,2}(\mathbb{R}^N)}.
			\end{split}
		\end{equation*}
Then, by using that the functions $\{\zeta_s((x-\mathbf{i})/\varepsilon)\}_{\mathbf{i}\in 2\,\mathbb{Z}^N}$ have disjoint supports, it is not difficult to see that for every $i\in\{1,\dots,N\}$ we have
\[
\sum_{{\bf i} \in  \mathbb{Z}^N \cap Q_n} \left(\zeta_{s}\left(\frac{\cdot-{\bf i}}{\varepsilon}\right)\right)_{x_i}=\max_{{\bf i} \in  \mathbb{Z}^N \cap Q_n} \left(\zeta_{s}\left(\frac{\cdot-{\bf i}}{\varepsilon}\right)\right)_{x_i}+ \min_{{\bf i} \in  \mathbb{Z}^N \cap Q_n} \left(\zeta_{s}\left(\frac{\cdot-{\bf i}}{\varepsilon}\right)\right)_{x_i}.
\]	
We can thus exploit the submodularity of Gagliardo-Slobodecki\u{\i} seminorms (see \cite[Theorem 3.2]{GiMos}) to infer that
\[
\begin{split}
\left[\sum_{{\bf i} \in  \mathbb{Z}^N \cap Q_n} \left(\zeta_{s}\left(\frac{\cdot-{\bf i}}{\varepsilon}\right)\right)_{x_i}\right]^2_{W^{s-1,2}(\mathbb{R}^N)}&\le 2\,\left[\max_{{\bf i} \in  \mathbb{Z}^N \cap Q_n} \left(\zeta_{s}\left(\frac{\cdot-{\bf i}}{\varepsilon}\right)\right)_{x_i}\right]^2_{W^{s-1,2}(\mathbb{R}^N)}\\
&+2\,\left[\min_{{\bf i} \in  \mathbb{Z}^N \cap Q_n} \left(\zeta_{s}\left(\frac{\cdot-{\bf i}}{\varepsilon}\right)\right)_{x_i}\right]^2_{W^{s-1,2}(\mathbb{R}^N)}\\
&\le 2\,\sum_{{\bf i} \in  \mathbb{Z}^N \cap Q_n} \left[\left(\zeta_{s}\left(\frac{\cdot-{\bf i}}{\varepsilon}\right)\right)_{x_i}\right]^2_{W^{s-1,2}(\mathbb{R}^N)}.
\end{split}
\]
This, in turn, permits to infer that
\[
\begin{split}
[\nabla \psi_{n, s, \varepsilon}]^2_{W^{s-1, 2}(\mathbb{R}^N)}&\le 2\sum_{i=1}^N\sum_{{\bf i} \in  \mathbb{Z}^N \cap Q_n} \left[\left(\zeta_{s}\left(\frac{\cdot-{\bf i}}{\varepsilon}\right)\right)_{x_i}\right]^2_{W^{s-1,2}(\mathbb{R}^N)}\\
&=2\,\sum_{{\bf i} \in  \mathbb{Z}^N \cap Q_n}\sum_{i=1}^N \left[\left(\zeta_{s}\left(\frac{\cdot-{\bf i}}{\varepsilon}\right)\right)_{x_i}\right]^2_{W^{s-1,2}(\mathbb{R}^N)}\\
&=2\,\sum_{{\bf i} \in  \mathbb{Z}^N \cap Q_n}\,\left[\nabla \zeta_s\left(\frac{\cdot-{\bf i}}{\varepsilon}\right)\right]^2_{W^{s-1,2}(\mathbb{R}^N)}\\
&=2\,(2\,n-1)^N\,\varepsilon^{N-2\,s}\,[\nabla \zeta_s]^2_{W^{s-1,2}(\mathbb{R}^N)}\\
&\le 4\,(2\,n-1)^N\,\varepsilon^{N-2\,s}\,\mathrm{Cap}_s(\{0\};Q_1).
\end{split}
\]
		This finally entails that 
		\begin{equation} \label{stima_I_1}
			\begin{split}
			\mathcal{I}_1 &\le C\,(2n-1)^N\,\varepsilon^{N-2\,s}\,\mathrm{Cap}_s\left(\left\{0\right\}; Q_1\right)\,\left[1+\left(\frac{\varepsilon}{n}\right)^{2\,s-2}\,\frac{1}{\Lambda^s(Q_1)}\,\Big(A_{s-1}\,\Lambda^s(Q_1)\Big)^\frac{1}{s}\right].			
			\end{split}
		\end{equation}
\vskip.2cm\noindent
{\bf 3. Estimate of $\mathcal{I}_2$.} By adding and subtract $\psi_{n,s,\varepsilon}(x)\,\nabla u_n(y)$, we can estimate 
		\begin{equation}
		\label{baistrocchi}
			\begin{split}
				\mathcal{I}_2 &\le 4\,\|\nabla u_n\|_{L^\infty(Q_n)}^2\,[\psi_{n,s,\varepsilon}]_{W^{s-1,2}(\mathbb{R}^N)}^2\\
				&+4\,\int_{\mathbb{R}^N} |\psi_{n,s,\varepsilon}(x)|^2 \left(\int_{\mathbb{R}^N} \frac{|\nabla u_n(x)-\nabla u_n(y)|^2}{|x-y|^{N+2\,(s-1)}}\,dy\right)^2\,dx.
				\end{split}
		\end{equation}
The first term on the right-hand side can be estimated by observing that
\[
[\psi_{n,s,\varepsilon}]_{W^{s-1,2}(\mathbb{R}^N)}^2=\left[\sum_{{\bf i} \in  2\,\mathbb{Z}^N \cap Q_n} \zeta_{s}\left(\frac{\cdot-\mathbf{i}}{\varepsilon}\right)\right]_{W^{s-1,2}(\mathbb{R}^N)}^2
\]
and then observing, similarly as above, that
\[
\sum_{{\bf i} \in  2\,\mathbb{Z}^N \cap Q_n} \zeta_{s}\left(\frac{x-\mathbf{i}}{\varepsilon}\right)= \max_{{\bf i} \in \mathbb{Z}^N \cap Q_n} \zeta_{s} \left(\frac{x-{\bf i}}{\varepsilon}\right)+ \min_{{\bf i} \in \mathbb{Z}^N \cap Q_n} \zeta_{s} \left(\frac{x-{\bf i}}{\varepsilon}\right),
\]		
thanks to the fact that the functions $\{\zeta_s((x-\mathbf{i})/\varepsilon)\}_{\mathbf{i}\in 2\,\mathbb{Z}^N}$ have pairwise disjoint supports. We can thus exploit again the submodularity of Gagliardo-Slobodecki\u{\i} seminorms  to infer that
		\begin{equation*}
			\begin{split}
				[\psi_{n, s, \varepsilon}]^2_{W^{s-1, 2}(\mathbb{R}^N)} &\le 2\left[\max_{{\bf i} \in \mathbb{Z}^N \cap Q_n} \zeta_{s} \left(\frac{\cdot-{\bf i}}{\varepsilon}\right)\right]_{W^{s-1, 2}(\mathbb{R}^N)}^2 + 2\left[\min_{{\bf i} \in \mathbb{Z}^N \cap Q_n} \zeta_{s} \left(\frac{\cdot-{\bf i}}{\varepsilon}\right) \right]_{W^{s-1, 2}(\mathbb{R}^N)}^2\\
				&\le 2\sum_{{\bf i} \in  \mathbb{Z}^N \cap Q_n} \left[\zeta_{s} \left(\frac{\cdot-{\bf i}}{\varepsilon}\right) \right]^2_{W^{s-1,2}(\mathbb{R}^N)} 	= 2\,(2n-1)^N\,\varepsilon^{N-2\,(s-1)}\,[\zeta_{s}]^2_{W^{s-1, 2}(\mathbb{R}^N)}.
			\end{split}
		\end{equation*}
		In turn, by \eqref{luna} of Proposition \ref{prop:mezzapoin} we can estimate
		\[
		[\zeta_{s}]^2_{W^{s-1, 2}(\mathbb{R}^N)}\le\left(\frac{1}{A_{s-1}}\,\frac{1}{\Lambda^s(\Omega)}\right)^\frac{1}{s}\,\frac{s^{2-\frac{1}{s}}}{(s-1)^{2\,(1-\frac{1}{s})}}	\,[\zeta_s]^2_{W^{s,2}(\mathbb{R}^N)} .
		\]
By using that $x\mapsto x^x$ is bounded from below by a positive constant, we get that
\[
\frac{s^{2-\frac{1}{s}}}{(s-1)^{2\,(1-\frac{1}{s})}}\le C,\qquad \text{for every}\ 1<s<2,
\]		
for some universal constant $C>0$.
This permits to infer that
\[
\begin{split}
[\psi_{n, s, \varepsilon}]^2_{W^{s-1, 2}(\mathbb{R}^N)}&\le C\,(2n-1)^N\,\varepsilon^{N-2\,(s-1)}\,\left(\frac{1}{A_{s-1}}\,\frac{1}{\Lambda^s(\Omega)}\right)^\frac{1}{s}\,[\zeta_s]^2_{W^{s,2}(\mathbb{R}^N)}\\
&\le 2\,C\,(2n-1)^N\,\varepsilon^{N-2\,(s-1)}\,\left(\frac{1}{A_{s-1}}\,\frac{1}{\Lambda^s(\Omega)}\right)^\frac{1}{s}\,\mathrm{Cap}_s(\{0\};Q_1).
\end{split}
\]	
By using this estimate in \eqref{baistrocchi} and the fact that
\[
\|\nabla u_n\|_{L^\infty(Q_n)}^2=\frac{1}{n^2}\,\|\nabla u\|_{L^\infty(Q_1)}^2,
\]	
we get
		\begin{equation} \label{stima_I_2}
			\begin{split}
				\mathcal{I}_2 &\le 
				C\,(2n-1)^N\,\varepsilon^{N-2\,s}\,\left(\frac{\varepsilon}{n}\right)^2\,\left(\frac{1}{A_{s-1}}\,\frac{1}{\Lambda^s(\Omega)}\right)^\frac{1}{s}\,\mathrm{Cap}_s(\{0\};Q_1) \\&+4\,\int_{\mathbb{R}^N} |\psi_{n,s,\varepsilon}(x)|^2 \left(\int_{\mathbb{R}^N} \frac{|\nabla u_n(x)-\nabla u_n(y)|^2}{|x-y|^{N+2\,(s-1)}}\,dy\right)^2\,dx, 
			\end{split}
		\end{equation}
		for a constant $C=C(N,\|u\|_{C^1(\overline{Q_1})})>0$. In order to estimate the last term, we proceed as follows
		\[
	\begin{split}
	\int_{\mathbb{R}^N} |\psi_{n,s,\varepsilon}(x)|^2& \left(\int_{\mathbb{R}^N} \frac{|\nabla u_n(x)-\nabla u_n(y)|^2}{|x-y|^{N+2\,(s-1)}}\,dy\right)^2\,dx\\
	&\le \|\psi_{n,s,\varepsilon}\|_{L^\infty(\mathbb{R}^N)}^2\,\big[\nabla u_n\big]^2_{W^{s-1,2}(\mathbb{R}^N)}\\
	&=n^{N-2\,s}\,\|\psi_{n,s,\varepsilon}\|_{L^\infty(\mathbb{R}^N)}^2\,\big[\nabla u\big]^2_{W^{s-1,2}(\mathbb{R}^N)}\\
	&\le \frac{C}{(2-s)\,(s-1)}\,n^{N-2\,s} \,\|\psi_{n,s,\varepsilon}\|_{L^\infty(\mathbb{R}^N)}^2\,\|u\|_{L^2(Q_1)}^{2\,(2-s)}\,\|\nabla u\|_{L^{2}(Q_1)}^{2\,(s-1)},
	\end{split}	
		\]
		where we used \cite[Corollary 2.2]{BPS}, in the last estimate.
For the $L^\infty$ norm, we can use the estimate \eqref{bound-l-infty}. In conclusion, we obtain
		\begin{equation} \label{stima_I_2}
			\begin{split}
				\mathcal{I}_2 &\le 
				C\,(2n-1)^N\,\varepsilon^{N-2\,s}\,\left(\frac{\varepsilon}{n}\right)^2\,\left(\frac{1}{A_{s-1}}\,\frac{1}{\Lambda^s(\Omega)}\right)^\frac{1}{s}\,\mathrm{Cap}_s(\{0\};Q_1) \\&+\frac{C}{(2-s)\,(s-1)}\,n^{N-2\,s} \,\left[1+\left(\theta_{N,s}\,\right)^{\frac{N}{2\,s}}\,\left(\frac{1}{\Lambda^s(Q_1)}\right)^{1-\frac{N}{2\,s}}\,\mathrm{Cap}_s(\{0\}; Q_1)\right], 
			\end{split}
		\end{equation}
		for some $C=C(N,\|u\|_{C^1(\overline{Q_1})})>0$.
\vskip.2cm\noindent
{\bf 4. Estimate of the denominator.}	We now need to estimate the $L^2$ norm of $u_n\,\psi_{n, s,\varepsilon}$ from below. By observing that 
		\begin{equation*}
			\begin{split}
				\sum_{{\bf i} \in  \mathbb{Z}^N \cap Q_n} \zeta_{s}\left(\frac{x-{\bf i}}{\varepsilon}\right) = 0, \qquad \text{for}\ x \in \mathbb{R}^N \setminus \displaystyle \bigcup_{{\bf i} \in 2\,\mathbb{Z}^N \cap Q_n} Q_{\varepsilon}({\bf i}), 
			\end{split}
		\end{equation*}
with a change of variable	we infer 
		\begin{equation} \label{lhs-esempio}
			\begin{split}
				\|u_n \psi_{n, s,\varepsilon}\|^2_{L^2(Q_n)}
				&= n^{N} \int_{Q_1} |u(y)|^2 \left(1-\sum_{{\bf i} \in 2\, \mathbb{Z}^N \cap Q_n} \zeta_{s}\left(\frac{n\,y-{\bf i}}{\varepsilon}\right)\right)^2 dy \ge n^{N} \int_{A_{\varepsilon}} |u|^2 dy,
			\end{split}
		\end{equation}
		where we set 
		\begin{equation*}
			A_{\varepsilon} := Q_1 \setminus \bigcup_{{\bf i} \in 2\,\mathbb{Z}^N \cap Q_n} Q_{\frac{\varepsilon}{n}}\left(\frac{{\bf i}}{n}\right).  
		\end{equation*}
		Observe also that 
		\begin{equation} \label{ok_esempio}
			\lim_{\varepsilon \to 0} \int_{A_{ \varepsilon}} |u|^2 dy = \int_{Q_1} |u|^2 dy = 1,
		\end{equation}
		by the Dominated Convergence Theorem.
\vskip.2cm\noindent
{\bf 5. Collecting the estimates}.	
We now collect all the previous estimates.
		By using \eqref{stima_I_1}, \eqref{stima_I_2} and \eqref{lhs-esempio} in \eqref{conto-esempio}, we infer for $N/2<s\le (3N)/4$
\[
\begin{split}
				\frac{\Lambda^s(\Omega)}{\mathrm{Cap}_s\left(\{0\}; Q_1\right)} &\le \frac{C}{\|u\|^2_{L^2(A_\varepsilon)}}\,\varepsilon^{N-2\,s}\,\left[1+\left(\frac{\varepsilon}{n}\right)^{2\,s-2}\,\frac{1}{\Lambda^s(Q_1)}\,\Big(A_{s-1}\,\Lambda^s(Q_1)\Big)^\frac{1}{s}\right]\\&
								+\frac{C}{\|u\|^2_{L^2(A_\varepsilon)}}\,\varepsilon^{N-2\,s}\,\left(\frac{\varepsilon}{n}\right)^2\,\left(\frac{1}{A_{s-1}}\,\frac{1}{\Lambda^s(\Omega)}\right)^\frac{1}{s}\\
&+				\frac{C}{\|u\|^2_{L^2(A_\varepsilon)}}\,\left(\frac{\varepsilon}{n}\right)^2\,\left(\frac{1}{A_{s-1}}\,\frac{1}{\Lambda^s(\Omega)}\right)^\frac{1}{s}\\
&+\frac{C}{\|u\|^2_{L^2(A_\varepsilon)}}\,\frac{1}{(s-1)\,\mathrm{Cap}_s\left(\{0\}; Q_1\right)}\,\frac{1}{n^{2\,s}} \,\left[1+\left(\theta_{N,s}\,\right)^{\frac{N}{2\,s}}\,\left(\frac{1}{\Lambda^s(Q_1)}\right)^{1-\frac{N}{2\,s}}\,\mathrm{Cap}_s(\{0\}; Q_1)\right]. 
			\end{split}
\]
Before passing to the limit as $s$ goes to $N/2$, we polish a little bit the previous estimate.
At first, we further use Corollary \ref{lm: capacità punto positiva} to estimate
\[
\frac{1}{(s-1)\,\mathrm{Cap}_s\left(\{0\}; Q_1\right)}\le \frac{\theta_{N,s}}{s-1}\,\frac{1}{(2\,\sqrt{N})^{N-2\,s}}.
\]
Then, we also estimate $\Lambda^s(Q_1)$ from below, by means of Proposition \ref{lm:lowerbound}. The resulting upper bound takes the following form for $N/2<s\le (3\,N)/4$
		\begin{equation}
		\label{avieri}
			\begin{split}
				\frac{\Lambda^s(\Omega)}{\mathrm{Cap}_s\left(\{0\}; Q_1\right)} &\le \frac{C}{\|u\|^2_{L^2(A_\varepsilon)}}\,\varepsilon^{N-2\,s}\,\left[1+\left(\frac{\varepsilon}{n}\right)^{2\,s-2}\,\frac{A_{s-1}}{\Big(\lambda(Q_1)\Big)^s}\,\Big(A_{s-1}\,\Lambda^s(Q_1)\Big)^\frac{1}{s}\right]\\&
								+\frac{C}{\|u\|^2_{L^2(A_\varepsilon)}}\,\left(\frac{\varepsilon}{n}\right)^2\,\frac{1}{\lambda(Q_1)}\,\Big[\varepsilon^{N-2\,s}\,+1\Big]\\
&+\frac{C}{\|u\|^2_{L^2(A_\varepsilon)}}\,\frac{\theta_{N,s}}{s-1}\,\frac{1}{n^{2\,s}} \,\left[1+\left(\theta_{N,s}\,\right)^{\frac{N}{2\,s}}\,\left(\frac{A_{s-1}}{\Big(\lambda(Q_1)\Big)^s}\right)^{1-\frac{N}{2\,s}}\,\mathrm{Cap}_s(\{0\}; Q_1)\right]\\
&=:\frac{C}{\|u\|^2_{L^2(A_\varepsilon)}}\,\sum_{i=1}^3 \mathcal{R}_i(s,n,\varepsilon), 
			\end{split}
		\end{equation}
		possibly for a different constant $C=C(N,\|u\|_{C^1(\overline{Q_1})})>0$.
		\vskip.2cm\noindent
		{\bf 6. Passing to the limit}.
		We would like to take the limit as $s$ goes to $N/2$ and shrink the radius of the holes $\varepsilon$ to $0$, at the same time. In doing this, we have to pay attention to those terms in the right-hand side of \eqref{avieri} which could be problematic. These are the following
\[
\varepsilon^{N-2\,s},\qquad \theta_{N,s},\qquad \mathrm{Cap}_s(\{0\}; Q_1)\qquad \text{and}\qquad \frac{1}{s-1},
\]
the last two only for dimension $N=2$. We will see that if we let the holes shrink sufficiently fast and let the number of holes $n$ diverges sufficiently fast, we will be able to guarantee that the right-hand side of \eqref{avieri} has a finite limit, thus eventually proving the result.
\par
We thus need to choose $\varepsilon$ and $n$, suitably depending on $s$. We take 
\[
\varepsilon = \left(\frac{1}{2}\right)^\frac{1}{2\,s-N}\qquad \text{so that}\qquad \lim_{s\searrow N/2}\varepsilon=0\quad \text{and}\quad \varepsilon^{N-2\,s}=2.
\]
As for $n$, we have to distinguish two cases:
\begin{itemize}
\item if $N=2$, we take\footnote{For every $\alpha\in\mathbb{N}$, we define
\[
\lfloor \alpha\rfloor=\max\{n\in\mathbb{Z}\, :\, n\le \alpha\}.
\]
By definition, we always have $\lfloor \alpha\rfloor+1\ge \alpha$.}		
\[
n=\left(\left\lfloor\frac{\left(\theta_{2,s}\,\right)^{1+\frac{1}{s}}}{(s-1)^2} \right\rfloor + 1\right)^\frac{1}{2s}. 
\]
By recalling the asymptotic behavior \eqref{eq: asintotica cost} of $\theta_{2,s}$ as $s\searrow 1$, we obtain in particular that
\[
\lim_{s\searrow 1} n=+\infty.
\]
In turn, we have 
\[
\lim_{s\searrow 1} \left(\frac{\varepsilon}{n}\right)^{2\,s-2}<+\infty\qquad \text{and}\qquad \lim_{s\searrow 1} \left(\frac{\varepsilon}{n}\right)^{2}=0.
\]
By using these facts, the asymptotics for $A_{s-1}$ given by \eqref{asym-Aesse} and Proposition \ref{prop:BBM}, we get
\[
\limsup_{s\searrow 1} \mathcal{R}_1(s,n,\varepsilon)\le C\qquad \text{and}\qquad \limsup_{s\searrow 1} \mathcal{R}_2(s,n,\varepsilon)\le C,
\]
for some $C=C(N,\|u\|_{C^1(\overline{Q_1})})>0$. We need to control the term $\mathcal{R}_3(s,n,\varepsilon)$: we have by construction
	\[
	\frac{1}{n^{2s}}\le \frac{(s-1)^2}{\left(\theta_{2,s}\,\right)^{1+\frac{1}{s}}},
	\]	
	so that
\[
\begin{split}
\mathcal{R}_3(s,n,\varepsilon)&\le (s-1)\,\Big(\theta_{2,s}\Big)^{-\frac{1}{s}} \,\left[1+\left(\theta_{2,s}\,\right)^{\frac{1}{s}}\,\left(\frac{A_{s-1}}{\Big(\lambda(Q_1)\Big)^s}\right)^{1-\frac{1}{s}}\,\mathrm{Cap}_s(\{0\}; Q_1)\right]\\
&=\left[(s-1)\,\Big(\theta_{2,s}\Big)^{-\frac{1}{s}}+\left(\frac{A_{s-1}}{\Big(\lambda(Q_1)\Big)^s}\right)^{1-\frac{1}{s}}\,(s-1)\,\mathrm{Cap}_s(\{0\}; Q_1)\right].
\end{split}
\]
We now recall that
\[
\theta_{2,s}\sim 1,\quad \text{as}\ s\searrow 1 \qquad \text{and}\qquad \lim_{s\searrow 1}(s-1)\,\mathrm{Cap}_s(\{0\}; Q_1)=0,
\]
where the second fact follows from Lemma \ref{lm:puntolimite}.
Together with the asymptotic behavior of $A_{s-1}$, this finally gives
\[
\limsup_{s\searrow 1} \mathcal{R}_3(s,n,\varepsilon)=0.
\]
By using these three limits and taking into account \eqref{ok_esempio}, from \eqref{avieri} we eventually get the desired conclusion \eqref{claim:opt} for $N=2$;
\vskip.2cm
\item if $N=3$, we observe that the term $1/(s-1)$ in $\mathcal{R}_3(s,n,\varepsilon)$ is no more singular in the limit as $s$ goes to $3/2$. We can thus take
\[
n=\left(\left\lfloor\Big(\theta_{3,s}\,\Big)^{1+\frac{3}{2s}} \right\rfloor + 1\right)^\frac{1}{2s},
\]
which still diverges to $+\infty$, as $s$ goes to $3/2$. Thus, we have again
\[
\limsup_{s\searrow 1} \mathcal{R}_1(s,n,\varepsilon)\le C\qquad \text{and}\qquad \limsup_{s\searrow 1} \mathcal{R}_2(s,n,\varepsilon)\le C.
\]
As for the third term, we have this time
\[
\begin{split}
\mathcal{R}_3(s,n,\varepsilon)&\le \frac{\Big(\theta_{3,s}\Big)^{-\frac{3}{2\,s}}}{s-1} \,\left[1+\left(\theta_{3,s}\,\right)^{\frac{3}{2\,s}}\,\left(\frac{A_{s-1}}{\Big(\lambda(Q_1)\Big)^s}\right)^{1-\frac{3}{2\,s}}\,\mathrm{Cap}_s(\{0\}; Q_1)\right]\\
&=\frac{1}{s-1}\,\left[\Big(\theta_{3,s}\Big)^{-\frac{3}{2\,s}}+\left(\frac{A_{s-1}}{\Big(\lambda(Q_1)\Big)^s}\right)^{1-\frac{3}{2\,s}}\,\mathrm{Cap}_s(\{0\}; Q_1)\right].
\end{split}
\]
By recalling that (we use again Lemma \ref{lm:puntolimite})
\[
\theta_{3,s}\sim \frac{1}{2\,s-3},\ \text{as}\ s\searrow \frac{3}{2} \qquad \text{and}\qquad \lim_{s\searrow \frac{3}{2}}\mathrm{Cap}_s(\{0\}; Q_1)=0,
\]
we can easily get
\[
\lim_{s\searrow \frac{3}{2}} \mathcal{R}_3(s,n,\varepsilon)=0.
\]
in the case $N=3$, as well. This is enough to conclude.
\end{itemize}
The proof is over, at last.
	\end{proof}
	\appendix 
	
\section{A real interpolation estimate}
	
In this section, we prove a technical estimate which was needed in order to obtain the fractional Poincar\'e-Wirtinger inequality of Proposition \ref{lm:pw_frac}.	This was an essential ingredient of the cornerstone estimate of Theorem \ref{thm:mp_frac}.
\par
The result contains a $K-$functional and exploits typical real interpolation techniques. We refer the reader to the classical book \cite{BL} for more details on interpolation theory.
	\begin{proposition} 
	\label{prop:Kfunz}
		Let $0 < \sigma <1$ and $0 < d \le D < +\infty$. For every $u \in C^1(\overline{Q_D(x_0)})$, we have 
		\begin{equation} \label{claim:k-funz}
			\int_{0}^{\max\{d, D/2\}} \left(\frac{K(t, u)}{t^\sigma}\right)^2 \frac{dt}{t} \le A_N\,\iint_{Q_d(x_0) \times Q_D(x_0)} \frac{|u(x)-u(y)|^2}{|x-y|^{N+2\,\sigma}}\,dx\,dy,
		\end{equation} 
		where we set 
		\begin{equation*}
			\begin{split}
				K(t, u) := \min_{v \in W^{1,2}(Q_d(x_0))} \Big[\|u-v\|_{L^2(Q_d(x_0))} + \|\nabla v\|_{L^2(Q_d(x_0))}\Big].
			\end{split}
		\end{equation*} 
	\end{proposition}
	\begin{proof}
		The proof proceeds along the lines of \cite[Proposition 3.3]{BiaBra24}, up to some minor modifications. Without loss of generality, we can assume that $x_0 = 0$. It will be useful to distinguish two cases: either $D \ge 2\,d$ or $D < 2\,d$.
		\vskip.2cm \noindent
		If $D \ge 2\,d$, we can infer that $Q_{D/2}(x) \subseteq Q_D$, for $x \in Q_d$. This yields 
		\begin{equation} \label{ineq:estendi_1}
			\begin{split}
				\iint_{Q_{d} \times Q_{D}} \frac{|u(x)-u(y)|^2}{|x-y|^{N+2\,\sigma}}\,dx\,dy \ge \int_{Q_d} \left(\int_{Q_{D/2}(x)} \frac{|u(x)-u(y)|^2}{|x-y|^{N+2\,\sigma}}dy\right)dx = \int_{Q_{D/2}} \frac{U(h)^2}{|h|^{N+2\,\sigma}} dh, 
			\end{split}
		\end{equation} 
		where we set 
		\begin{equation*}
			U(h) = \left(\int_{Q_d} |u(x+h)-u(x)|^2 dx\right)^\frac{1}{2}, \qquad \text{for}\ h \in Q_{D/2}. 
		\end{equation*}
		We also define 
		\begin{equation*}
			\overline{U}(\varrho) := \frac{1}{\mathcal{H}^{N-1}(\partial Q_\varrho)}\,\int_{\partial Q_\varrho} U\,d\mathcal{H}^{N-1},\qquad \text{for}\ 0 < \varrho \le D/2. 
		\end{equation*}
		Then, by Jensen's inequality, the Coarea Formula\footnote{We use the Coarea Formula in the following form
		\[
		\int_{Q_{D/2}} F\,|\nabla f|\,dx=\int_0^{D/2} \left(\int_{\{y\, :\, f(y)=\varrho\}} F\,\mathcal{H}^{N-1}(x)\right)\,d\varrho,
		\]
		with $f$ given by the Lipschitz function $f(x)=\|x\|_{\ell^\infty}=\max_{i\in\{1,\dots,N\}} |x_i|$. We observe that $|\nabla f(x)|=1$ for almost every $x\in\mathbb{R}^N$ and that $\{y\,:\, f(y)=\varrho\}=\partial Q_\varrho$.} and \eqref{ineq:estendi_1}, we get 
		\begin{equation} \label{utile1}
			\begin{split}
				\int_{0}^{D/2} \overline{U}(\varrho)^2\,\frac{d\varrho}{\varrho^{1+2\,\sigma}} &\le \frac{1}{2^N N} \int_{0}^{D/2} \left(\int_{\partial Q_\varrho} U^2\,d\mathcal{H}^{N-1}\right)\frac{d\varrho}{\varrho^{N+2\,\sigma}}\\
				&= \frac{1}{2^N N} \int_{Q_{D/2}} \frac{U(h)^2}{\|h\|_{\ell^\infty}^{N+2\,\sigma}}\,dh\\
				& \le \frac{1}{2^N} \int_{Q_{D/2}} \frac{U(h)^2}{|h|^{N+2\,\sigma}}\,dh\le \frac{1}{2^N}\iint_{Q_d \times Q_{D}} \frac{|u(x)-u(y)|^2}{|x-y|^{N+2\,\sigma}}\,dx\,dy.
			\end{split}
		\end{equation} 
		We consider the function
		\begin{equation*}
			\psi(x) := \frac{N+1}{2^N}\,\left(1-\|x\|_{\ell^\infty}\right)_{+}, \quad \text{for}\ x \in \mathbb{R}^N,
		\end{equation*}
		and observe that this is Lipschitz function vanishing outside $Q_1$, with ${\rm supp}(\psi)  = \overline{Q_1}$ and $\|\psi\|_{L^1(\mathbb{R}^N)} = 1$. Thus, the rescaled function given by 
		\begin{equation*}
			\psi_t(x) := \frac{1}{t^N} \psi\left(\frac{x}{t}\right), \qquad \text{for}\ t >0,
		\end{equation*}
		has support coinciding with $\overline{Q_t}$ and $\|\psi_t\|_{L^1(\mathbb{R}^N)} =1$. This entails that $\psi_t * u \in W^{1,2}(Q_d)$, for every $0 < t < D/2$, and so 
		\begin{equation} \label{ineq:test-k-funz}
			\begin{split}
				K(t,u) \le \|u - \psi_t*u\|_{L^2(Q_d)} + t\,\|\nabla (\psi_t* u)\|_{L^2(Q_d)}, \qquad \text{for}\ 0 < t < D/2.
			\end{split}
		\end{equation} 
		To estimate the first quantity, by Minkowski's integral inequality and by Fubini's Theorem, we get 
		\begin{equation} \label{ineq:test-k-1}
			\begin{split}
				\|u-\psi_t*u\|_{L^2(Q_d)} &= \left(\int_{Q_d}\left|\int_{Q_t} (u(x) - u(x-y))\,\psi_t(y)\,dy\right|^2 dx\right)^\frac{1}{2} \\
				&\le \int_{Q_t} \left(\int_{Q_d} \left|u(x)-u(x-y)\right|^2 dx\right)^\frac{1}{2} \psi_t(y)\,dy\\
				&= \int_{Q_t} U(-y)\,\psi_t(y)\,dy = \int_{Q_t} U(y)\,\psi_t(y)\,dy,
			\end{split}
		\end{equation}
	thanks to the fact that $\psi_t(-y))=\psi_t(y)$. For the last term, we use Coarea Formula as before, so that
	\[
	\begin{split}
	\int_{Q_t} U(y)\,\psi_t(y)\,dy&=\int_0^t \left(\int_{\partial Q_\varrho} U(y)\,\psi_t(y)\,d\mathcal{H}^{N-1}\right)\,d\varrho\\
	&=\frac{N+1}{2^N}\,\frac{1}{t^N}\,\int_0^t\left(\int_{\partial Q_\varrho} U(y)\,\left(1-\frac{\varrho}{t}\right)\,d\mathcal{H}^{N-1}\right)\,d\varrho\\
	&\le \frac{N+1}{2^N}\,\frac{1}{t^N}\,\int_0^t\,\mathcal{H}^{N-1}(\partial Q_\varrho)\left(\frac{1}{\mathcal{H}^{N-1}(\partial Q_\varrho)}\,\int_{\partial Q_\varrho} U(y)\,d\mathcal{H}^{N-1}\right)\,d\varrho\\
	&\le \frac{N\,(N+1)}{t} \int_0^t \overline{U}(\varrho)\,d\varrho,
	\end{split}
	\]	
		where in the last inequality we used that $\mathcal{H}^{N-1}(\partial Q_\varrho) \le \mathcal{H}^{N-1}(\partial Q_t) = 2^N N\,t^{N-1}$, for $0 < \varrho \le t$. From \eqref{ineq:test-k-1}, we thus get
\[
	\|u-\psi_t*u\|_{L^2(Q_d)} \le \frac{N\,(N+1)}{t} \int_0^t \overline{U}(\varrho)\,d\varrho.
\]
To estimate the second term in \eqref{ineq:test-k-funz}, it is useful to notice that $\int_{Q_t} \nabla \psi_t(y)\,dy = 0$, 
	thanks to the Divergence Theorem. This permits to rewrite
		\begin{equation*}
			\nabla(\psi_t*u)(x) = (\nabla \psi_t)*u(x) = \int_{Q_t} \nabla \psi_t(y)\,(u(x-y) - u(x))\,dy, 
		\end{equation*}
		for every $x \in Q_d$ and every $0 < t < D/2$. By arguing as for \eqref{ineq:test-k-1}, this yields
		\begin{equation*}
			\begin{split}
				\|\nabla (\psi_t*u)\|_{L^2(Q_d)} &= \left(\int_{Q_d} \left|\int_{Q_t} \nabla \psi_t(y)\,(u(x-y) - u(x))\,dy\right|^2 dx\right)^\frac{1}{2} \\
				&\le \int_{Q_t} \left(\int_{Q_d} |u(x-y) - u(x)|^2 dx\right)^\frac{1}{2} |\nabla \psi_t(y)|\,dy\\
				&\le \frac{N+1}{2^N\, t^{N+1}} \int_{Q_t} U(-y)\,dy \le \frac{N\,(N+1)}{t^2}\,\int_0^t \overline{U}(\varrho)\,d\varrho. 
			\end{split}
		\end{equation*}
		By spending the previous estimates in \eqref{ineq:test-k-funz}, we get 
		\begin{equation*}
			K(t,u) \le \frac{2\,N\,(N+1)}{t} \int_{0}^{t} \overline{U}(\varrho)\,d\varrho, \quad \mbox{ for } 0 < t < D/2.
		\end{equation*}
		By Hardy's inequality, see for example \cite[Teorema 1]{Ta_Hardy}, and by \eqref{utile1}, this entails that 
		\begin{equation*}
			\begin{split}
				\int_0^{D/2} \left(\frac{K(t, u)}{t^\sigma}\right)^2 \frac{dt}{t} &\le (2\,N\,(N+1))^2 \int_0^{D/2} \left(\frac{1}{t} \int_0^t \overline{U}(\varrho)\,d\varrho\right)^2 \frac{dt}{t^{1+2\,\sigma}} \\
				&\le \left(\frac{2\,N\,(N+1)}{\sigma+1}\right)^2\int_0^{D/2} \overline{U}(t)^2 \frac{dt}{t^{1+2\,\sigma}}\\
				&\le  A_N\,\iint_{Q_d \times Q_D} \frac{|u(x)-u(y)|^2}{|x-y|^{N+2\,\sigma}}\,dx\,dy,
			\end{split}
		\end{equation*}
		where we set $A_N = N^2\,(N+1)^2/2^{N-2}$. This concludes the proof, in the case $D>2\,d$.
		\vskip.2cm \noindent
		If $D< 2\,d$, we take into account the extension operator $\mathcal{E}_{Q_d}$, provided by \cite[Corollary 3.2]{BiaBra24}, and set $\widetilde{u} := \mathcal{E}_{Q_d}(u)$, which is defined on the whole $\mathbb{R}^N$. In particular, $\widetilde{u} = u$ on $Q_d$ and we have
\[
\frac{1}{C_N}\,\iint_{Q_{2\,d} \times Q_{2\,d}} \frac{|\widetilde{u}(x) - \widetilde{u}(y)|^2}{|x-y|^{N+2\,\sigma}} dx\,dy\le \iint_{Q_{d} \times Q_d} \frac{|u(x)-u(y)|^2}{|x-y|^{N+2\,\sigma}} dx\,dy 
	\]	
By taking this estimate into account, we can thus minorize the rightmost term in \eqref{claim:k-funz}  as follows 
		\begin{equation*}
			\begin{split}
				\iint_{Q_{d} \times Q_{D}} \frac{|u(x)-u(y)|^2}{|x-y|^{N+2\,\sigma}} dx\,dy 
				&\ge  \frac{1}{C_N}\,\int_{Q_{d}} \left(\int_{Q_{d}(x)} \frac{|\widetilde{u}(x)-\widetilde{u}(y)|^2}{|x-y|^{N+2\,\sigma}} dy\right) dx\\
				&= \frac{1}{C_N}\,\iint_{Q_{d}\times Q_{d}} \frac{|\widetilde{u}(x+h)-\widetilde{u}(x)|^2}{|h|^{N+2\,\sigma}} dx\,dh.
			\end{split}
		\end{equation*} 
		where in the first inequality we used that $Q_{d}(x) \subseteq Q_{2\,d}$, for  $x \in Q_{d}$.
		As in the previous case, we set 
		\begin{equation*}
			V(h) := \left(\int_{Q_d} |\widetilde{u}(x+h) - \widetilde{u}(x)|^2\,dx\right)^\frac{1}{2}, \qquad \text{for}\ h \in Q_d,
		\end{equation*}
		thus the previous inequality can be rewritten as
		\begin{equation*} 
			\begin{split}
				\iint_{Q_d \times Q_D} \frac{|u(x)-u(y)|^2}{|x-y|^{N+2\,\sigma}}\,dx\,dy \ge \frac{1}{C_N}\,\int_{Q_d} \frac{V(h)^2}{|h|^{N+2\,\sigma}}\,dh. 
			\end{split}
		\end{equation*}
		Similarly as above, we also set 
		\begin{equation*}
			\overline{V}(\varrho) :=\frac{1}{\mathcal{H}^{N-1}(\partial Q_\varrho)}\, \int_{\partial Q_\varrho} V\,d\mathcal{H}^{N-1}, \qquad \text{for}\ 0 < \varrho \le d. 
		\end{equation*}
		By spending the last inequality and by arguing as in \eqref{utile1}, we get
		\begin{equation} \label{utile1_new}
			\begin{split}
				\int_{0}^{d} \overline{V}(\varrho)^2\,\frac{d\varrho}{\varrho^{1+2\,\sigma}} \le \frac{C_N}{2^N}\,\iint_{Q_d \times Q_D} \frac{|u(x)-u(y)|^2}{|x-y|^{N+2\,\sigma}}\,dx\,dy.
			\end{split}
		\end{equation} 
		Since $\psi_t * \widetilde{u} \in W^{1,2}(Q_d)$, by definition of $K-$functional we have
		\begin{equation} \label{ineq:test-k-funz_new}
			\begin{split}
				K(t,u) \le \|u - \psi_t*\widetilde{u}\|_{L^2(Q_d)} + t\,\|\nabla (\psi_t* \widetilde{u})\|_{L^2(Q_d)}, \quad \mbox{ for } 0<t <d.
			\end{split}
		\end{equation} 
		Then, by using that $u = \widetilde{u}$ on $Q_d$ and reasoning as in \eqref{ineq:test-k-1}, we infer that 
	\[
			\begin{split}
				\|u-\psi_t*\widetilde{u}\|_{L^2(Q_d)} &= \left(\int_{Q_d}\left|\int_{Q_t} (\widetilde{u}(x) - \widetilde{u}(x-y))\,\psi_t(y)\,dy\right|^2 dx\right)^\frac{1}{2} \le \frac{N\,(N+1)}{t} \int_0^t \overline{V}(\varrho)\,d\varrho,
			\end{split}
	\]
		and that  
		\begin{equation*}
			\begin{split}
				\|\nabla (\psi_t*\widetilde{u})\|_{L^2(Q_d)} &= \left(\int_{Q_d} \left|\int_{Q_t} \nabla \psi_t(y)\,(\widetilde{u}(x-y) - \widetilde{u}(x))\,dy\right|^2 dx\right)^\frac{1}{2} \le \frac{N\,(N+1)}{t^2}\,\int_0^t \overline{V}(\varrho)\,d\varrho. 
			\end{split}
		\end{equation*}
		By spending the previous estimates in \eqref{ineq:test-k-funz_new}, we get 
		\begin{equation*}
			K(t,u) \le \frac{2\,N\,(N+1)}{t} \int_{0}^{t} \overline{V}(\varrho)\,d\varrho, \qquad \text{for}\ 0<t <d.
		\end{equation*}
	We can now conclude as before: by Hardy's inequality and by \eqref{utile1_new} 
	\[
			\begin{split}
				\int_0^{d} \left(\frac{K(t, u)}{t^\sigma}\right)^2 \frac{dt}{t} &\le (2\,N\,(N+1))^2 \int_0^d \left(\frac{1}{t} \int_0^t \overline{V}(\varrho)\,d\varrho\right)^2 \frac{dt}{t^{1+2\,\sigma}} \\
				&\le \left(\frac{2\,N\,(N+1)}{\sigma+1}\right)^2\int_0^d \overline{V}(t)^2 \frac{dt}{t^{1+2\,\sigma}}\le  A_N\,\iint_{Q_d \times Q_D} \frac{|u(x)-u(y)|^2}{|x-y|^{N+2\,\sigma}}\,dx\,dy,
			\end{split}
	\]
		where we set $A_N = C_N\,N^2\,(N+1)^2/2^{N-2}$. This yields the claimed estimate in this case, as well.
	\end{proof}
	As an easy consequence of the previous result, we get the following ``asymmetric'' estimate.
\begin{corollary}
	Let $0 < \sigma <1$, $d>0$ and $x_0\in\mathbb{R}^N$. For every $u \in C^\infty_0(\mathbb{R}^N)$, we have
	\[
				\int_{0}^{+\infty} \left(\frac{K(t, u)}{t^\sigma}\right)^2 \frac{dt}{t} \le A_N\,\iint_{Q_d(x_0) \times \mathbb{R}^N} \frac{|u(x)-u(y)|^2}{|x-y|^{N+2\,\sigma}}\,dx\,dy,
	\]
where $A_N$ is the same constant as in Proposition \ref{prop:Kfunz} and the $K-$functional is still defined by
\[
K(t, u)= \min_{v \in W^{1,2}(Q_d(x_0))} \Big[\|u-v\|_{L^2(Q_d(x_0))} + \|\nabla v\|_{L^2(Q_d(x_0))}\Big].
\]
\end{corollary}	
\begin{proof}
For every $D\ge d$ we clearly have
\[
\iint_{Q_d(x_0) \times Q_D(x_0)} \frac{|u(x)-u(y)|^2}{|x-y|^{N+2\,\sigma}}\,dx\,dy\le \iint_{Q_d(x_0) \times \mathbb{R}^N} \frac{|u(x)-u(y)|^2}{|x-y|^{N+2\,\sigma}}\,dx\,dy.
\]
By using this fact in \eqref{claim:k-funz}, we can conclude by letting $D$ diverge to $+\infty$ and appealing to the Monotone Convergence Theorem.
\end{proof}


\begin{thebibliography}{100}


\bibitem{AFN}	 L. Abatangelo, V. Felli, B. Noris,  On simple eigenvalues of the fractional Laplacian under removal of small fractional capacity sets, Commun. Contemp. Math., {\bf 22} (2020), 1950071, 32 pp.

	\bibitem{An} A. Ancona, On strong barriers and inequality of Hardy for domains in $\mathbb{R}^n$, J. London Math. Soc., {\bf 34} (1986), 274--290.

\bibitem{AnBa} C. A. Antonini, P. Baroni, Partial regularity for vectorial local-nonlocal problems, preprint (2026), available at {\tt https://cvgmt.sns.it/paper/7861/}	
		
	\bibitem{vanBu} M. van den Berg, D.  Bucur, On the torsion function for simply connected, open sets in $\mathbb{R}^2$, Potential Anal., {\bf 63} (2025), 413--429.
		
			\bibitem{BC} R. Ba\~{n}uelos, T. Carroll, Brownian motion and the fundamental frequency of a drum, Duke Math. J., {\bf 75} (1994), 575--602.
		
		\bibitem{BL}  J. Bergh, J. L\"ofstr\"om, {\it Interpolation spaces. An introduction}. Grundlehren der Mathematischen Wissenschaften, {\bf 223}. Springer-Verlag, Berlin-New York, 1976.
		
				\bibitem{BiaBra24} F. Bianchi, L. Brasco, An optimal lower bound in fractional spectral geometry for planar sets with topological constraints, J. London Math. Soc., {\bf 109} (2024), Paper No. e12814, 45 pp.
	
	\bibitem{BiaBra22} F. Bianchi, L. Brasco, The fractional Makai-Hayman inequality, Ann. Mat. Pura Appl. (4), {\bf 201} (2022), 2471--2504.
			
				\bibitem{BiaBraZag} F. Bianchi, L. Brasco, A.~C. Zagati, On the sharp Hardy inequality in Sobolev-Slobodecki\u{\i} spaces, Math. Ann. {\bf 390} (2024), 493--555
		
		\bibitem{BozBra2} F. Bozzola, L. Brasco, Variations on the capacitary inradius, Discrete Contin. Dyn. Syst. Ser. S, {\bf 20} (2026), 206--236.
		
		\bibitem{BozBra} F. Bozzola, L. Brasco, The role of topology and capacity in some bounds for principal frequencies, J. Geom. Anal., {\bf 34} (2024), Paper No. 299, 46 pp.
				
				\bibitem{BozTal} F. Bozzola, M. Talluri, Maz'ya-type bounds for sharp constants in fractional Poincar\'e-Sobolev inequalities, Calc. Var. Partial Differential Equations, {\bf 65} (2026), Paper No. 87 
				
				\bibitem{BBM} J. Bourgain, H. Brezis, P. Mironescu, Another look at Sobolev spaces, in {\it Optimal control and partial differential equations}, 439--455, IOS, Amsterdam
				
				\bibitem{Brasco_book} L. Brasco, {\it Handbook of Calculus of Variations for Absolute Beginners}. Unitext, {\bf 163}, La Mat. per il 3+2,
Springer, Cham, 2025.
				
				\bibitem{BPS} L. Brasco, E. Parini and M. Squassina, Stability of variational eigenvalues for the fractional $p-$Laplacian, Discrete Contin. Dyn. Syst. {\bf 36} (2016), 1813--1845.
				
	\bibitem{BuVa} C. Bucur, E. Valdinoci, {\it Nonlocal diffusion and applications}. Lect. Notes Unione Mat. Ital., {\bf 20}, Springer, [Cham], Unione Matematica Italiana, Bologna, 2016. 
	
				
			\bibitem{Cr} C. B. Croke, The first eigenvalue of the Laplacian for plane domains, Proc. Amer. Math. Soc., {\bf 81} (1981), 304--305.
		
				\bibitem{EE} D. E. Edmunds, W. D. Evans, {\it Fractional Sobolev spaces and inequalities}, Cambridge Tracts in Mathematics, {\bf 230}. Cambridge University Press, Cambridge, 2023. 
			
				\bibitem{GiMos} N. Gigli, S. J. N. Mosconi, The abstract Lewy-Stampacchia inequality and applications, J. Math. Pures Appl. (9), {\bf 104} (2015), 258--275.
			
			\bibitem{Ha} W. K. Hayman, Some bounds for principal frequency, Applicable Anal., {\bf 7} (1977/78), 247--254.	
				
				\bibitem{Hen} A. Henrot, {\it Extremum problems for eigenvalues of elliptic operators}. Frontiers in Mathematics. Birkhauser Verlag, Basel, 2006.
				
				\bibitem{He} J. Hersch, Sur la fr\'equence fondamentale d'une membrane vibrante: \'evaluations par d\'efaut et principe de maximum, Z. Angew. Math. Phys., {\bf 11} (1960), 387--413. 
				
				\bibitem{Hor} L. H\"ormander, {\it The analysis of linear partial differential operators I. Distribution theory and Fourier Analysis.} Reprint of the second (1990) edition. Classics in Mathematics. Springer-Verlag, Berlin, 2003.
			
			\bibitem{Ka} R. Kajikiya, A priori estimate for the first eigenvalue of the $p-$Laplacian, Differential Integral Equations, {\bf 28} (2015), 1011--1028.
			
				\bibitem{Leo_frac} G. Leoni, {\it A first course in fractional Sobolev spaces}, Graduate Studies in Mathematics, 229, Amer. Math. Soc., Providence, RI, (2023).
	
	\bibitem{Ma} E. Makai, A lower estimation of the principal frequencies of simply connected membranes, Acta Math. Acad. Sci. Hungar., {\bf 16} (1965), 319--323.
			
				\bibitem{Maz} V. Maz'ya, {\it Sobolev spaces with applications to elliptic partial differential equations}. Second, revised and augmented edition. Springer, Heidelberg, 2011.
				
				\bibitem{MS} V. Maz'ya, T.~O. Shaposhnikova, On the Bourgain, Brezis, and Mironescu theorem concerning limiting embeddings of fractional Sobolev spaces, J. Funct. Anal., {\bf 195} (2002), 230--238.
				
				\bibitem{Os} R. Osserman, A note on Hayman's theorem on the bass note of a drum, Comment. Math. Helvetici, {\bf 52} (1977), 545--555.
				
	\bibitem{Sa} A. Salda\~na, On fractional higher-order Dirichlet boundary value problems: between the Laplacian and the bilaplacian. {\it Mexican mathematicians in the world -- trends and recent contributions}, 255--277. Contemp. Math., {\bf 775}.
American Mathematical Society, [Providence], RI, 2021.
	
					\bibitem{Ta_Hardy} G. Talenti, Sopra una diseguaglianza integrale, Ann. Scuola Norm. Sup. Pisa Cl. Sci. (3), {\bf 21} (1967), 167--188.
			
			\bibitem{Tay} M. E. Taylor, Estimate on the fundamental frequency of a drum, Duke Math. J., {\bf 46} (1979), 447--453.
				
				\bibitem{Tri} H. Triebel, {\it Theory of function spaces}. Reprint of 1983 edition, Modern Birkh\"auser Classics, Birkh\"auser/Springer Basel AG, Basel, 2010.
				
		\end{thebibliography}
\end{document}